%% file: DDO.tex
\documentclass[11pt]{amsart}
\usepackage{amsmath, amscd, amsthm} 
\usepackage{amssymb, amsfonts} 
\usepackage{enumerate}
\usepackage{bbm}
\usepackage{color}
\usepackage[all]{xy} 
\usepackage[margin=0.9in]{geometry}
\usepackage{palatino}
\usepackage{mathrsfs}
\usepackage{booktabs}
\usepackage{aliascnt}
\usepackage{mathrsfs}
\usepackage[svgnames]{xcolor} 
\usepackage{hyperref}
\usepackage{mathrsfs}
\usepackage{tikz}
\usepackage{cleveref}
\usepackage[english]{babel}
\usepackage{url}
\usepackage{enumitem}

\usepackage{dynkin-diagrams}

\RequirePackage[color       = blue!20,
                bordercolor = black,
                textsize    = tiny,
                figwidth    = 0.99\linewidth]{todonotes}
                
\usepackage{etoolbox}
\makeatletter
\patchcmd{\l@chapter}{1.0em}{0.8em}{}{}
\makeatother


\title{Punctured tubular neighborhoods and stable homotopy at infinity}

\author{Adrien Dubouloz}
\address{IMB UMR5584, CNRS, Universit{\'e} de Bourgogne Franche-Comt{\'e}, Dijon, France}
\email{Adrien.Dubouloz@u-bourgogne.fr}
\author{Fr{\'e}d{\'e}ric D{\'e}glise}
\address{ENS de Lyon, UMPA, UMR 5669, 46 all{\'e}e d'Italie, 69364 Lyon Cedex 07, France}
\email{frederic.deglise@ens-lyon.fr}
\author{Paul Arne {\O}stv{\ae}r}
\address{Department of Mathematics Federigo Enriques, University of Milan, Italy \&
Department of Mathematics, University of Oslo, Norway}
\email{paul.oestvaer@unimi.it, paularne@math.uio.no}

\date{\today}
\newtheorem{thm}{Theorem}[subsection]
\newtheorem{prop}[thm]{Proposition}
\newtheorem{lm}[thm]{Lemma}
\newtheorem{cor}[thm]{Corollary}

\theoremstyle{remark}
\newtheorem{rem}[thm]{Remark}
\newtheorem{notation}[thm]{Notations}
\newtheorem{ex}[thm]{Example}

\theoremstyle{definition}
\newtheorem{df}[thm]{Definition}
\newtheorem{num}[thm]{}

\numberwithin{equation}{thm}

\newtheorem{thm*}{Theorem}

\input{command}
\begin{document}

\begin{abstract}
We initiate a study of punctured tubular neighborhoods and homotopy theory at infinity in motivic settings.
We use the six functor formalism to give an intrinsic definition of the stable motivic homotopy type at infinity 
of an algebraic variety.
Our main computational tools include cdh-descent for normal crossing divisors, 
Euler classes, Gysin maps, and homotopy purity.
Under $\ell$-adic realization, 
the motive at infinity recovers a formula for vanishing cycles due to Rapoport-Zink;
similar results hold for Steenbrink's limiting Hodge structures and Wildeshaus' boundary motives.
Under the topological Betti realization, 
the stable motivic homotopy type at infinity of an algebraic variety recovers the singular complex at infinity 
of the corresponding topological space. We coin the notion of homotopically smooth morphisms with respect to a motivic $\infty$-category and 
use it to show a generalization to virtual vector bundles of Morel-Voevodsky's purity theorem, 
which yields an escalated form of Atiyah duality with compact support.
Further, 
we study a quadratic refinement of intersection degrees, 
taking values in motivic cohomotopy groups.
For relative surfaces, 
we show the stable motivic homotopy type at infinity witnesses a quadratic version of Mumford's plumbing construction 
for smooth complex algebraic surfaces.
Our construction and computation of stable motivic links of Du Val singularities on normal surfaces are expressed entirely 
in terms of Dynkin diagrams.
In characteristic $p>0$,
this improves on Artin's analysis of Du Val singularities through \'etale local fundamental groups.
The main results in the paper are also valid for $\ell$-adic sheaves, 
mixed Hodge modules, and, more generally, motivic $\infty$-categories.
\end{abstract}

\maketitle

\setcounter{tocdepth}{2}
\tableofcontents

\section{Introduction}
\label{section:introduction}

\input{intro}

\subsubsection{Acknowledgements}
The authors are grateful to Aravind Asok, Robin Carlier, Jean Fasel, Fangzhou Jin, Marc Levine, and Kirsten Wickelgren for their collaboration, discussions, and encouragement on some of the topics in this paper.
Our referee provided a valuable report that clarified some constructions and results in this paper.
We gratefully acknowledge the support of the 
Centre for Advanced Study at the Norwegian Academy of Science and Letters in Oslo,
Norway, which funded and hosted our research project ``Motivic Geometry" during the 2020/21 academic year,
and we extend our thanks to the French ``Investissements d'Avenir" project ISITE-BFC (ANR-15-IDEX-0008), 
the French ANR project ``HQDiag" (ANR-21-CE40-0015),
and the RCN Frontier Research Group Project no.~250399 ``Motivic Hopf Equations" 
and no.~312472 ``Equations in Motivic Homotopy."
{\O}stv{\ae}r acknowledges the generous support from the Alexander von Humboldt Foundation and 
The Radboud Excellence Initiative.

\subsection{The motivic formalism}
\label{sec:notations}

Throughout, 
all schemes are quasi-coherent and quasi-compact (=qcqs), 
and all separated and smooth maps are assumed to be of finite type.
The natural framework for this paper is Morel-Voevodsky's stable homotopy category $\SH(S)$ 
of the base scheme $S$.
Owing to the works \cite{Ayoub,Ayoub2}, \cite{CD3}, for varying $S$,  
these categories satisfy \emph{Grothendieck's six functors formalism}, 
which we will use extensively.
The noetherian hypothesis was eliminated in \cite[Appendix C]{hoyoisquadratic}.
Most of the results in this paper, 
however, 
can be stated in the general formalism of Grothendieck's six functors,
as axiomatized in \cite{CD3}. 
We will freely use the language, constructions, and notations from \emph{loc. cit.},
together with its natural $\infty$-categorical enhancement of \cite{KhanCoef, DrewGall} 
(which applies to premotivic model categories). 
Let us fix a \emph{motivic triangulated category} $\T$, 
see \cite[Definition 2.4.45]{CD3}, 
which also admits an $\infty$-categorical enhancement (e.g., it arises from a premotivic model category). 
We refer to $\T$ as a motivic $\infty$-category and note 
that $\T$ satisfies Grothendieck's six functors formalism, 
summarized, for example, in \cite[2.4.50]{CD3}.
The added generality of \cite{KhanCoef} verifies that the pair of adjoint functors
$(f^*,f_*)$, $(p_!,p^!)$ for $p$ separated, 
and $(\otimes,\uHom)$ are in fact adjunctions of $\infty$-categories.
The above applies to the following examples.

\begin{itemize}
\item $\SH$ -- the stable motivic homotopy category, see e.g., \cite{Ayoub, KhanCoef}.
\vspace{0.05in}
\item $\DM_\QQ$ -- rational mixed motives, see \cite[Part IV]{CD3}.
\vspace{0.05in}
\item $\DM$ -- motives defined as modules over Spitzweck's 
motivic cohomology ring spectrum relative to $\ZZ$,  
see \cite{SpiMod}.\footnote{This viewpoint was advocated in \cite{orpaomodulescras,orpaomodules}.
If one restricts to schemes over a prime field $k$ and inverts the characteristic exponent of $k$, 
one can employ $\cdh$-motives as defined in \cite{CD4} (using $\cdh$-sheaves with transfers).}
\vspace{0.05in}
\item $\tDM$ -- Milnor-Witt motives defined as modules over Milnor-Witt motivic cohomology, 
 if one restricts to base schemes defined over some field $k$ of characteristic not $2$;
 see \cite{BCDFO}, \cite{2020arXiv200602086B}, \cite{zbMATH07224517}.
\vspace{0.05in}
\item $\DM_\et=\mathrm{DA}_\et$ -- \'etale mixed motives, see \cite{AyoubEt,CD5}.
\vspace{0.05in}
\item $D(-_\et,\ZZ_\ell)$ -- $\ell$-adic \'etale sheaves on $\ZZ[1/\ell]$-schemes,
$\ell$ a prime number, 
see \cite{BBD}, \cite[7.2.18]{CD5}, 
and on excellent schemes, 
also its subcategory $D^b_c(-_\et,\ZZ_\ell)$ of bounded complexes with constructible cohomology.
\vspace{0.05in}
\item $\DB$ -- analytical sheaves on $k$-schemes for a complex embedding $\sigma:k \rightarrow \CC$,
$\DB(X)$ is the derived category of sheaves on the analytical site $X^\sigma(\CC)$.
This is classical; see also \cite{Ayoub3}. 
More generally,
given any mixed Weil theory $E$ over a base field $k$,
by restricting to $k$-schemes, 
one has the category $D_E$ of modules over the ring spectrum associated with $E$.
See \cite[\textsection 17.2]{CD3} for details.
\vspace{0.05in}
\item $\DHdg$ -- the category of motivic Hodge modules,
which corresponds to complexes of Saito's mixed Hodge modules of geometric origin
(obtained by the realization of mixed motives), see \cite{Drew}.
\end{itemize}

These examples are naturally related via premotivic adjunctions subject to our conventions above: 
\begin{equation}
\label{eq:motivic_cat}
\begin{split}
\xymatrix@R=10pt{
&&&& \DB \\
\SH\ar^{\tilde M}[r] & \tDM\ar^\pi[r] & \DM\ar^{a^\et}[r] & 
\DM_\et\ar^{\rho_\ell}[rr]\ar^{\rho_B}[ru]\ar_{\rho_{\mathrm{Hdg}}}[rd] && 
D(-_\et,\ZZ_\ell) \\
&&&& \DHdg
}
\end{split}
\end{equation}

\begin{itemize}
\item By our definitions of $\tDM$ and $\DM$,
the first two functors are induced by taking free modules. 
See \cite[\textsection 7.2]{CD3}, \cite{orpaomodules} for accounts using model categories.\footnote{The construction 
can be carried out more easily using (monoidal) $\infty$-categories developed in \cite{LurieHA}.}
\item The functor $a^\et$ changes the topology,
see \cite{2017arXiv171106258E}, 
taking into account the Dold-Kan correspondence and the $E_\infty$-ring spectra representing motivic cohomology and \'etale motivic cohomology.
\item The functor $\rho_B$ is defined in \cite{Ayoub3} (see \cite{CD3} for mixed Weil theories).
\item The functors $\rho_\ell$ and $\rho_{\mathrm{Hdg}}$ are defined in \cite{CD5} and \cite{Drew}, 
respectively.
\end{itemize}
Formally, 
being part of a premotivic adjunction,
each of the functors in \eqref{eq:motivic_cat} admits a natural right adjoint.
Thus, by construction, they commute with $f^*$, $p_!$, $\otimes$.
Moreover, 
when restricting to (quasi-) excellent schemes, they also commute with the other three operations in 
Grothendieck’s six functors formalism,
see the indicated references.
With rational coefficients, 
both $\tilde M$ and $a^\et$ are equivalences (see \cite{DFJK} and \cite{CD3}, respectively).
Furthermore, 
$\SH_\QQ \rightarrow \DM_\QQ$ is split with complementary factor \emph{Morel's minus part} of $\SH$ by \cite[16.2]{CD3}.
The reader should feel free to keep in mind a general $\T$, 
or specialize to $\SH$ and one of the realization functors in \eqref{eq:motivic_cat}.

\subsection{Conventions on divisors, vector bundles and virtual vector bundles}
\label{sub:conventions}

\indent We adopt the following standard convention concerning normal crossing and smooth normal crossing divisors: A \emph{smooth normal crossings} divisor on a locally noetherian scheme $X$ is an effective Cartier divisor $D\subset X$  such that for every point $x\in D$ the local ring $\cO_{X,x}$ is regular and there there exists a regular system of parameters $x_1,\ldots, x_d$ is the maximal ideal of $\cO_{X,x}$,  $1\leq r\leq d=\dim_x X$ such that $D$ is cut out by $x_1,\ldots, x_r$ in $\cO_{X,x}$. We say that a Cartier divisor $D$ on $X$ has \emph{normal crossings} if for every point $x\in D$ there exists an \'etale neighborhood $U\to X$ of $X$ such that $D\times_X U$ is a smooth normal crossings divisor on $U$. In  
\Cref{sec:crossing_sing} we will introduce variants of these notions for more general crossing singularities.  

We adopt the following convention for the correspondence between coherent locally free sheaves and vector bundles:
the vector bundle $E=\VV(\mathcal E)$ associated with a coherent locally free sheaf of $\mathcal{O}_X$-modules
$\mathcal{E}$ on a scheme $X$ is the relative spectrum of the symmetric algebra $\mathrm{Sym}(\mathcal{E})$.
For a vector bundle $p:V\rightarrow X$, we denote by $V^\times$ the complement of the zero section.

Concerning locally free sheaves and corresponding vector bundles associated with differential properties for 
morphisms of schemes, 
we adopt the following conventions:
\begin{itemize}
\item Given a smooth morphism $f:X\to S$, 
let $\Omega_f=\Omega_{X/S}$ be the sheaf of relative Kähler differentials of $f$ and call it the \emph{cotangent sheaf} of $f$.
 Its associated vector bundle, the relative spectrum of the symmetric algebra of $\Omega_f$,
 is the \emph{tangent bundle} $T_f=T_{X/S}$ of $f$.
\item Given a regular closed immersion $i:Z\to X$, with corresponding ideal sheaf
 $\mathcal{I}_Z \subset \mathcal{O}_X$, its \emph{conormal sheaf} is the
 $\mathcal O_Z$-module $\mathcal{C}_i=\mathcal C_{Z/X}=\mathcal{I}_Z/\mathcal{I}_Z^2$.
 Its associated vector bundle is the \emph{normal bundle} $N_{Z/X}$ of $Z$ in $X$.
\item We denote by $\mathcal E \otimes \mathcal F$ the tensor product of $\mathcal O_X$-modules
 and by $\mathcal E^\vee:=\mathcal{H}om_X(\mathcal E,\mathcal{O}_X)$ the dual.
\end{itemize}

Given any morphism of $f:X\to S$, we let $\cL_f=\cL_{X/S}$
 be its associated cotangent complex. In general, this is a complex of $\mathcal O_X$-modules.
 When $f$ is a local complete intersection morphism (lci for short), $\cL_f$ is a perfect
 complex. Moreover, when $f:X\to S$ is lci smoothable, say $f=p\circ i:X\to Y\to S$ where $i:X\to Y$ is a regular closed immersion
 and $p:Y\to S$ is smooth, we have $\cL_f=(\mathcal{C}_{i}\to i^*\Omega_p)$ where $i^*\Omega_p$ and $\mathcal C_{i}$ are in homological degree $0$ and $1$, respectively.

We will use Deligne's category  $\uK(X)$  of virtual coherent locally free sheaves of $\mathcal{O}_X$-modules 
on a scheme $X$  (see \cite{DelDet}).
 Given a locally free sheaf $\mathcal{E}$ on $X$, we denote by $\twist {\mathcal{E}}$ its image
 in $\uK(X)$. The correspondence between coherent locally free sheaves and vector bundles extends using the same convention as above to a correspondence between virtual locally free sheaves $\cV$ and their associated virtual vector bundles $v=``\mathbb{V}(\cV)"$. For a morphism of schemes $f:X\to Y$ and a (virtual) locally free sheaf $\cV$ on $Y$, we denote by $f^{-1}\cV$ the pullback of $\cV$ to $X$. 
 
 Recall also that $\uK(X)$ can be described using Thomason's K-theory space
 $K(X)$ (the infinite loop space associated with Thomason's K-theory spectrum,
  \cite[3.1]{Thomason}) as follows:
  we view the simplicial set $K(X)$ as an $\infty$-category and
 consider its associated $\infty$-groupoid $K(X)^\simeq$ (the sub-$\infty$-category
 generated by $1$-morphisms that are equivalences). Then $\uK(X)$ is the
 homotopy category associated with $K(X)^\simeq$ --- according to
 \cite[4.12, end of 4.4]{DelDet} and \cite[3.1.1]{Thomason}. This presentation
  has the advantage of giving an explicit functor
$$
\mathrm D_{perf}(X) \rightarrow \uK(X), \mathcal K \mapsto \twist{\mathcal K}
$$
by associating to a perfect complex $\mathcal K$ of $\mathcal O_X$-modules
 the corresponding $0$-simplex of $K(X)$,
 which follows from the very construction of Thomason using complicial biWaldhausen
 categories. 

Recall Deligne's graded determinant functor of Picard categories (\cite[Ex. 4.13]{DelDet})
$$
\uK(X) \xrightarrow{(\rk,\det)} \uZ_X \times \uPic(X), \cV \mapsto (\rk\cV,\det\cV)
$$
where $\uPic(X)$ denote Deligne's Picard category  of invertible sheaves on $X$, and for a virtual locally free sheaf $\cV$, 
$\det\cV$ is the  \emph{determinant} of $\cV$ 
and $\rk\cV$ is its virtual rank.
  
For an lci morphism $f:X \rightarrow S$, the \emph{virtual tangent bundle}  $\tau_f=\tau_{X/S}$ of $X/S$ is the virtual vector bundle on $X$ associated to $\twist{\cL_f}$. The \emph{canonical sheaf} $\omega_f=\omega_{X/S}$ of $X/S$ is the determinant $\det\twist{\cL_f}$ of $\twist{\cL_f}$.

\subsection{Limits and colimits in $\infty$-categories}

\begin{num}\label{num:ifty-diagrams}

This work will extensively use the concept of limits and, dually, colimits in an $\infty$-category. The primary references for this material are \cite[\textsection 4]{JoyalQC}, \cite[\textsection 1.2.13]{LurieHTT}, and \cite[\textsection 6.2]{CisinskiHC}.

Let us recall the basic ideas. Given a simplicial set \( K \) and an $\infty$-category \( \iC \) modeled by a quasi-category, a \( K \)-diagram in \( \iC \) is defined as a map of simplicial sets \( f:K \rightarrow \iC \). All our examples will derive from a category \( \mathcal{I} \), where \( K = N \mathcal{I} \) represents the nerve of \( \mathcal{I} \). It is useful to think of the functors\footnote{We note that we abusively denote these functors by \( f:\mathcal{I} \rightarrow \iC \).} \( N\mathcal{I} \rightarrow \iC \) as a \emph{homotopy coherent} \( \mathcal{I} \)-diagram (see \cite[1.2.6]{LurieHTT}).

For a general \( K \)-diagram \( f:K \rightarrow \iC \), we can associate the slice $\infty$-category \( \iC/f \) (and the coslice category \( f\backslash \iC \))\footnote{These categories are denoted by \( \iC_{-/f} \) and \( \iC_{f/-} \) respectively in \cite{LurieHTT}.} which intuitively consists of objects \( X \) in \( \iC \) such that for any \( 0 \)-simplex \( i \), there exist maps \( X \rightarrow f(i) \) and homotopy coherent diagrams for all \( 1 \)-simplexes \( p \in K \) 
$$
\xymatrix@C=30pt@R=-2pt{
& f(i) \ar^{p_*=f(p)}[dd] \\
X \ar[ru] \ar[rd] & \\
& f(j)
}
$$
and so on. 
Formally, the slice $\infty$-category \( \iC/f \) can be defined via the join construction \( \star \) of simplicial sets participating in the adjunction
$$
\Hom_f(X \star K,\iC) \simeq \Hom(X,\iC/f)
$$
See \cite[Prop. 3.2 and p. 214]{JoyalQC}, \cite[\textsection 1.2.9.2]{LurieHTT}, or \cite[3.4.14 and 6.2.1]{CisinskiHC}. The coslice is defined dually by the formula \( f\backslash\iC=(\iC^{op}/f^{op})^{op} \). Note that if \( \iC \) is a quasi-category, then so are \( \iC/f \) and \( f\backslash \iC \); see \cite[Cor. 3.9]{JoyalQC}. One of the advantages of quasi-categories is that the notion of an initial (or final) object is well-behaved. In particular, if such an object exists, the space of initial (or final) objects is contractible, making it unique in the $\infty$-categorical sense.
\end{num}

\begin{df}
The limit (resp. colimit) of a $K$-diagram $f:K \rightarrow \iC$ exists if $\iC/f$ (resp. $f\backslash \iC$) admits an initial (resp. final) object. We denote by 
\begin{align*}
\lim f &= \lim_{i \in K} f(i) \\
\text{resp. } \colim f &= \underset{i \in K}{\colim} f(i)
\end{align*}
any such initial (resp. final) object, referring to it as the limit (resp. colimit) of $f$ 
(usually, 
we treat it as an object of $\iC$, rather than as an object of $\iC/f$ (resp. $f\backslash \iC$)).
\end{df}

One of the most important properties for us is the following (see \cite[Prop. 6.2.9]{CisinskiHC}):
\begin{prop}
Assume that all $K$-limits (resp. $K$-colimits) exist in $\iC$. Then the $\infty$-functor $f \mapsto \lim f$ (resp. $f \mapsto \colim f$) is left (resp. right) adjoint to the constant diagram functor $\mathrm{ct}:\iC \rightarrow \mathrm{Fun}(K,\iC)$.
\end{prop}
Here, we denote by $\mathrm{Fun}(K,\iC)$ the quasi-category of $K$-diagrams, which is also referred to as $\uHom(K,\iC)$ in \emph{loc. cit.}; indeed, it is the internal Hom of the monoidal category of simplicial sets.

\begin{rem}
\begin{enumerate}
\item The preceding proposition applies in particular to presentable $\infty$-categories, as they are both complete and cocomplete. This means they admit $K$-limits and $K$-colimits for any simplicial set $K$ (see \cite[Def. 5.5.0.1, Cor. 5.5.2.4]{LurieHTT}).
\item The preceding result is significant because it immediately connects the notions of limits and colimits in an $\infty$-category $\iC$, associated with a model category $\mathcal M$, to the concepts of homotopy limits and colimits relative to $\mathcal M$ (see e.g., \cite[\textsection 2.3]{CisinskiHC} for the latter). Specifically, a Quillen adjunction of model categories induces an adjunction of the associated $\infty$-categories. See also \cite[\textsection 7.9]{CisinskiHC}.
\end{enumerate}
\end{rem}

We conclude this section with a useful lemma for computing limits and colimits in 
$\infty$-categories 
(which we have not been able to locate in the literature, but see also \cite{ZhenLinSE}).

\begin{lm}[Replacement lemma]\label{prop:replacement}
Let $f:K \rightarrow \iC$ be a $K$-diagram. Assume that for each $0$-simplex $i$ of $K$, we are given an isomorphism in $\iC$ denoted by 
$$
\psi_i:f(i) \rightarrow X_i
$$
Then there exists a $K$-diagram $f':K \rightarrow \iC$ and an isomorphism $\phi:f \rightarrow f'$ of $K$-diagrams such that for all $0$-simplices $i$ of $K$, we have $f'(i)=X_i$ and the map $\phi(i):f(i) \rightarrow f'(i)=X_i$ is equal to $\psi_i$.
\end{lm}
\begin{proof}
Let us consider \( K_0 \), the discrete simplicial set of \( 0 \)-simplices of \( K \). The canonical map \( s: K_0 \rightarrow K \) is a monomorphism. Therefore, the induced restriction map
\[
s^*: \mathrm{Fun}(K, \iC) \rightarrow \mathrm{Fun}(K_0, \iC)
\]
is an isofibration (as defined in \cite[Def. 2.3]{JoyalQC}, where it is referred to as quasi-fibrant, or in \cite[Def. 3.3.15]{CisinskiHC}).\footnote{This follows from the fact that the Joyal model structure on simplicial sets is cartesian. The map \( s \) is a cofibration for this model structure, and fibrations between fibrant objects (i.e., quasi-categories) are isofibrations. 
For a direct proof, see \cite[Tag 01F3]{Kerodon}.} 
Now, the collection of all the isomorphisms \( \psi_i \) defines an equivalence \( \psi: (f(i))_{i \in K_0} \rightarrow (X_i)_{i \in K_0} \) in \( \mathrm{Fun}(K_0, \iC) \). Since \( s^*(f) = (f(i))_{i \in K_0} \) by definition, and \( s^* \) is an isofibration, there exists an equivalence \( \phi: f \rightarrow f' \) for some \( K \)-diagram \( f' \) such that \( s^*(\phi) = \psi \).
\end{proof}

\section{Complements on six functors}
\label{sec:recall}
\input{recall}

\section{Canonical resolutions of crossing singularities}
\label{section:NCboundary}
\input{normal-rev}
\section{Punctured tubular neighborhoods and stable homotopy at infinity}
\label{section:shai}
\input{shtpinfty}

\section{Motivic plumbing}
\label{section:Mumford}
\input{Mumford-rev}

\section{Appendix: quadratic orientations and isomorphisms, cycles and degree}
\label{section:ovbaqi}
\input{appendix-sheafversion}

\bibliographystyle{abbrv}
\bibliography{htpinfty}

\end{document}

%% file: command.tex
\DeclareMathOperator{\Hom}{Hom}
\DeclareMathOperator{\End}{End}
\DeclareMathOperator{\Spec}{Spec}
\DeclareMathOperator{\uHom}{\underline{Hom}}
\DeclareMathOperator{\Map}{Map}
\DeclareMathOperator{\Th}{Th}
\DeclareMathOperator{\Tw}{Tw} 

\DeclareMathOperator{\ho}{h} 

\DeclareMathOperator{\Isom}{Isom} 

\DeclareMathOperator{\CH}{CH}
\DeclareMathOperator{\CHt}{\widetilde{CH}}
\DeclareMathOperator{\Zt}{\tilde Z}
\newcommand{\CHW}[1]{\widetilde{\mathrm{CH}}{}^{#1}}
\newcommand{\ZW}[1]{\tilde Z{}^{#1}}
\newcommand{\ZWhlg}[1]{\tilde Z{}_{#1}}

\DeclareMathOperator{\GW}{GW}

\DeclareMathOperator{\HM}{\underline H} 
\newcommand{\HMat}{\underline{\mathrm{H}}^{\mathrm{AT}}} 

\DeclareMathOperator{\Pic}{Pic}

\newcommand{\KMW}{\mathrm K^{MW}}

\newcommand{\Or}{\mathscr Or} 

\newcommand{\cech}{\check{\mathrm S}}
\newcommand{\ord}{\mathrm{ord}}
\newcommand{\cecho}{\check{\mathrm S}^{\ord}}
\newcommand{\cC}{\check{\mathrm C}} 

\newcommand{\derR}{\mathbf{R}}

\DeclareMathOperator{\colim}{colim}

\newcommand{\twist}[1]{\langle #1 \rangle} 

\DeclareMathOperator{\uK}{\underline K} 
\DeclareMathOperator{\Ksp}{\mathcal K} 
\DeclareMathOperator{\VV}{\mathbb V} 
\DeclareMathOperator{\uPic}{\underline{Pic}} 
\DeclareMathOperator{\uZ}{\underline{\ZZ}} 
\DeclareMathOperator{\uPicO}{\underline{Pic}^{or}} 

\DeclareMathOperator{\tdeg}{\widetilde{deg}}
\newcommand{\TN}{\operatorname{TN}^\times} 

\newcommand{\qiso}{\rightarrowtail} 

\DeclareMathOperator{\SH}{SH}
\DeclareMathOperator{\DM}{DM}
\DeclareMathOperator{\DMAT}{DM^{AT}}
\DeclareMathOperator{\DA}{D_{\AA^1}}
\newcommand{\DAx}[1]{\mathrm D_{\AA^1,#1}}
\DeclareMathOperator{\tDM}{\widetilde{DM}}
\DeclareMathOperator{\DB}{D_{\mathit B}^\sigma}
\DeclareMathOperator{\DHdg}{D^\mathit m_{Hdg}}
\DeclareMathOperator{\MMAT}{MM^{AT}}
\DeclareMathOperator{\MMN}{\mathcal M}

\DeclareMathOperator{\Der}{D}  
\DeclareMathOperator{\Sh}{Sh}  
\DeclareMathOperator{\Sp}{Sp}  

\DeclareMathOperator{\T}{\mathscr T} 
\DeclareMathOperator{\iC}{\mathscr C} 

\DeclareMathOperator{\Sch}{Sch}

\DeclareMathOperator{\Sm}{Sm}
\DeclareMathOperator{\hSm}{h-Sm}
\newcommand{\hSmp}{\operatorname{h-Sm}^{\mathrm{prop}}}

\DeclareMathOperator{\HZ}{H\ZZ} 

\DeclareMathOperator{\Id}{Id}

\newcommand{\NN} {\mathbb N}
\newcommand{\ZZ} {\mathbb Z}
\newcommand{\QQ} {\mathbb Q}
\newcommand{\RR} {\mathbb R}
\newcommand{\CC} {\mathbb C}
\renewcommand{\AA} {\mathbf A}
\newcommand{\PP} {\mathbf P}
\newcommand{\GG} {\mathbf{G}_m}

\newcommand{\GGx}[1] {\mathbf{G}_{m,#1}}
\newcommand{\uGGx}[1] {\underline{\mathbf{G}}_{m,#1}} 
\newcommand{\SL}{\mathbf{SL}}
\newcommand{\E}{\mathbb E}

\newcommand{\cO}{\mathcal O}
\newcommand{\cL}{\mathcal L}
\newcommand{\cM}{\mathcal M}
\newcommand{\cN}{\mathcal N}
\newcommand{\calC}{\mathcal C}

\newcommand{\cD}{\mathcal D} 

\newcommand{\cV}{\mathcal V} 

\newcommand{\uKMW}{\underline{\mathbf K}^{MW}}

\newcommand{\un}{\mathbf{1}} 
\newcommand{\thom}{\tau} 

\newcommand{\Dinj}{\Delta^{\mathrm{inj}}}    

\newcommand{\htp}{\Pi} 
\newcommand{\ehtp}{\bar \Pi} 
\newcommand{\eehtp}{\bar{\underline \Pi}} 
\newcommand{\uZZ} {\underline{\mathbb Z}} 

\newcommand{\cohtp}{\mathrm H} 

\newcommand{\zar}{{\textrm{Zar}}}
\newcommand{\nis}{{\textrm{Nis}}}
\newcommand{\et}{\textrm{\'et}}
\newcommand{\cdh}{{\textrm{cdh}}}

\newcommand{\h}  {{\textrm{h}}}

\DeclareMathOperator{\rk}{rk}
\DeclareMathOperator{\Tr}{Tr}

%% file: intro.tex
\subsection{Context and motivation}
Topology at infinity is essentially the study of topological properties that persistently occur in complements 
of compact sets.
A space is intuitively simply connected at infinity if one can collapse loops far away from any 
small subspace. 
Euclidean space ${\bf R}^{n}$, 
$n\geq 3$, 
is the unique open contractible $n$-manifold that is simply connected at infinity.
For example, 
the Whitehead manifold is not simply connected at infinity and therefore not homeomorphic to ${\bf R}^{3}$.
This article describes our first attempt at finding a unified theory of punctured tubular neighborhoods and 
homotopy at infinity for open manifolds and smooth varieties.
Our overriding goal is to develop a study of intrinsic motivic invariants which can distinguish between 
$\AA^{1}$-contractible varieties.
For background on motivic homotopy theory and $\AA^{1}$-contractible varieties, 
we refer to the survey \cite{asokostvar}.
The quest for finding invariants that can help classify smooth varieties over fields up to $\AA^{1}$-homotopy 
can be traced back to work by Asok-Morel \cite{zbMATH05918255}.
Their ideas on $\AA^{1}$-h-cobordisms and $\AA^{1}$-surgery theory, 
with applications towards vector bundles over projective spaces in Asok-Kebekus-Wendt \cite{zbMATH06761479}, 
have inspired our search for motivic invariants with a pronounced geometric topological flavor.
Another great source of inspiration is Zariski's cancelation problem \cite{zbMATH06623646},
which remains difficult because of the lack of computable invariants available to distinguish non-isomorphic 
$\AA^{1}$-contractible smooth affine varieties such as the Koras-Russell cubic threefold and $\AA^{3}$ 
(see \cite{zbMATH06999230}, \cite{zbMATH06741275}).
Our notion of motivic homotopy theory at infinity combines ideas appearing in the works of 
Spitzweck \cite{zbMATH05076906},  
Wildeshaus \cite{Wild1}, 
Levine \cite{zbMATH05165887}, 
Asok-Doran \cite{zbMATH05238016}, 
and Asok-{\O}stv{\ae}r \cite{asokostvar}. 
\vspace{0.1in}

Our approach makes extensive use of the six-functor formalism in stable motivic homotopy theory, 
as developed in \cite{Ayoub,CD3};
we review and complement this material in \Cref{section:shai}.
Let $S$ be a qcqs (quasi-compact quasi-separated) base scheme. 
Its stable motivic homotopy category $\SH(S)$ is a closed symmetric monoidal $\infty$-category, 
see, e.g., \cite{DRO,hoyoisquadratic,JardineMSS,robalo}.
To any separated $S$-scheme of finite type $f\colon X\to S$ we define $\Pi_{S}^{\infty}(X)$, 
the \emph{stable motivic homotopy type at infinity of $X$}, 
by the homotopy exact sequence
\begin{equation}
\label{equation:hes}
\Pi_{S}^{\infty}(X) \rightarrow f_!f^!(\un_S) \xrightarrow{\alpha_X} f_*f^!(\un_S)
\end{equation}
Here $\un_S$ is the motivic sphere spectrum over $S$, 
$f_!f^!(\un_S)=\htp_S(X)$ is the stable homotopy type of $X$ and
$ f_*f^!(\un_S)=\htp^c_S(X)$ is the properly supported stable homotopy type of $X$.
The canonical morphism $\alpha_X$ is obtained from the six-functor formalism for the stable motivic 
homotopy category $\SH(S)$, 
which implies the following fundamental properties.
\vspace{0.05in}
\begin{itemize}
\item If $X/S$ is smooth, then $f_!f^!(\un_S)=\Sigma^\infty X_+$ is the motivic suspension spectrum of $X$
\vspace{0.05in}
\item If $X/S$ is proper, then $\alpha_X$ is an isomorphism
\vspace{0.05in}
\item The morphism $\alpha_X$ is covariant with respect to proper morphisms and contravariant with respect to 
\'etale morphisms
\end{itemize}

With the intrinsic definition of $\Pi_{S}^{\infty}(X)$ in \eqref{equation:hes} we deduce a number of novel properties 
in the spirit of proper homotopy theory.
Let us fix a compactification $\bar X$ of $X$ over $S$ and denote by $\partial X$ its reduced \emph{boundary}.
Then the induced immersions $j:X \rightarrow \bar X$, $i:\partial X \rightarrow X$ form a diagram of $S$-schemes
\begin{equation}
\label{equation:compactdiagram}
\xymatrix@R=14pt@C=30pt{
X\ar@{^(->}^j[r]\ar_{f}[rd] & \bar X\ar[d]
& \partial X\ar@{_(->}_i[l]\ar^{g}[ld] \\
& S & }
\end{equation}
We observe the stable homotopy type at infinity of $X$ is determined by the data in \eqref{equation:compactdiagram} 
via a canonical equivalence
\begin{equation}
\label{equation:piinfinitycompact}
\Pi_{S}^\infty(X) 
\simeq 
g_*i^*j_*f^!(\un_S)
\end{equation}
This shows that $\Pi_{S}^\infty(X)$ is independent of the chosen compactification and that our construction has 
properties analogous to Deligne's vanishing cycle functor for \'etale sheaves, 
see \cite{SGA7II}.
We may reformulate \eqref{equation:piinfinitycompact} by means of the canonically induced homotopy exact sequence
\begin{equation}
\label{equation:piinfinityexact}
\Pi_{S}^\infty(X) \to
\Pi_{S}(\partial X) \oplus \Pi_{S}(X) 
\xrightarrow{i_*+j_*} 
\Pi_{S}(\bar X)
\end{equation}

In the notation in \eqref{equation:compactdiagram}, 
let us assume $\bar X$, $\partial X$ are smooth $S$-schemes, 
and write $N$ for the normal bundle of $\partial X$ in $\bar X$.
In \Cref{subsection:shtaivptn} we use the Euler class $e(N)$ in $\SH(S)$ to deduce the homotopy exact sequence
\begin{equation}
\label{equation:eulerhes}
\Pi_{S}^\infty(X) 
\rightarrow 
\Pi_{S}(\partial X)
\xrightarrow{e(N)} 
\Sigma^\infty\Th_{S}(N)
\end{equation}
It is helpful to think of the passage from \eqref{equation:hes} to \eqref{equation:eulerhes} in the language of 
problem-solving.
Our ``problem'' is to understand $\Pi_{S}^\infty(X)$ and the ``solution'' in the smooth case is the Euler class for 
the normal bundle of the closed immersion $\partial X\not\hookrightarrow\bar X$.
\vspace{0.1in}

In the following, 
we further assume $\bar X$ is a smooth proper $S$-scheme and $\partial X$ is a normal crossing divisor on $\bar X$.
We may write $\partial X=\cup_{i \in I} \partial_i X$ as the union of its irreducible components $\partial_i X$, 
so there is a canonical closed immersion $\nu_i:\partial_i X \rightarrow \bar X$.
For any subset $J \subset I$, 
we equip $\partial_J X:=\cap_{j \in J} \partial _j X$ with its reduced subscheme structure,
where $\cap$ is suggestive notation for fiber products over the boundary $\partial X$.
If $J \subset K$, there is a canonical proper morphism $\nu_K^J:\partial_K X \rightarrow \partial_J X$.
By means of descent for the cdh-covering 
$$
\sqcup_{i\in I}\partial_i X\to \partial X
$$ 
we identify $\Pi_{S}(\partial X)$ with the 
colimit\footnote{Limits and colimits in this paper are taken in the sense of $\infty$-categories. 
To construct functorial Gysin maps we appeal to \Cref{prop:replacement}.}
of the naturally induced diagram in $\SH(S)$
\begin{equation}
\label{equation:partialXnc}
\Pi_{S}(\partial_I X)
\longrightarrow
\bigoplus_{\sharp J= \sharp I - 1} \Pi_{S}(\partial_J X)
\begin{smallmatrix}
\longrightarrow\\
\cdots \\
\longrightarrow 
\end{smallmatrix} 
\bigoplus_{\sharp J= \sharp I - 2} \Pi_{S}(\partial_J X)
\begin{smallmatrix}
\longrightarrow\\
\cdots \\
\longrightarrow 
\end{smallmatrix} 
\cdots
\begin{smallmatrix}
\longrightarrow\\
\cdots \\
\longrightarrow 
\end{smallmatrix} 
\bigoplus_{i\in I} \Pi_{S}(\partial_i X)
\end{equation} 
The face map on the summand $\Pi_{S}(\partial_K X)$ is defined by the pushforward maps
$$
\sum_{J\subset K, \sharp J= \sharp K-1} 
(\nu_{K}^J)_*
$$
Similarly, 
we identify $\Sigma^\infty\Th_{S}(N)$ with the limit of the naturally induced diagram in $\SH(S)$
\begin{equation}
\label{equation:partialXnc2}
\bigoplus_{i\in I} 
\Sigma^\infty\Th_{S}(N_i)
\begin{smallmatrix}
\longrightarrow\\
\cdots \\
\longrightarrow 
\end{smallmatrix} 
\bigoplus_{\sharp J = 2} \Sigma^\infty\Th_{S}(N_J)
\begin{smallmatrix}
\longrightarrow\\
\cdots \\
\longrightarrow 
\end{smallmatrix} 
\bigoplus_{\sharp J = 3} \Sigma^\infty\Th_{S}(N_J)
\begin{smallmatrix}
\longrightarrow\\
\cdots \\
\longrightarrow 
\end{smallmatrix} 
\cdots
\longrightarrow 
\Sigma^\infty\Th_{S}(N_I)
\end{equation}
Here, 
$N_J$ is the normal bundle of $\partial_J X$ in $\bar X$, 
and the coface map on the summand $\Sigma^\infty\Th_{S}(N_K)$ is defined by the Gysin maps
$$
\sum_{J\subset K, \sharp J= \sharp K-1} 
(\nu_{K}^J)^!
$$ 
Our general computations culminate in \Cref{thm:maincomputation}, 
where we identify $\Pi_{S}^\infty(X)$ with the homotopy fiber of the map
$$
\colim_{n \in (\Dinj)^{op}} \left(\bigoplus_{J \subset I, \sharp J=n+1} 
\Pi_{S}(\partial_JX)\right)
\xrightarrow \mu 
\underset{n \in \Dinj}\lim \left(\bigoplus_{J \subset I, \sharp J=m+1} 
\Sigma^\infty\Th_{S}(N_J)\right)
$$
induced by
$$
(\mu_{i,j})_{i,j\in I}
\colon
\bigoplus_{i \in I} \Pi_{S}(\partial_iX) 
\longrightarrow 
\bigoplus_{j \in I} \Sigma^\infty\Th_{S}(N_j)
$$
More precisely, 
$\mu_{i,j}$
is shorthand for the composite map
$$
\Pi_{S}(\partial_iX) 
\xrightarrow{\nu_{i*}}
\Pi_{S}(\bar X) 
\rightarrow
\Sigma^\infty\left(\frac{\bar X}{\bar X-\partial_jX}\right)
\xrightarrow \simeq
\Sigma^\infty\Th_{S}(N_j)
$$

To refine these techniques, we develop a theory of duality with compact support. 
We generalize the homotopy purity theorem and give new examples of rigid objects in the process.
Our approach is based on the notion of a homotopically smooth morphism.
If $f:X \rightarrow S$ is a smoothable lci morphism with virtual bundle $\tau_f$ over $X$, 
we say that $f$ is \emph{homotopically smooth} (\emph{h-smooth}) if the naturally induced morphism 
$$
\mathfrak p_f:\Th(\tau_f) \rightarrow f^!(\un_S)
$$
is an isomorphism (see \Cref{df:hsmooth} for more details).
Any closed immersion between smooth varieties over a field is h-smooth.
When $f$ is h-smooth and $i:Z \rightarrow X$ is a closed immersion with $Z/S$ h-smooth, 
\Cref{thm:generalizedhomotopypurity} shows the relative purity isomorphism
$$
\htp_S(X/X-Z,v) \simeq \htp_S(Z,i^*v+N_i)
$$
Here, 
$v$ is a virtual vector bundle over $X$ and $N_i$ is the (necessarily regular) normal bundle of 
$i:Z \rightarrow X$.
Under the additional assumption that $\htp_S(X,v)$ is rigid, 
we show in \Cref{subsection:asdr} the duality with compact support isomorphism
$$
\htp_S(X,v)^\vee \simeq \htp_S^c(X,-v-\tau_f)
$$
This duality isomorphism can be seen as a motivic analog of classical topological results due to Atiyah 
\cite[\S3]{atiyah:thomcomplexes}, Milnor-Spanier \cite[Lemma 2]{milnorspanier}.
As an application,
we identify the stable motivic homotopy type at infinity of hyperplane arrangements in 
\Cref{subsection:smhtaioha}.
\vspace{0.1in}

We define the punctured tubular neighborhood $\TN_S(X,Z)$ of a closed immersion $i\colon Z\to X$ in 
\Cref{section:shai}.
For points on hypersurfaces in affine space, 
this key invariant specializes in links considered successfully in topology by Milnor and Mumford 
(see \cite{Milnor}, \cite{mumfordihes}).
It turns out that $\TN_S(X,Z)$ is a local invariant in the sense that it only depends on a 
Nisnevich neighborhood 
of $Z$ in $X$, and, moreover, it satisfies a cdh-excision property (see \Cref{cor:cdh-invariance}).
The geometric content of our construction is transparently visible in examples, 
e.g., 
for an ordinary double point on a threefold (see \Cref{ex:doublepointon3fold}).
We invite the interested reader to compare with Levine's notion of motivic punctured tubular neighborhoods in 
\cite{zbMATH05165887}. 

In the situation with the compactification of a separated morphism of finite type $f\colon X\to S$, 
see \eqref{equation:compactdiagram}, 
\Cref{prop:basic_comput_thp-infty} shows there exists a canonical isomorphism
$$
\htp^\infty_S(X) \simeq \TN_S(\bar X,\partial X)
$$
which is natural in $(\bar X,X,\partial X)$, 
covariantly functorial for proper maps, 
and contravariantly functorial for \'etale maps.
Via this isomorphism, we can study stable motivic homotopy types at infinity through the 
geometric construction of punctured tubular neighborhoods.
This perspective helps us clarify a few simple and unifying principles across motivic $\infty$-categories.
For example, 
we generalize Wildeshaus' analytic invariance theorem for boundary motives \cite[Theorem 5.1]{Wild1}:
A closed pair of $S$-schemes $(X,Z)$ means a closed immersion $Z\not\hookrightarrow X$ of $S$-schemes, 
and a morphism $\phi\colon (Y,T)\to (X,Z)$ is an $S$-morphism $\phi\colon Y\to X$ such that $\phi^{-1}(Z)=T$.
Suppose $f:T \rightarrow Z$ is an isomorphism that extends to an isomorphism of the respective formal completions
$\mathfrak f:\hat Y_T \rightarrow \hat X_Z$.
If $S$ is an excellent scheme, 
\Cref{thm:analytic} shows that there exists a canonical isomorphism
$$
\mathfrak f^*:
\TN_S(Y,T) 
\xrightarrow{\simeq} 
\TN_S(X,Z)
$$
In particular, 
the stable motivic homotopy type at infinity functor satisfies analytical invariance.
\Cref{prop:iso_compute_pre_Mumford} provides a way of identifying punctured tubular 
neighborhoods, without appealing to orientations,
in terms of (the homotopy fiber of) a geometrically defined fundamental class. 
\vspace{0.1in}

In \Cref{section:Mumford}, 
we employ punctured tubular neighborhoods to study a theory of motivic plumbing on surfaces;
this constitutes a refinement and extension of Mumford's seminal work in \cite{mumfordihes}.
 It provides a successful transportation of a construction from surgery theory into motivic
  homotopy, extending the ideas of \cite{zbMATH05918255}.
The setting is a closed pair \((X, D)\) consisting of a smooth surface \(X\) over a field \(k\), along with a normal crossing divisor \(D\) in \(X\) that is proper over \(k\). We will refer to this pair as a \emph{log-pair over \(k\)}. Additionally, as stated in \Cref{num:hyp_Mumford1}, we assume that for all \(i \in I\), the component \(D_i\) has a rational point \(x_i \in D_i(k)\) that does not belong to any other components of \(D\).

One part of \Cref{thm:smhatoomatrix-general}, 
which is a stable motivic homotopical analog of Mumford's calculation
in \cite{mumfordihes} obtained via the \emph{plumbing construction}, 
states that if the invertible sheaves $\omega_X|_D$ over $D$, 
and $\omega_i$ over $D_i$ for any $i \in I$, are orientable, 
then the punctured tubular neighborhood $\TN_k(X,D)$ 
--- or equivalently when $X$ is proper (\Cref{prop:basic_comput_thp-infty}) 
the homotopy at infinity $\htp^\infty_k(X-D)$ ---
is isomorphic to the cone of a map of the form
(we make the entries of the matrix explicit depending on choices of orientation classes, 
and $\htp(\mathcal D)$ denotes the ``Artin part'' of $\htp(D)$ defined in \Cref{prop:Mumford_source})
$$
\begin{pmatrix}
a & b' \\
b & \mu
\end{pmatrix}:
\htp(\mathcal D) \oplus \bigoplus_{i \in I} \un_k(1)[2] 
\rightarrow 
\htp(\mathcal D)^\vee(2)[4]  \oplus \bigoplus_{j \in I} \un_k(1)[2] 
$$
We refer to 
$\mu=(\mu_{ij})\colon\bigoplus_{i\in I}\un_k(1)[2]\rightarrow\bigoplus_{j\in I}\un_k(1)[2]$ 
as the ``quadratic Mumford matrix" since, 
over the complex numbers, 
the above specializes to computations carried out in \cite{mumfordihes}.
Its coefficients take values in the endomorphism ring of the sphere spectrum or unit $\un_k$.
We interpret $\mu_{ij}$ as the class of a quadratic form 
$(\partial_i X,\partial_j X)_{quad} \in \GW(k)$
 in the Grothendieck-Witt ring called the {\it quadratic degree} 
of the intersections of the divisors $\partial_iX$ and $\partial_j X$.
The close connection with quadratic forms arises since elements of the $i$th Chow-Witt group are represented by 
formal sums of subvarieties $Z$ of codimenison $i$ equipped with an element of $\GW(k(Z))$.
Moreover, 
the rank of the quadratic degree equals the corresponding Mumford degree.
\vspace{0.1in}

In \Cref{sec:KthNCD}, 
we discuss algebraic $K$-theory and Picard groups of $1$-dimensional schemes 
and normal crossing divisors on regular $2$-dimensional schemes.
We demonstrate that Thom spaces over a (possibly singular) 
$1$-dimensional base scheme can be trivialized if an orientation class exists.
The main result, 
\Cref{thm:orientationncdsurfaces}, 
identifies the pointed set of orientation classes of line bundles
over (eventually singular) $1$-dimensional schemes.  
Our findings in \Cref{subsection:thetacharacteristic} are applicable to 
arbitrary normal crossing divisors on surfaces;
if each branch has a positive genus, 
we assume they are oriented, or in other words, 
equipped with a Theta characteristic.
The results in Section 5 depend on our notion of an orientation class 
introduced in \Cref{subsection:ovbaqi}.
We show that several constructions in motivic homotopy theory, 
e.g., 
quadratic degree \cite{zbMATH07217788},
Gysin maps for Chow-Witt groups \cite{arXiv:2403.09266}, \cite{zbMATH07217786},  
and quadratic linking degrees \cite{arXiv:2210.11048} 
depend on choosing an orientation class, 
see \Cref{subsection:q0caqd}.
\vspace{0.1in}

Further, we specialize our results to motives.
When $k$ is a finite field, a global field, or a number ring, 
we have the motivic $t$-structure on rational Artin-Tate motives at our disposal
(see \cite{LevineAT} for the case of fields, and \cite{ScholbachAT} for number rings).
We let $\DMAT(K,\QQ)$ be the triangulated category of (constructible) rational Artin-Tate motives.
From \cite{LevineAT} it follows that $\DMAT(K,\QQ)$ admits a motivic t-structure,
whose heart is the Tannakian category $\MMAT(K,\QQ)$ of Artin-Tate motives.
In particular, 
one gets a homological and monoidal functor
$$
\HM_0:\DMAT(K,\QQ) \rightarrow \MMAT(K,\QQ)
$$
We define the Artin-Tate motive
$$
\HM_i(\TN(X,D)):=\HM_0(\TN(X,D)[-i])
$$
as the $i$-th (motivic) homology of the punctured tubular neighborhood of $(X,D)$.
When $X$ is in addition proper over $K$, 
this is the homology of the boundary motive of $(X-D)$
(see \Cref{ex:realization} and \Cref{prop:basic_comput_thp-infty}), 
or the \emph{motivic homology at infinity}
$$
\HM_i^\infty(X-D)=\HM_i(\TN(X,D))
$$
In \Cref{cor:Mumford_ATmotives} we show the homology motive $\HM_i(X)$ vanishes for $i\not\in [0,3]$ 
and there is an exact sequence in the Tannakian category $\MMAT(S,\QQ)$ of Artin-Tate motives
\begin{align*}
0 \rightarrow &\HM_3(\TN(X,D)) \rightarrow \bigoplus_{i \in I} \un_S(2)
\xrightarrow{\sum_{i<j} p_{ij}^{i!}-p_{ij}^{j!}} \bigoplus_{i<j} M_S(D_{ij})(2)  \\
& \rightarrow \HM_2(\TN(X,D)) \rightarrow \bigoplus_{i \in I} \un_S(1) 
\xrightarrow{\ \mu\ } \bigoplus_{j \in I} \un_S(1) \\
& \rightarrow \HM_1(\TN(X,D)) \rightarrow \bigoplus_{i<j} M_S(D_{ij})
\xrightarrow{\sum_{i<j} p^i_{ij*}-p^j_{ij*}} \bigoplus_{i \in I} \un_S 
\rightarrow \HM_0(\TN(X,D)) \rightarrow 0
\end{align*}
Here $\mu$ is the quadratic Mumford matrix and $M_S(D_{ij})$ is the 
mixed Artin-Tate motive of $D_{ij}=D_i\times_X D_j$.
In the above, 
$\HM_0(\TN(X,D))$ and $\HM_3(\TN(X,D))$ are pure of respective weights $0$ and $-4$,
while $\HM_1(\TN(X,D))$ and $\HM_2(\TN(X,D))$ are mixed of weights $\{0,-2\}$ and $\{-2,-4\}$, 
respectively (see \cite{IvMo} for the notion of weights).
We extend the above result to the case where the components of $D$
may have positive genus, 
at the price of working in the category of integral Nori motives $\MMN(K,\ZZ)$
when $K$ is a field of characteristic $0$ with a fixed complex embedding; see \Cref{prop:Mumford_motives} for a precise formulation.

Moreover, 
we study the example of \emph{Ramanujam's surface} $\Sigma$ \cite{Ramanujam}.
Over the complex numbers, 
it is a topologically contractible affine algebraic surface which is not 
homeomorphic to the affine plane.
Working over a field $k$ of characteristic different from $2$, 
\Cref{ex:ramsurface} identifies $\Sigma$'s integral motive at infinity $M^\infty(\Sigma)$ with $\un_k\oplus \un_k(2)[3]$.
\vspace{0.1in}

Our setup provides universal formulas in the various realizations of motives, 
e.g., 
$\ell$-adic, rigid, syntomic, Galois representations, etc.
For example, 
the computation \eqref{equation:RZgeneralization} specializes under $\ell$-adic realization 
to the Rapoport-Zink formula for vanishing cycles \cite[Lemma 2.5]{RapZink}, 
and similarly for Steenbrink's limit Hodge structures \cite{Steenb}.
We expect that \Cref{cor:snccorollary} yields an explicit formula for Ayoub's nearby cycles in the semi-stable case,
cf.~\cite{zbMATH05194871}.
\vspace{0.1in}

We illustrate the general case with concrete examples of $\mathbf{A}^{1}$-equivalent smooth affine surfaces with 
non-isomorphic stable motivic homotopy types at infinity.
For any integer $n>0$, 
the Danielewski surface $D_{n}$ is the closed subscheme of $\mathbf{A}^{3}$ cut out by the equation $x^{n}z=y(y-1)$, 
see \cite{Danielewskisurfaces}.
We note that $D_{1}$ is the Jouanolou device over $\mathbf{P}^{1}$;
in fact, 
$D_{n}$ is $\mathbf{A}^{1}$-equivalent to $\mathbf{P}^{1}$ \cite[\S 3.4]{asokostvar}.
Over any field $k$, 
one can distinguish between $\Pi_{k}^{\infty}(D_{m})$ and $\Pi_{k}^{\infty}(D_{n})$ for $m\neq n$
by viewing Danielewski surfaces as affine modifications of $\mathbf{A}^{2}$.
We refer to \Cref{subsection:examples} for precise statements and further examples, 
\cite{zbMATH07149737} for background on $\mathbf{A}^{1}$-contractibility of affine modifications, 
and \cite{fieseler} for first homology at infinity of Danielewski surfaces over the complex numbers.
The affine modifications give an affirmative answer to Problem 3.4.5 in \cite{asokostvar}.
\vspace{0.1in}

At this stage, 
we should come clean on some technical points concerning fundamental classes and orientations.
First, 
our setup gives a quadratic generalization of Mumford's plumbing construction 
\cite{mumfordihes} using Chow-Witt groups.
While Mumford uses orientations on the normal bundles of the branches,
which are copies of the projective line, 
much of the subtleties in our setting come from working with twisted Milnor-Witt $K$-theory sheaves.  
The latter is needed to compute the quadratic degree maps of the intersections of the branches 
taking values in the Grothendieck-Witt ring.
On the one hand, we develop the idea of parallelization to compute 
"the fundamental class of the diagonal" in terms of 
motivic fundamental classes \cite{DJK}.
In another direction closely related to differential geometry and quadratic enumerative geometry, 
we discuss the foundations for orientations of algebraic vector bundles via quadratic isomorphisms. 
Making clever choices of orientation classes is a key point in our computations of 
quadratic Mumford matrices.
This approach enables us to  
compute stable motivic invariants without appealing to $\SL$-orientations.
\Cref{section:ovbaqi} explains this material, 
e.g., 
the orientation classes of invertible sheaves on arbitrary schemes,
where we also introduce and show some fundamental properties of quadratic Picard groupoids.
\vspace{0.1in}

Punctured tubular neighborhoods can also be applied to the study of isolated singularities of surfaces, in particular rational double points, also known as Du Val singularities. In characteristic $p>0$,
Artin \cite{zbMATH03557923} showed that the \'etale local fundamental group of such a singularities 
cannot always distinguish between double and regular points.
We show that, 
with the exception of $E_8$-type singularity, 
the stable motivic link $\TN(\Gamma)$ of a Du Val singularity is different from the stable motivic link 
of $\TN(\mathbb{A}^2_{k},\{0\})=\un_k\oplus \un_k(2)[3]$. 
In particular, 
$\TN(\Gamma)$ distinguishes Du Val singularities other than $E_8$ from regular points. 
For $E_8$ and the complex numbers, 
the identification $\TN(E_8)\simeq\TN(\mathbb{A}^2_{k},\{0\})$ reflects the fact that the topological 
link of $E_8$ is the Poincaré homology $3$-sphere $\Sigma(2,3,5)$ \cite{zbMATH02668474}, 
a compact topological $3$-manifold with the same singular homology groups as $S^3$, 
whose fundamental group is isomorphic to the binary dodecahedral group.
We refer to \Cref{tab:DuVal-links} for a summary of our computation of stable motivic links of Du Val 
singularities.
\vspace{0.1in}

A final comment is that defining the stable homotopy type at infinity $\htp_{S}^{\infty}$ is the first step towards a refined invariant in 
unstable motivic homotopy theory.
The problem of defining unstable motivic homotopy types at infinity witness the tension 
between unstable and stable motivic homotopy theory.
For example, 
the six functor formalism is not available in the unstable setting.
To remedy this, 
one can take into account all possible smooth compactifications.
Nonetheless, 
some of the techniques developed in this paper will carry over to unstable motivic homotopy categories, 
e.g., 
the calculations in \Cref{sec:crossing_sing} hold in the cdh-topology, 
and one can expect more developments along these lines.

\begin{rem}
\label{remark:generality}
This paper's results hold more generally for any motivic $\infty$-category such as 
triangulated and abelian mixed motives, Artin-Tate motives, \'etale motives, 
torsion and $\ell$-adic categories, mixed Hodge modules, ... in place of $\SH$. 
If there exists a realization functor that commutes with the six operations, e.g., 
the Betti or $\ell$-adic realizations, then this follows from the universality of $\SH$.
\end{rem}

\subsubsection{Conventions}
Our results are couched in the axiomatic setting of \cite{CD3}, \cite{Khan} which complements \cite{Ayoub}.
We fix a \emph{motivic $\infty$-category} (\cite[Definition 2.4.45]{CD3}) $\T$ over the category of qcqs schemes, 
i.e., 
a \emph{monoidal stable homotopy functor} according to \cite{Ayoub}. 
Our primary example is the motivic stable homotopy category $\SH$.
In the language of presentable stable monoidal $\infty$-categories \cite{Khan}, 
$\SH$ is the initial motivic $\infty$-category.
Thus there is a unique morphism of motivic $\infty$-categories $\SH \rightarrow \T$.
To maintain intuition, we shall refer to the objects of $\T(S)$ as $\T$-spectra over $S$.
For more details, see Section \ref{sec:notations}.

%% file: recall.tex
\subsection{Thom spaces}
\begin{num}\label{num:Thom_spaces}
The \emph{Thom space} of a vector bundle $p:V \rightarrow X$ with zero section $s:X\to V$ is the object
$$
\Th(V)=\Th_X(V):=p_\sharp s_*(\un_X) \in \T(X)
$$
Here $p_\sharp$ is the left adjoint of $p^*$. For a coherent locally free sheaf of $\cO_X$-modules $\mathcal{E}$, we use also sometimes use the notation $\Th(\mathcal{E})$ as a short hand for $\Th(\mathbb{V}(\mathcal{E}))$. The \emph{Tate twist} is a particular case of this notation, 
namely, 
we have $\un_X(n)=\Th(\cO_X^n)[-2n]=\Th(\mathbf{A}^n_X)[-2n]$.
According to the stability property of $\T$ (\cite[2.4.4, 2.4.14]{CD3}),
the object $\Th(V)$ is $\otimes$-invertible in $\T(X)$ with $\otimes$-inverse (\cite[2.4.1, 2.4.12]{CD3})
$$
\Th(-V):=s^!p^*(\un_X)=s^!(\un_V)
$$
The construction of Thom spaces is functorial in $V$ and, as a consequence of the localization property of $\T$ 
(\cite[2.4.6, 2.4.10]{CD3}),
it uniquely extends to a monoidal functor with values in the associated homotopy category
 (cf.~\cite[2.4.15]{CD3} and \cite[1.5.18]{Ayoub})
$$
\Th:\uK(X) \rightarrow \ho \T(X)
$$
from Deligne's category $\uK(X)$ of virtual locally free sheaves on $X$.

For an arbitrary (resp. separated) morphism of schemes $f:Y\rightarrow X$ and a virtual vector bundle $v$ over $X$, 
the projection formula and the $\otimes$-invertibility of $\Th(V)$ imply the exchange isomorphism
\begin{equation}
    \label{eq:exchange-iso}
f^*\Th(v) \xrightarrow{\simeq} \Th(f^{-1}v) 
\quad \quad (\textrm{resp. } 
\Th(f^{-1}v)\otimes f^!(\un_X) \xrightarrow{\simeq} f^!\Th(v))
\end{equation}
\end{num}
 
To comply with Morel-Voevodsky's definition, we introduce the following.
\begin{df}\label{df:Thom_spaces}
Let $f:X \rightarrow S$ be a smooth morphism and let $v$ a virtual vector bundle over $X$.
The \emph{Thom space of $v$ relative to $S$} is the 
object 
$$
\Th_S(v)=f_\sharp(\Th(v))
$$
Beware that when $f$ is not the identity, the functor $\Th_S$ is not monoidal.

In the sequel, 
when we do not indicate the base of a Thom space, 
we consider it over the same base scheme as the virtual bundle.
\end{df}

\begin{ex}\label{ex:thom_spaces}
\begin{enumerate}
\item If $\T=\SH$ and $v=\twist V$ for a vector bundle $V/X$, then by homotopy purity 
$
\Th_S(v)
\simeq 
\Sigma^\infty (V/V^\times)
$. 
\item If $\T=\tDM$, 
the Thom space $\Th_S(v)$ depends only on the rank and determinant of $v$ 
(see \cite[\textsection 7]{DFJK} for a more precise statement).
\item If $\T$ is oriented in the sense of \cite[2.4.38]{CD3}, 
e.g., 
any category under $\tDM$ in \eqref{eq:motivic_cat}, 
then for every virtual vector bundle $v$ of virtual rank $n$ on a smooth $S$-scheme $p:X\to S$, 
there is a canonical \emph{Thom isomorphism}
$
\Th_S(v) \xrightarrow{\simeq} \un_S(n)[2n]
$
compatible with pullbacks and the $\otimes$-structure on the functor $\Th$.
Since Thom spaces are always reduced to Tate twists for oriented theories,   
this is mainly interesting for generalized theories such as Chow-Witt groups, 
hermitian K-theory, 
and stable (co)homotopy.
\end{enumerate}
\end{ex}

\begin{rem}
Following the procedure of \cite[\textsection 16.2]{BH_norms},
 it is possible to refine the construction of the Thom space at the $\infty$-categorical level.
 More precisely, one builds a monoidal $\infty$-functor, still denoted as above,
$$
\Th:\Ksp(X) \rightarrow \T(X)
$$
where $\Ksp(X)$ is the monoidal $\infty$-groupoid associated with Thomason-Trobaugh K-theory space
 of $X$.\footnote{In \emph{op. cit.}, when considered with values in the stable motivic homotopy category,
 this functor is called the motivic $J$-homomorphism.}
 Note also that this $\infty$-functor can in fact be made natural in $X$,
 with respect to the contravariant $\infty$-functoriality
 of its source and target. 
 In the sequel of this work, we will not need this refinement, as
 we will not use the higher functoriality of Thom spaces.
\end{rem}

\subsection{Internal theories and functoriality}
The six functors formalism encodes the axioms of four (co)homology theories; 
see, e.g., \cite{BO} for the combination of cohomology and Borel-Moore homology. 
Next, we give a systematic definition from the motivic point of view.

\begin{df}
\label{df:internal_theories}
Let $f:X \rightarrow S$ be a separated morphism and let $v$ a virtual vector bundle over $X$.
One associates to $X/S$ and $v$ the following objects of $\T(S)$:
\begin{itemize}
\item Homotopy: $\htp_S(X,v)=f_!(\Th(v) \otimes f^!(\un_S))$
\item Cohomotopy: $\cohtp_S(X,v)=f_*(\Th(v) \otimes f^*(\un_S)) \simeq f_*(\Th(v))$
\item Borel-Moore (or properly supported) homotopy: $\htp^c_S(X,v)=f_*(\Th(v) \otimes f^!(\un_S))$
\item Properly supported cohomotopy: $\cohtp^c_S(X,v)=f_!(\Th(v) \otimes f^*(\un_S)) \simeq f_!(\Th(v))$
\end{itemize}
When $v=0$, we simply write $\htp_S(X)$, $\cohtp_S(X)$, $\htp^c_S(X)$, $\cohtp^c_S(X)$.

The natural transformation $\alpha_f:f_! \rightarrow f_*$ yields canonical maps:
\begin{align}\label{eq:forget_supp}
\alpha_{X/S}&:\htp_S(X,v) \rightarrow \htp^c_S(X,v) \\
\label{eq:forget_supp2}\alpha'_{X/S}&:\cohtp^c_S(X,v) \rightarrow \cohtp_S(X,v) \text{ (``forgetting proper support")}
\end{align}
Both $\alpha_{X/S}$ and $\alpha'_{X/S}$ are isomorphisms whenever $X/S$ is proper.
\end{df}

\begin{rem}
If $X/S$ is smooth separated, 
$\htp_S(X)$ is called the premotive of $X/S$ in \cite{CD3}. 
For all $\T$, 
with the exception of $D^b_c(-,\ZZ_\ell)$, 
the objects $\Pi_S(X)(n)$ for $X/S$ smooth generate $\T(X)$ under colimits.
\end{rem}

\begin{ex}\label{ex:htp&cohtp}
Here is a summary comparing our notations with more familiar ones.
\begin{enumerate}
\item $\T=\SH$ and $X/S$ smooth: $\htp_S(X)=\Sigma^\infty X_+$ and for a vector bundle $V$ on $X$, 
we have $\htp_S(X,\langle V \rangle)=\Sigma^\infty \Th(V)$.
\item $\T=\DM$ and $X/S$ smooth: $\htp_S(X)$ is Voevodsky's motive $M_S(X)$ of $X/S$.
 When $X/S$ is proper and $X$ is regular, $\cohtp_S(X)=:h_S(X)$ is the \emph{relative Chow-motive} of $X/S$.
 It is a pure motive of weight $0$ in the sense of Bondarko. 
 See \cite{Jin} for the comparison of these objects with Corti-Hanamura's definition.
\item $\T=\DM$, $k$ a perfect field, $X/k$ smooth separated:
 $\htp_k(X)=M(X)=\underline C_*L(X)$, where, with the notations of \cite[chap. 5]{FSV},
 $\underline C_*$ is the Suslin complex functor, and $L(X)$ is the sheaf with transfers represented by $X$.
 If $k$ is of characteristic $0$,
 or one works with $\DM[1/p]$ if $k$ has characteristic $p>0$,
 then $\htp_k^c(X)=M^c(X)=\underline C_*L^c(X)$
 where $L^c(X)$ the sheaf of quasi-finite correspondences (see \cite[chap. 5]{FSV} in characteristic $0$
 and \cite[8.10]{CD5} in general).
\item $\T=D^b_{c}(-_\et,\ZZ_\ell)$ and $f:X \rightarrow S$ any morphism: $\cohtp_S(X)=\derR f_*(\ZZ_\ell)$ is the complex computing \'etale cohomology of $X$ in $D^b_{c}(S_\et,\ZZ_\ell)$.
 In particular, if $S=\Spec(k)$, the complex compute absolute \'etale cohomology of $X$ after forgetting
 the action of the absolute Galois group of $k$.
 Similarly, $\cohtp_S^c(X)$ computes cohomology with compact support.
\item $\T=\DM_h$: using the model category of \cite{CD5}, for a smooth $S$-scheme $X$,
  $\htp_S(X)$ is obtained as the infinite suspension of the $\h$-sheaf represented by $X$.
\end{enumerate}
\end{ex}

\begin{rem}
As explained in \Cref{sec:notations}, the comparison functors from $\SH$ to the other motivic categories $\T$
 considered in \emph{loc. cit.} commute with the six operations provided that one restricts to excellent base schemes.
In particular, the four internal theories considered in $\SH$ realize the corresponding theories in $\T$ -- of course, this universal property of $\SH$ was at the heart of Voevodsky's theory since the beginning.
See \cite{DrewGall} for a complete account incorporating the six functors.
Practically any assertion concerning these internal theories proved in $\SH$ is equally valid in $\T$.
\end{rem}

%

\begin{num}
\label{num:natural_fct} 
\textit{Natural functoriality}:
For a morphism $f \colon Y \rightarrow X$ between  separated $S$-schemes, we have the following naturally induced maps (which explain our choice of terminology): 
\begin{itemize}
\item $f_*:\htp_S(Y,f^{-1}v) \rightarrow \htp_S(X,v)$
\item $f^*:\cohtp_S(X,v) \rightarrow \cohtp_S(Y,f^{-1}v)$
\item $f_*:\htp^c_S(Y,f^{-1}v) \rightarrow \htp^c_S(X,v)$, when $f$ is proper
\item $f^*:\cohtp^c_S(X,v) \rightarrow \cohtp^c_S(Y,f^{-1}v)$, when $f$ is proper
\end{itemize}
In addition, 
when $f$ is proper then the comparison maps $\alpha_{X/S}$ and $\alpha'_{X/S}$ (see \eqref{eq:forget_supp} and \eqref{eq:forget_supp2} ) are compatible with 
$f_*$ and $f^*$.
\end{num}

\begin{rem}\label{rem:functoriality}
\begin{enumerate}
\item Given an arbitrary virtual bundle $w$ on $Y$, there is in general
 no pushforward map on (internal) homotopy $\htp_S(Y,w) \rightarrow \htp_S(X,v)$.
 To get such a map, one has to give an isomorphism $w \simeq f^{-1}(v)$ of virtual vector bundles.
\item Each of the above functoriality can in fact be enhanced into an $\infty$-functor.
 See \cite[2.1.11]{EHKSY} for the precise formulation.
\end{enumerate}
\end{rem}

\begin{ex}
\label{ex:topological_invariance}
Suppose $X/S$ is a separated $S$-scheme, 
and let $\nu:X_0 \to  X$ be the immersion on the underlying reduced subscheme
 (in fact any nil-immersion will work). 
The localization property for $\T$ implies that $(\nu^*,\nu_*)$ is an equivalence of categories (\cite[2.3.6]{CD3}).
As $\nu_*=\nu_!$ it follows that $\nu^*=\nu^!$.
For any virtual vector bundle $v$ on $X$ and $v_0=\nu^*(v)$,
one deduces the naturally induced isomorphisms 
\begin{align*}
\nu_*:\htp_S(X_0,v_0) \xrightarrow{\simeq} \htp_S(X,v), &\ \nu_*:\htp^c_S(X_0,v_0) \xrightarrow{\simeq} \htp^c_S(X,v) \\
\nu^*:\cohtp_S(X,v) \xrightarrow{\simeq} \cohtp_S(X_0,v_0), &\ \nu^*:\cohtp^c_S(X,v) \xrightarrow{\simeq} \cohtp^c_S(X_0,v_0)
\end{align*}
In particular, with $v=0$, we get
$$
\htp_X(X_0) 
\simeq \htp^c_X(X_0)
\simeq \cohtp_X(X_0)
\simeq \cohtp^c_X(X_0)
\simeq \un_X
$$
\end{ex}

 \begin{num}\label{def:stably-A1-cont}
 A smooth separated $S$-scheme $f:X \rightarrow S$ is said to be \emph{stably $\AA^1$-contractible over $S$}  
 if the induced map 
  $f_*:\htp_S(X) \rightarrow \un_S$ is an isomorphism. Note that due to the existence of the conservative family $(s^*)_{s \in S}$ of \cite[Prop. 4.3.17]{CD3}, this property is equivalent to ask that for every point $s \in S$, the fiber $X_s$ is stably $\AA^1$-contractible
 over $\kappa(s)$.  
 \end{num}

\begin{lm}
\label{lm:stably_A1-contract_virtual_vb} 
Let $S$ be a regular scheme and suppose $f:X \rightarrow S$ is stably $\AA^1$-contractible over $S$. 
Then every virtual bundle $v$ over $X$ is constant relative to $S$, 
i.e., 
$v=f^*v_0$ for some virtual vector bundle $v_0$ over $S$.

Moreover, let $T$ be the  tangent bundle of $X/S$ and let $v_0$ be the virtual vector bundle over $S$ such that $\langle T \rangle=f^*v_0$. Then there is a naturally induced isomorphism
$$
f_*f^!(-) \simeq \Th_S(v_0) \otimes -
$$
\end{lm}
\begin{proof}
The first assertion is a consequence of the representability of $K_0$ in $\SH(S)$. To prove the assertion, one considers for every object $\E$ of $\T(S)$ the composite of exchange isomorphisms
$$
f_*f^!(\E) \stackrel{(a)}\simeq f_*(\Th(T) \otimes f^*(\E))=f_*(\Th(f^*v_0) \otimes f^*(\E))
\stackrel{(b)}\simeq \Th(v_0) \otimes f_*f^*(\E) \stackrel{(c)}\simeq \Th(v_0) \otimes \E
$$
Here (a) is an instance of the relative purity isomorphism, 
(b) follows from the fact that $\Th(v_0)$ is $\otimes$-invertible, 
and (c) holds because $f$ is a stable $\AA^1$-weak equivalence and since $f$ is smooth,
 one has: $f_*f^*(\E)\simeq\uHom(\Pi_S(X),\E)$.\footnote{Recall the last isomorphism follows from the axioms
 of premotivic categories: indeed by the smooth projection formula, $f_\sharp f^*(-)=\Pi_S(X) \otimes -$
 and we conclude as $f_*f^*$ is right adjoint to $f_\sharp f^*$.} 
\end{proof}

The following statement is analogous to the traditional definition of relative homology and cohomology.

\begin{df}
\label{df:quotient}
Let $f\colon Y \rightarrow X$ be a morphism of separated $S$-schemes and let $v$ be a virtual vector bundle $v$ over $X$. 
We denote the homotopy cofiber of $f_*:\htp_S(Y,f^{-1}v) \rightarrow \htp_S(X,v)$ by $\htp_S(X/Y,v)$ so that there is a
homotopy exact sequence
$$
\htp_S(Y,f^{-1}v) \xrightarrow{f_*} \htp_S(X,v) \rightarrow \htp_S(X/Y,v)
$$
Dually, we denote the homotopy fiber of $f^*:\cohtp_S(X,v) \rightarrow \cohtp_S(Y,f^{-1}v)$ by $\cohtp_S(X/Y,v)$ so that there is a
homotopy exact sequence
$$
\cohtp_S(X/Y,v) \rightarrow \cohtp_S(X,v) \xrightarrow{f^*} \cohtp_S(Y,f^{-1}v)
$$
\end{df}


\subsection{Fundamental classes, homotopical smoothness and purity}

\begin{num}
\label{num:Gysin}
\textit{Exceptional functoriality (Gysin maps)}: 
Due to the existence of the fundamental classes introduced in \cite{DJK} the four theories in 
\Cref{df:internal_theories} satisfy exceptional functoriality (see \cite[4.3.4]{DJK} 
for the general case of a triangulated motivic category).

Let $f:Y \rightarrow X$ be a smoothable lci morphism, 
i.e., 
$f$ factors as a regular closed immersion followed by a smooth morphism, 
with cotangent complex $\cL_f$ and associated virtual tangent bundle $\tau_f$. 
One deduces, 
from the system of fundamental classes in \cite[Theorem 3.3.2]{DJK}, 
the canonical natural transformation
\begin{equation}
\label{eq:fdl_class}
\mathfrak p_f(-):\Th(\tau_f) \otimes f^* \rightarrow f^!
\end{equation}
By adjunction, 
one deduces trace and cotrace maps (see \textsection 4.3.4 in \emph{loc. cit.})
$$
\begin{array}{lcr}
\mathrm{tr}_f:f_!(\Th(\tau_f) \otimes f^*) \rightarrow Id & \textrm{and} & 
\mathrm{cotr}_f:Id \rightarrow f_*(\Th(-\tau_f) \otimes f^!)
\end{array}
$$
The latter maps induce the \emph{Gysin maps}:
\begin{itemize}
\item $f^!:\htp_S(X,v) \rightarrow \htp_S(Y,f^{-1}v-\tau_f)$, when $f$ is proper
\item $f_!:\cohtp_S(Y,f^{-1}v+\tau_f) \rightarrow \cohtp_S(X,v)$, when $f$ is proper
\item $f^!:\htp^c_S(X,v) \rightarrow \htp^c_S(Y,f^{-1}v-\tau_f)$
\item $f_!:\cohtp^c_S(Y,f^{-1}v+\tau_f) \rightarrow \cohtp^c_S(X,v)$
\end{itemize}
Again, 
assuming $f$ is proper, 
the comparison maps $\alpha_{X/S}$ and $\alpha'_{X/S}$ are compatible with the above Gysin morphisms in the obvious sense.
\end{num}

\begin{rem}
In \Cref{sec:functorial-Gysin}, we will show how to turn some of the above Gysin maps into an $\infty$-functor.
\end{rem}

\begin{num}\label{num:fond-classes}\textit{Fundamental classes}. 
Characteristic classes are cohomology classes  used for classification and computations. It is also possible to define these invariants as cohomotopy  classes. 
Recall also that fundamental classes extend to bivariant homotopy (suitably twisted), see \cite{DJK} as already mentioned in \Cref{num:Gysin}.

\begin{ex}
\label{num:Euler_seq}
\textit{Euler exact sequence and Euler classes}.
Let $f:X\rightarrow S$ be a smooth $S$-scheme and let $V=\mathbb{V}(\mathcal{E})$ be a vector bundle of rank $r$ on $X$. From the localization triangle associated with the zero section $s$ of $V$ and the homotopy property
$\htp_S(V) \simeq \htp_S(X)$, one derives the homotopy exact sequence
$$
\Th_S(V)[-1] \rightarrow \htp_S(V^\times) \rightarrow \htp_S(X) \xrightarrow{s^!} \Th_S(V)
$$
Note that, by definition, 
when $X=S$, 
then $s^!:\un_X \rightarrow \Th(V)$ is the realization in $\T(X)$ of $V$'s Euler class $e(V)\in\SH(X)$
defined in \cite[Definition 3.1.2]{DJK}.
When $f:X \rightarrow S$ is not the identity, then $s^!$ is the image of the realization of $e(V)$ by $f_!$.
This justifies our notation $e_S(V,\T)=s^!$.
In particular,
note that $e_S(V,\T)$ is zero whenever $V$ contains the trivial line bundle $\AA^1_X$ as a direct summand 
(\emph{loc. cit.}, Corollary 3.1.8).

In the case $S$ is the spectrum of a field, we have the following: 
\begin{enumerate}
\item When $\T=\DM$ or, more generally, when $\T$ is oriented, the \emph{motivic Euler class}
 $$e(V):\un_X \rightarrow \Th(V)\simeq \un_X(n)[2n]$$ corresponds to the top Chern class
 $c_n(V)$ under the isomorphism $\mathbb H_M^{2n,n}(X) \simeq \CH^n(X)$.
\item As a map in $\tDM(X)$, 
the realization of the stable homotopy Euler class $e(V)$ 
corresponds to Barge-Morel-Fasel's Euler class
 in the Chow-Witt group $\CHt^n(X,\det \mathcal{E}^{\vee})$ of $X$ twisted by the determinant of $\mathcal{E}^{\vee}$. 
\end{enumerate}
\end{ex}

 For a smoothable lci morphism $f:X \rightarrow S$ with virtual tangent bundle $\tau_f$
 one has the canonical class
$$
\eta_f:\Th(\tau_f) \rightarrow f^!(\un_S)
$$
which we will consider as a homotopy class in
$$ 
H_0^{\T}(X/S,\tau_f):=[\Th(\tau_f),f^!(\un_S)]=[f_!(\Th(\tau_f)),\un_S] 
$$
for the bivariant homology theory (with respect to ${\T}$) of $X/S$ and twist $\tau_f$.
In fact, this bivariant class is a cohomotopy class; that is, an element of the abelian group
$$
H_{\T}^n(X,\tau_f):=[\un_X,\Th(\tau_f)[n]]
$$
We impose the following assumptions.
\begin{enumerate}
\item $f$ is proper.
\item there exists a virtual bundle $v$ over $S$ and
 an isomorphism $\epsilon:\tau_f \simeq f^{-1}(v)$.
 The couple $(\epsilon,v)$, or simply $\epsilon$ when $v$ is clear,  will be called an \emph{$f$-parallelization} 
 of $\tau_f$.
\end{enumerate}
In this case, we can consider the composite map
$$
H_{\T}^0(X) \xrightarrow{\epsilon_*} H_{\T}^0(X,\tau_f-f^{-1}v)
 \xrightarrow{f_!} H_{\T}^0(S,-v)
$$
Here, the choice of $\epsilon$ yields the first map, and the second one is the Gysin map in cohomotopy 
(see \emph{op. cit.}).
The image of the unit element $1$ in cohomotopy $H_{\T}^0(X)$
can be deduced from the fundamental class $\eta_f$ via the composite
$$
\Th(v) \xrightarrow{ad_f} f_*f^*(\Th(v))\simeq f_!(\Th(f^{-1}v))
 \xrightarrow{\epsilon_*} f_!(\Th(\tau_f)) \xrightarrow{\eta_f} \un_S
$$
\end{num}
\begin{df}
Let $f:X \rightarrow S$ be a proper smoothable lci map
 with an $f$-parallelization $(\epsilon,v)$ of its virtual tangent bundle.
 The associated twisted fundamental class is given by 
$$
\eta_f^\epsilon=f_!\epsilon_*(1) \in H_{\T}^0(S,-v)
$$
When $f=i:Z \rightarrow X$ is a regular closed immersion,
 and we consider an $f$-parallelization $(\epsilon,v)$
 of its normal bundle $N_i$, corresponding to an $f$-parallelization 
 $\epsilon':\tau_i=-\twist{N_i} \rightarrow -v$, we also
 define the twisted fundamental class of $(Z,\epsilon)$ in $X$ as
$$
[Z]_X^\epsilon=f_!\epsilon'_*(1) \in H_{\T}^0(X,v)
$$
\end{df}

\begin{ex}
In our definition, the reader might be surprised by the cohomotopical index $0$.
 The "true" degree is hidden in the twist. In particular,
 for $\T=\DM$ (resp. $\tDM$),
 and a rank $d$ virtual bundle $v$ over a smooth $k$-scheme $X$, we have
$$
H^0_{\DM}(X,v) \simeq \CH^d(X), \; \text{(resp. } H^0_{\tDM}(X,v) \simeq \CHW d(X,\det v)\text{)} 
$$
The Chow (resp. Chow-Witt) group of $X$ (resp. twisted by the invertible sheaf $\det(v)$). For $\T=\SH$, there is also a canonical isomorphism $
H^0_{\SH}(X,v) \simeq \CHt^d\big(X,\det(v)\big)
$, see \Cref{prop:cohomology-tchd} in the Appendix. 
In the motivic case or any of the oriented triangulated motivic categories
 of \eqref{eq:motivic_cat},
 the motivic fundamental class of a closed immersion $i:Z \rightarrow X$
 and $f$-parallelization $(\epsilon,v)$ is the usual cycle class of
 $Z$ in $\CH^d(X)$ (resp. in the relevant cohomology in degree $2d$ and twist $d$).
 It is independent of the chosen $f$-parallelization.
 This is not the case in the category of Milnor-Witt motives and in $\SH$, 
 as modifying the twist $\cL$ in $\CHW d(X,\cL)$ can change the group. 
 \end{ex}

\begin{ex}\label{ex:i-parallelization}
Given a regular closed immersion $i:Z \rightarrow X$, a  way to obtain an $i$-parallelization of the normal bundle $N_i$
 is to consider an lci morphism $p:X \rightarrow Z'$ such that $p \circ i$ is \'etale.
 Indeed, in that case, if $\tau_p$ denotes the virtual tangent bundle of $p$,
 we get a canonical isomorphism $\epsilon:\twist{N_i} \simeq i^{-1}\tau_p$
 as the tangent bundle of $p \circ i$ is trivial. 

An important example for us comes from the diagonal immersion
 $\delta:X \rightarrow X \times_S X$ of a smooth $S$-scheme $X$.
It admits two smooth retractions given by the projections $p_j$, for $j=1,2$. 
We denote the corresponding twisted fundamental classes by 
$$
[\Delta_{X/S}]^j_{X\times X} \in H^0_{\T}(X \times_S X,p_j^{-1}\twist{T_{X/S}})
$$ 
\end{ex}

\begin{rem}
The fundamental classes defined above are virtual in the sense that
they live in a group twisted by a virtual vector bundle.
 For regular closed immersions $i:Z \rightarrow X$,
 the twisting virtual bundle will be of non-negative rank
 and $\eta_i$ corresponds to the usual fundamental class of $Z$ in $X$.
 On the contrary, for a smooth proper morphism $f:X \rightarrow S$, the twisting virtual 
 bundle will be of non-positive rank. In fact, $\eta_f$ is rather the analog 
 of cobordism classes (see \cite[Def. 2.1.6]{MorLNM}).
 For an extension of the above fundamental classes to derived stacks,
 we refer the reader to \cite{KhanVirt}.
\end{rem}

\begin{num}
\label{num:h-smoothness}
\textit{Homotopical smoothness and purity.} 
\begin{df}
\label{df:hsmooth}
(See also \cite[Definition 4.3.7]{DJK}). Let $f:X \rightarrow S$ be a smoothable lci morphism with virtual tangent bundle $\tau_f$. We say that $f$ is \emph{homotopically smooth} (\emph{h-smooth}) with respect to the motivic $\infty$-category 
$\T$ if the natural transformation $$\mathfrak p_f(-):\Th(\tau_f) \otimes f^* \rightarrow f^!$$ (see \eqref{eq:fdl_class}) evaluated at the sphere spectrum $\un_S$ is an isomorphism $\mathfrak p_f:\Th(\tau_f) \rightarrow f^!(\un_S)$.
\end{df}

\begin{num}\label{num:hsmooth}
One gets the following basic properties of h-smoothness:
considering composable lci smoothable morphisms $f$, $g$, $h=f \circ g$ (which is also lci smoothable),
if $f$ and $g$ (resp.~$f$ and $h$) are h-smooth, 
then so is $h$ (resp.~$g$). 
Moreover, 
if $g^!$ is conservative, $g$ and $h$ being h-smooth implies that $f$ is h-smooth. 
On the other hand, h-smoothness is not stable under base change.
\end{num}

\begin{ex}
Here are some examples of h-smooth maps $f:X \rightarrow S$.
\begin{itemize}
\item $f$ is smooth
\item $X$, $S$ are smooth over some base $B$ and $f$ is a morphism of $B$-schemes
\item $X$, $S$ are regular over a field $k$ and $\T$ is continuous, see \cite[Appendix A]{DFJK}
(all our examples are continuous in this sense)
\item (Absolute purity) $X$ and $S$ are regular and $\T=\SH_\QQ, \DM_\QQ, \DM_\et, D(-_\et,\ZZ_\ell)$
\end{itemize}
\end{ex}

In particular, 
a closed immersion between smooth varieties over a field is h-smooth. On the other hand, 
not all regular closed immersions are h-smooth: 
\begin{ex}
 Consider the regular closed immersion $$i:Z=Z_1\cup_{\{o\}} Z_2\rightarrow X=\AA^{2}$$ 
of the union of coordinate axes $Z_j\simeq \AA^1$, $j=1,2$  
in the affine plane $\AA^{2}$ over a field $k$. We claim that $i$ is not h-smooth
 (see \Cref{cor:comput_i^!_smc} and \Cref{ex:comput_i^!_smc} for more context).

The normal bundle $N_{Z/X}$ is the trivial line bundle of rank $1$.
 Let $i_0:\{o\}\rightarrow Z$ be the induced closed immersion
 and note that the composite immersion $i \circ i_0:\{o\}\rightarrow X$ is h-smooth, with trivial normal bundle $N_{\{o\}/X}$ of rank $2$.
 Now apply cdh-descent to the canonically induced cdh-distinguished square of closed immersions 
$$
\xymatrix@=16pt{ \{o\} \ar[dr]^{i_0} \ar[r]^{i_{0,1}} \ar[d]_{i_{0,2}} & Z_1 \ar[d]^{i_1} \\ 
Z_2 \ar[r]^{i_2} & Z}
$$
We obtain the homotopy exact sequence 
$$
\un_Z\rightarrow i_{1*}\un_{Z_1}\oplus i_{2*}\un_{Z_2}\rightarrow i_{0*}\un_{\{o\}}
$$
Applying $i_0^!$ to this sequence and using the base change isomorphisms 
$i_0^!i_{j,*}(\un_{Z_j})\simeq i_{0,j}^!(\un_{Z_j})$ and the purity isomorphisms 
$i_{0,j}^!(\un_{Z_j})\simeq \Th_{\{o\}}(-N_{\{o\}/Z_j})\simeq \un_k(-1)[-2]$ for 
the h-smooth closed immersions $i_{0,j}:\{o\}\rightarrow Z_j$ we get the homotopy exact sequence
$$
i_0^!(\un_Z)\rightarrow \un_k(-1)[-2]\oplus \un_k(-1)[-2]\rightarrow \un_k
$$
The second map in the above sequence is given by a pair of elements in $\pi_{2r,r}(k)$ for some $r<0$.
Hence, it is trivial, 
and we obtain the isomorphism (see \Cref{cor:comput_i^!_smc} for a generalization)
$$
i_0^!(\un_Z)\simeq \un_k(-1)[-2]\oplus \un_k(-1)[-2]\oplus \un_k[-1]
$$
On the other hand, if $i$ was h-smooth, we would have $i^!(\un_X)\simeq \Th_Z(-N_{Z/X})$.
Hence, by applying $i_0^!$ and using \eqref{eq:exchange-iso} and the $\otimes$-invertibility of 
$\Th_{\{o\}}(i_0^{-1}N_{Z/X})$, we would obtain isomorphisms 
$$ 
i_0^!(\un_Z)\simeq 
i_0^!i^!(\un_X)\otimes \Th_{\{o\}}(i_0^{-1}N_{Z/X})\simeq 
\Th_{\{o\}}(-N_{\{o\}/X})\otimes \Th_{\{o\}}(i_0^{-1}N_{Z/X})\simeq 
\un_k(-1)[-2] 
$$
\end{ex}

The h-smoothness property allows one to compare the four different theories in 
Definition \ref{df:internal_theories} and generalizes the smooth case.
The following isomorphisms can be seen as (internal) duality isomorphism,
 extending the classical duality between homology and cohomology with compact support.
\begin{prop}
\label{ex:purity&comparison} 
Let $f:X \rightarrow S$ be an h-smooth morphism with virtual tangent bundle $\tau_f$. Then the purity isomorphism $\mathfrak{p}_f:\Th(\tau_f) \rightarrow f^!(\un_S)$ induces isomorphisms
\begin{align*}
\htp_S(X,v)=f_!\big(\Th(v) \otimes f^!(\un_S)\big) &\xrightarrow{\mathfrak p_f^{-1}} f_!\big(\Th(v) \otimes 
\Th(\tau_f)\big)=\cohtp_S^c(X,v+\tau_f) \\
\htp^c_S(X,v)=f_*\big(\Th(v) \otimes f^!(\un_S)\big) &\xrightarrow{\mathfrak p_f^{-1}} f_*\big(\Th(v) \otimes 
\Th(\tau_f)\big)=\cohtp_S(X,v+\tau_f)
\end{align*}
Moreover, 
these isomorphisms transform the natural functoriality (resp.~Gysin map) in the source to the Gysin map
(resp.~natural functoriality) on the target.
\end{prop}
The first statement is clear. The last one is a direct consequence of the definitions
 --- both the purity isomorphisms and the Gysin maps are obtained by multiplication
  by a fundamental class ---
 and from the ``associativity formula" for fundamental classes in \cite[Theorem 3.3.2]{DJK}.

of the purity isomorphism, as multiplication by

\subsection{Closed pairs}\label{sec:closed_pairs}

\begin{num}
\label{num:closed_pairs}
A \emph{closed $S$-pair} is a pair $(X,Z)$ consisting of a separated $S$-scheme $f:X\rightarrow Z$ 
and a closed subscheme $i:Z\hookrightarrow X$ of $X$. 
For such a pair, we denote by $j:X-Z\rightarrow X$ the complementary open immersion, 
so that we have a commutative diagram
\begin{equation}
\label{eq:loc-notation}    
\xymatrix@R=8pt@C=40pt{
Z\ar@{^(->}^i[r]\ar_p[rd] & X\ar^{f}[d] & X-Z\ar@{_(->}_/4pt/j[l]\ar^q[ld] \\
& S &
}
\end{equation}
According to \Cref{df:quotient}, one associates to such a closed $S$-pair the $\T$-spectrum $\htp_S(X/X-Z)$
 (resp. $\cohtp_S(X/X-Z)$) which corresponds to the homotopy (resp. cohomotopy) of $X$ with support in $Z$. \\
A morphism $(\Phi,\varphi):(Y,T) \rightarrow (X,Z)$ of closed $S$-pairs is a topologically cartesian commutative diagram
\begin{equation}
\label{equation:morphismclosedpair}
\xymatrix@=14pt{
T\ar[d]_{\varphi}\ar[r] & Y\ar[d]^{\Phi} \\
Z\ar[r] & X
}
\end{equation}
Here, the horizontal maps are closed immersions. 
Note that $\htp_S(X/X-Z)$ (resp. $\cohtp_S(X/X-Z)$ is covariantly (resp. contravariantly) functorial
for morphisms of closed $S$-pairs. 
A morphism of closed $S$-pairs $(\Phi,\varphi)$ is said to be cartesian if 
\eqref{equation:morphismclosedpair} is cartesian as a diagram of schemes. 
It is said to be Nisnevich-excisive (resp. $\cdh$-excisive) if \eqref{equation:morphismclosedpair} 
is Nisnevich-distinguished (resp. $\cdh$-distinguished) in the sense of \cite{zbMATH05678570}.
An excisive morphism of closed $S$-pairs induces an isomorphism in $\T(S)$. Indeed, 
this follows from Nisnevich excision,  which is implied by the localization property in \cite[3.3.4]{CD3}.
\end{num}
\begin{df}
\label{df:closedSpairs}
A closed $S$-pair $(X,Z)$ is \emph{weakly smooth (resp.~weakly h-smooth)} 
if there exists a Nisnevich neighborhood $V$ of $Z$ in $X$ such that $V$ and $Z$ are smooth 
(resp.~h-smooth, see \Cref{df:hsmooth}) over $S$.
\end{df}

We note that for closed $S$-pairs as in \Cref{df:closedSpairs}, 
the closed immersion $i:Z\rightarrow X$ is necessarily regular with normal bundle $N_{Z/X}$.
\end{num}

\begin{num}
\label{num:res-cot-cplx}
Suppose $(X,Z)$ is a closed $S$-pair with the property that $X$ is h-smooth over $S$ in some 
Nisnevich neighborhood of its closed subscheme $Z$. 
Then, although the cotangent complex $\cL_{X/S}$ might not be a perfect complex on $X$, 
by assumption,  
it restricts to a perfect complex on a suitable Nisnevich neighborhood of $Z$ in $X$.
Thus, one can canonically define $i^{-1} \tau_{X/S}$ as a virtual vector bundle on $Z$ 
(by choosing an appropriate Nisnevich neighborhood and showing that it is independent of the choice).
\end{num}
We extend the Morel-Voevodsky homotopy purity theorem as follows, 
see also \Cref{thm:generalizedhomotopypurity2} for a refinement when $Z$ has smooth crossing singularities.

\begin{thm}
\label{thm:generalizedhomotopypurity}
Let $(X,Z)$ be a closed $S$-pair and let $v$ be virtual vector bundle on $X$. Then the following hold:
\begin{enumerate}
    \item If $X$ is h-smooth over $S$ in a Nisnevich neighborhood of $Z$, then there are  canonical purity isomorphisms 
\begin{equation}
\label{eq:purity-1}
\begin{array}{rcl}
    \htp_S(X/X-Z,v) & \simeq & \cohtp_S^c(Z,i^{-1}v+i^{-1}\tau_{X/S}) \\ 
    \cohtp_S(X/X-Z,v) & \simeq &\htp^c_S(Z, i^{-1}v-i^{-1}\tau_{X/S})
\end{array}
\end{equation}
    \item If moreover $(X,Z)$ is weakly h-smooth, then there are canonical purity isomorphisms \begin{equation}
\label{eq:purity}
\begin{array}{rcl}
\htp_S(X/X-Z,v) & \simeq & \htp_S(Z,i^{-1}v+\twist{N_{Z/X}}) \\
\cohtp_S(X/X-Z,v) & \simeq & \cohtp_S(Z, i^{-1}v-\twist{N_{Z/X}}) 
\end{array}
\end{equation}
\end{enumerate}
\end{thm}

\begin{proof}
By Nisnevich excision for closed $S$-pairs, we are reduced to the case where $f:X\to S$ is h-smooth, with virtual tangent bundle $\tau_f$. The fact that the two isomorphisms do not depend on the choice of a Nisnevich neighborhood follows by 
the functoriality of the excision isomorphism. With the notation \eqref{eq:loc-notation}, by inserting $\Th(v) \otimes f^!(\un_S)$ in the localization
exact homotopy  sequence  
$$
j_!j^! \rightarrow Id \rightarrow i_*i^*
$$
and applying $f_!$
 we get the exact homotopy 
$$
\Pi_S(X-Z,j^{-1}v) \rightarrow \Pi_S(X,v) \rightarrow p_!\big(\Th(i^{-1}v) \otimes i^*f^!(\un_S)\big)
$$
Here 
we used the identifications
\begin{align*}
f_!j_!j^!\big(\Th(v) \otimes f^!(\un_S)\big) &\simeq 
q_!\big(\Th(j^{-1}v) \otimes q^!(\un_S)\big)=\Pi_S(X-Z,j^{-1}v) \\
f_!i_*i^*\big(\Th(v) \otimes f^!(\un_S)\big) & \simeq 
f_!i_!\big(\Th(i^{-1}v) \otimes i^*f^!(\un_S)\big)=p_!\big(\Th(i^{-1}v) \otimes i^*f^!(\un_S)\big)
\end{align*}
In particular, 
there is an isomorphism
$$
\htp_S(X/X-Z,v) \simeq p_!\big(\Th(i^{-1}v) \otimes i^*f^!(\un_S)\big)
$$
The purity isomorphism then yields the desired isomorphism
\begin{align*}
\htp_S(X/X-Z,v) \simeq p_!\big(\Th(i^{-1}v) \otimes i^*f^!(\un_S)\big) & \xrightarrow{\mathfrak p_f^{-1}} p_!\big(\Th(i^{-1}v) \otimes i^*(\Th(\tau_f) \otimes f^*(\un_S))\big) \\
& = p_!\big(\Th(i^{-1}v+i^{-1}\tau_f)\big)=\cohtp_S^c(Z,i^{-1}v+i^{-1}\tau_f) 
\end{align*}
In the case where $Z/S$ is h-smooth, with virtual tangent bundle $\tau_p$, the  purity isomorphism $\mathfrak{p}_p$ in  \Cref{ex:purity&comparison}  yields in turn an isomorphism 
$$\cohtp_S^c(Z,i^{-1} v+i^{-1}\tau_f)\cong \htp_S(Z,i^{-1}v+i^{-1}\tau_f-\tau_p)= \htp_S(Z,i^{-1}v+\twist{N_{Z/X}})$$
The second isomorphism in \Cref{thm:generalizedhomotopypurity} 
is now a direct consequence of the h-smoothness property of $Z/S$.

The dual statements for $\cohtp_S(X/X-Z,v)$ follow from similar arguments applied to the dual localization homotopy exact sequence 
$
i_!i^! \rightarrow Id \rightarrow j_*j^*
$.
\end{proof}

\begin{rem}
One should be cautious about the functoriality of the purity isomorphisms concerning arbitrary morphisms of h-smooth closed $S$-pairs, as it does not hold true in the naive sense unless a transversality assumption is added (as indicated in \cite[3.2.9(i)]{DJK}). For more general statements regarding motives, we refer interested readers to \cite[\textsection 2.4]{Deg5}. Additionally, we will introduce a method for establishing basic functoriality for some related Gysin morphisms.
\end{rem}

\subsection{Computations of weak duals} \label{subsec:weak-duals}
\begin{num}
Recall \cite[5.2]{Deg8} that an object $M$ of a monoidal category with unit $\un$ 
is said to be \emph{rigid}
(or \emph{strongly dualizable}) with dual $M^\vee$ if there exists pairing and co-pairing maps
$$
\mu:M \otimes M^\vee \rightarrow \un, \epsilon:\un \rightarrow M^\vee \otimes M
$$
satisfying relations that express the functors $M \otimes -$ and $- \otimes M^\vee$
as both left and right adjoints. 
In a general symmetric monoidal category, 
if an object $M$ is rigid, 
then $\uHom(M,\un)$ is a (strong) dual of $M$, 
and the duality pairing is given by the evaluation map $M \otimes \uHom(M,\un) \rightarrow \un$.  
This justifies the terminology \emph{weak dual} of $M$ for the object $\uHom(M,\un)$. 
Next, we highlight some weaker results which will be useful in the remaining. 
We first pin down a notion that appears to be missing in previous works on the six functors formalism.
\end{num}

\begin{df}
A separated morphism $f:X \rightarrow S$ is called \emph{pre-$\T$-dualizing} 
if the map
\begin{equation}
\label{eq:pre-dualizing}
\un_X \rightarrow \uHom\big(f^!(\un_S),f^!(\un_S)\big)
\end{equation}
obtained by adjunction from the identity of $f^!(\un_S)$ is an isomorphism in $\T(X)$.
\end{df}

\begin{ex} According to \Cref{df:hsmooth}, any h-smooth morphism is pre-dualizing.
\end{ex}
\begin{rem}
\label{rem:pre-dualizing}
The notion of a pre-dualizing morphism is closely linked with Grothendieck-Verdier duality, 
as shown in \cite[4.4.11]{CD3}.
In fact, 
if $f^!(\un_S)$ is a dualizing object (\cite[Definition 4.4.4]{CD3}), 
then $f$ is pre-dualizing. 
Thus, 
it follows from \cite{Ayoub} that $f$ is pre-$\SH$-dualizing as soon as its target is smooth 
over a field of characteristic $0$.
In many cases, 
if the target of $f$ is regular, then $f$ is pre-dualizing:
see \cite{Gabber} for $D(-_\et,\ZZ_\ell)$,
\cite{CD3} for $\DM$, 
\cite{CD4} for $\DM_\et$, 
and \cite{DFJK} for $\SH_\QQ$. 
\end{rem}

The following proposition provides formulas for some weak duals, hence for potential strong duals when they exist. 

\begin{prop}\label{prop:weak-dual}
Let $f:X \rightarrow S$ be a separated $S$-scheme and let $v$ be a virtual vector bundle over $X$. Then the following hold:
\begin{enumerate}
\item  There exists a canonical isomorphism
$$
\uHom(\cohtp_S^c(X,v),\un_S) \xrightarrow \simeq \htp_S^c(X,-v)
$$
which is functorial in $X$, 
for both the natural functoriality for proper maps 
(\ref{num:natural_fct}) and for the Gysin morphisms for smoothable lci morphisms (\ref{num:Gysin}).

\item If, moreover, $f$ is pre-dualizing, then there exists an isomorphism
$$
\uHom(\htp_S(X,v),\un_S) \xrightarrow \simeq \cohtp_S(X,-v)
$$
which is again functorial for the natural functorialities and Gysin maps.

\item If moreover $f$ is h-smooth, with virtual tangent bundle $\tau_f$, then the purity isomorphism $\mathfrak{p}_f$ induces canonical isomorphisms
$$
\begin{array}{rrcl}
 \qquad &
\uHom\big(\htp_S(X,v),\un_S\big) \simeq \htp_S^c(X,-v-\tau_f) & \textrm{and} & \uHom\big(\cohtp^c_S(X,v),\un_S\big) \simeq \cohtp_S(X,-v+\tau_f)
\end{array}
$$
which are natural with respect to the natural functorialities and the Gysin maps, both restricted to proper morphisms.
\end{enumerate}
\end{prop}
\begin{proof}
To prove the isomorphism in (1) we use
\begin{align*}
\uHom(\cohtp_S^c(X,v),\un_S)=\uHom(f_!(\Th(v)),\un_S) & \xrightarrow{(a)} f_*\uHom\big(\Th(v),f^!(\un_S)\big) \\
& \stackrel{(b)} \simeq f_*\big(\Th(-v) \otimes f^!(\un_S)\big)=\htp_S^c(X,-v)
\end{align*}
Here, 
(a) (resp.~(b)) follows from the internal interpretation of the fact that $f^!$ is right adjoint to $f_!$ 
(resp. that $\Th(v)$ is $\otimes$-invertible).

To deduce (2), we consider the isomorphisms 
\begin{align*}
\uHom(\htp_S(X,v),\un_S)=\uHom\big(f_!\big(\Th(v) \otimes f^!(\un_S)\big),\un_S\big) 
& \xrightarrow{(a)} f_*\uHom\big(\Th(v) \otimes f^!(\un_S),f^!(\un_S)\big) \\
& \stackrel{(b)} \simeq f_*\Big(\Th(-v) \otimes \uHom\big(f^!(\un_S),f^!(\un_S)\big)\Big) \\
& \stackrel{(c)} \simeq f_*\big(\Th(-v) \otimes \un_X\big)=\cohtp_S(X,-v)
\end{align*}
Here, 
(a) and (b) are justified as before in (1), and (c) follows from the assumption that $f$ is pre-dualizing.
The isomorphisms in (3) are a combination of (1) and (2), 
and the isomorphisms of \Cref{ex:purity&comparison}.

Each functoriality statement is clear by construction. 
\end{proof}

\begin{ex} \label{ex:wts-dual}Here are known examples to which \Cref{prop:weak-dual} applies to give formulas for strong duals: 
\begin{enumerate}
\item For $\T=\SH(k)$, where $k$ is a field of characteristic $0$, according to \cite[Theorem 1.4]{RiouDual} any constructible spectrum is rigid. It follows from \cite{Ayoub2} that the six operations preserve constructibility for morphisms of $k$-schemes finite type.

(a) In particular, 
$\cohtp^c_k(X,v)$ and $\htp_k^c(X,v)$ are both rigid, and the point (1) above shows that $\htp_k^c(X,v)$ is dual to $\cohtp^c_k(X,-v)$ (and reciprocally).

(b) Similarly, $\htp_k(X,v)$ and $\cohtp_k(X,v)$ are constructible, and thus rigid.
As \Cref{rem:pre-dualizing} shows that $X/k$ is pre-dualizing, point (2) of the above proposition shows that  $\htp_k(X,v)$ is dual to $\cohtp_k(X,-v)$. See \Cref{thm:strong_duality} for a generalization.

(c) Finally, if $X$ is smooth, point (3) shows that $\htp_k(X,v)$ is dual to $\htp^c_k(X,-v-\twist{T_{X/k}})$,
which is the expected generalization of Poincar\'e duality. 
This result will be extended in
\Cref{thm:strong_duality-pair}.
\item Using \cite[Theorem 2.4.9]{BD1} (see also \cite[Theorem 5.8]{mhskpao}), 
the same results hold in $\SH(k)[1/p]$ if $k$ has positive characteristic $p$.
\end{enumerate}
\end{ex}

The situation is more complicated over a base scheme $S$ of positive dimension.
When $X/S$ is smooth and proper, \Cref{ex:duality} shows that 
$\Pi_S(X,v)=\Pi_S^c(X,v)$ is rigid for any virtual vector bundle $v$. 
Theorems \ref{thm:strong_duality}, \ref{thm:generalizedhomotopypurity2} and \Cref{cor:complementdivisor_duality} below give several new examples of rigid relative spectra and motives. 
In general, 
neither properness nor smoothness alone ensures rigidness, see \Cref{ex:nonrigodness}.

\begin{ex}\label{ex:duality}
\textit{Poincar\'e duality} (see \cite[5.4]{Deg8}). Let $f:X \rightarrow S$ be a smooth proper $S$-scheme with tangent bundle $T$. Then, for any virtual bundle $v$ over $X$,
$\htp_S(X,v)$ is rigid with dual $$\htp_S(X,-\twist {T}-v)=\Th_S(-v-\twist {T})$$
Note that the given expression of the dual corresponds to that in \Cref{prop:weak-dual}(2) via the purity isomorphism $\htp_S(X,-v-\twist{T})\simeq \cohtp^c_S(X,-v)=\cohtp_S(X,-v)$ of \Cref{ex:purity&comparison}. 

Indeed, letting  $\delta:X \rightarrow X \times_S X$ be the diagonal closed immersion,
the pairing and co-pairing maps are given by the composite maps
\begin{align*}
& \htp_S(X,v) \otimes \htp_S(X,-v-\twist {T})
\stackrel{(*)} \simeq \htp_S(X \times_S X,-\twist{p_1^{-1}T_f})
\xrightarrow{\delta^!} \htp_S(X) \xrightarrow{f_*} \un_S \\
& \un_S \xrightarrow{f^!} \htp_S(X,-\twist {T})
\xrightarrow{\delta_*}
\htp_S(X \times_S X,-\twist{p_1^{-1}T})
\stackrel{(*)} \simeq \htp_S(X,-v-\twist {T}) \otimes \htp_S(X,v)
\end{align*}
Here the labels $(*)$'s are instances of the K\"unneth isomorphism \eqref{eq:kunneth2} given in the next subsection. 
The required identities follow from the base change formula for Gysin morphisms in \cite[3.3.2(iii)]{DJK}.
\end{ex}

\begin{ex}
\label{ex:nonrigodness}
Let $i:Z \rightarrow S$ be a h-smooth closed immersion (e.g., $Z$ and $S$ are smooth over a field $k$) 
with nonempty open complement $j:U\rightarrow S$. 
We claim that $\htp_S(U)=j_!(\un_U)$ is not rigid. 
Indeed, 
assuming the contrary, 
according to \Cref{prop:weak-dual} its dual would be isomorphic to $j_*(\un_U)$. 
Since $i^*$ is monoidal, 
it would follow that $i^*j_!(\un_U)$ is rigid with dual $i^*j_*(\un_U)$. 
The first spectrum is trivial, 
whereas purity identifies the second one with an extension of $\un_Z$ by $\Th(N_{Z/S})$, 
which is thus necessarily a nontrivial spectrum.
An identical (dual) argument shows that $\htp_S(Z)$ is not rigid. 

In a similar vein, 
\cite[Remark 8.2]{motiviclandweber} gives the following:
let $S=\Spec(R)$ be the spectrum of a discrete valuation $R$ with quotient field $K$.  
Then  $\Pi_S(\Spec(K))$ is not rigid in $\SH(S)$.
\end{ex}

\subsection{K\"unneth isomorphisms}

We collect here several variants of K\"unneth formulas (see also \Cref{prop:Kunneth-smooth-nc}).

\begin{ex}\textit{K{\"u}nneth isomorphisms}.\label{ex:kunneth}
Let $X$, $Y$ be separated $S$-schemes and $v$, $w$ be virtual vector bundles over $X$, $Y$, respectively.
Then, one deduces from the projection and base change formulas a canonical isomorphism
(obtained from exchange isomorphisms, see \cite{CD3})
\begin{equation}
\label{eq:kunneth1}
\cohtp^c_S(X,v) \otimes \cohtp^c_S(Y,w) \simeq \cohtp^c_S(X \times_S Y,p_1^{-1}v+p_2^{-1}w)
\end{equation}
If $X$ and $Y$ are in addition smooth over $S$, 
then we have the more usual K{\"u}nneth formula (see \cite[1.1.37]{CD3})
\begin{equation}
\label{eq:kunneth2}
\htp_S(X,v) \otimes \htp_S(Y,w) \simeq \htp_S(X \times_S Y,p_1^{-1}v+p_2^{-1}w)
\end{equation}
Using the relative purity isomorphism, 
one can also deduce \eqref{eq:kunneth2} from the previous one.
\Cref{ex:non-kunneth} shows the second K{\"u}nneth formula \eqref{eq:kunneth2} fails in the non-smooth case.
\end{ex}

\begin{ex}\label{ex:non-kunneth}
One can extend the K\"unneth formula \eqref{eq:kunneth2} to the non smooth case
 (see below for example) but one still needs assumptions.
 Indeed, one cannot replace in general smoothness by h-smoothness.
 For example, for the zero section $s:X\rightarrow \AA^n_X=S$, $n\geq 1$, one has 
  $\Pi_S(X)=s_*(\un_X)(n)[2n]$ and 
$$
\Pi_S(X) \otimes_S \Pi_S(X)=s_*(\un_X)(n)[2n] \otimes s_*(\un_X)(n)[2n]=s_*(\un_X)(2n)[4n]
$$
The latter is different from $\Pi_S(X \times_S X)=\Pi_S(X)$
 (in any of our motivic $\infty$-categories).
\end{ex}

\begin{num}
In the following result, 
we give some new cases of K{\"u}nneth formulas to compute stable homotopy at infinity
 (see Propositions \ref{cor:Kunneth-infty} and \ref{cor:Kunneth-A1-Cont}).
To a cartesian square of separated morphisms
$$
\xymatrix@=16pt{
X \times_S Y\ar_q[d]\ar|h[rd]\ar^-p[r] & Y\ar^g[d] \\
X\ar_f[r] & S
}
$$
we associate the following commutative diagram of exchange transformations and the map $\alpha_?$ 
forgetting proper support
\begin{equation}
\label{equation:etfps}
\begin{split}
\xymatrix@R=10pt@C=34pt{
f_!f^!(\un) \otimes g_!g^!(\un)\ar_\sim[d]\ar^{\alpha_f \otimes \alpha_g}[r] & f_*f^!(\un) \otimes g_*g^!(\un)\ar^{(1)}[d] \\
g_!(g^*f_!f^!(\un) \otimes g^!(\un))\ar_\sim[d]\ar^{\alpha_g(\alpha_f)}[r] & g_*(g^*f_*f^!(\un) \otimes g^!(\un))\ar^{(2)}[d] \\
g_!(p_!q^*f^!(\un) \otimes g^!(\un))\ar_\sim[d]\ar^{\alpha_g(\alpha_p)}[r] & g_*(p_*q^*f^!(\un) \otimes g^!(\un))\ar^{(3)}[d] \\
h_!(q^*f^!(\un) \otimes p^*g^!(\un))\ar_{(4)}[d]\ar^{\alpha_h}[r] & h_*(q^*f^!(\un) \otimes p^*g^!(\un))\ar^{(4)}[d] \\
h_!(h^!(\un))\ar^{\alpha_h}[r] & h_*(h^!(\un)) \\
}
\end{split}
\end{equation}
Here, 
$\alpha_r$ denotes any map induced by the natural transformation $r_! \rightarrow r_*$.
\end{num}
\begin{thm}\label{th:Kunneth-infty}
With the above notation, assume that one of the following conditions is satisfied:
\begin{enumerate}[label=\roman*)]
\item $Y$ is smooth and proper over $S$.
\item $S$ is the spectrum of a field $k$ of characteristic exponent $p$ and either $\T$ is $\ZZ[1/p]$-linear 
or receives a realization functor from $\DM_\et$ as in \eqref{eq:motivic_cat}.
\item $Y$ is smooth and stably $\AA^1$-contractible over $S$
with stably constant tangent bundle $T_g$ (see \Cref{def:stably-A1-cont}). 
\end{enumerate}
Then all the vertical maps in \eqref{equation:etfps} are isomorphisms, 
and there is an induced commutative diagram
$$
\xymatrix@R=10pt@C=44pt{
\htp_S(X) \otimes \htp_S(Y)\ar^{\alpha_X \otimes \alpha_Y}[r]\ar_\sim[d] & \htp^c_S(X) \otimes \htp^c_S(Y)\ar^\sim[d] \\
\htp_S(X \times_S Y)\ar^{\alpha_{XY}}[r] & \htp^c_S(X \times_S Y)
}
$$
\end{thm}
\begin{proof}
In each case, we have to prove that the morphisms (1) to (4) in \eqref{equation:etfps} are isomorphisms. 
Case i) is transparent.
Next, we consider Case ii). 
If $\T$ is $\ZZ[1/p]$-linear then
all the isomorphisms follow from \cite[Theorem 2.4.6]{JY1} with $Y_1=Y_2=S$, $X_1=X$, $X=Y$.
More precisely, 
the composite of (1), (2), and (3) is an isomorphism due to point (2) of 2.4.6, and (4) is an isomorphism by (3) of 2.4.6.
If $\T$ receives a functor from $\DM_\et$, one can reduce to the latter case by appealing to
\cite[Sec. 3.1]{CisInter}.

It remains to prove the assertion in Case iii). 
The isomorphism (4) follows from the fact that $g$ (resp.~$q$) is smooth with tangent bundle $T_g$ (resp.~$T_q=p^*T_g$), 
and from the relative purity isomorphism
$$
q^*f^!(\un_S) \otimes p^*g^!(\un_S) 
\simeq 
q^*f^!(\un_S) \otimes p^*\Th(T_g) 
\simeq 
q^!f^!(\un_S)=h^!(\un_S)
$$

Using \Cref{lm:stably_A1-contract_virtual_vb} 
applied respectively to $q$ and $g$, one deduces
$$
h_*h^!(\un_S)=f_*q_*q^!f^!(\un_S) 
\simeq f_* \Th(f^*v_0) \otimes f^!(\un_S)=\Th(v_0) \otimes f_*f^!(\un_S)
\simeq f_*f^!(\un_S) \otimes g_*g^!(\un_S) 
$$
where $v_0$ is the virtual vector bundle over $S$ such that $\twist{T_g}=g^*v_0$. 

It is now a formal, though lengthy, 
exercise to check that the preceding isomorphism is equal to the composition of the maps (1)-(4). 
\end{proof}

\subsection{Functorial Gysin morphisms}\label{sec:functorial-Gysin}

\begin{num}
We now show how to deduce $\infty$-functorial Gysin maps
 out of purity isomorphisms (in fact, duality) and from the $\infty$-categorical ``replacement lemma''
 of \Cref{prop:replacement}.\footnote{We thank Robin Carlier for explaining this trick.}

Let us fix a base scheme $S$ and a virtual bundle $v$ on $S$.
 We will denote by $\hSm_S$ (resp. $\hSmp$ the category of $h$-smooth $S$-schemes (\Cref{df:hsmooth}),
  with arbitrary $S$-morphisms (resp. with proper $S$-morphisms).
 Given a scheme $X$ in $\hSm_S$, with structural morphism $f:X \rightarrow S$, we will denote
 by $\tau_X=\tau_f$ the virtual tangent bundle associated with $f$, and by $v_X$ the pullback
 of $v$ to $X$.
\end{num}
\begin{prop}\label{prop:enhanced-Gysin}
There exists $\infty$-functors
\begin{align*}
\htp^!_S:(\hSmp_S)^{op} & \rightarrow \T(S), X/S \mapsto \htp_S(X,v_X-\tau_X) \\
\cohtp_{S!}:\hSmp_S & \rightarrow \T(S), X/S \mapsto \cohtp_S(X,v_X+\tau_X) \\
\htp^{c!}_S:(\hSm_S)^{op} & \rightarrow \T(S), X/S \mapsto \htp^c_S(X,v_X-\tau_X) \\
\cohtp^c_{S!}:\hSm_S & \rightarrow \T(S), X/S \mapsto \cohtp^c_S(X,v_X+\tau_X)
\end{align*}
together with natural isomorphisms of $\infty$-functors:
\begin{align*}
\cohtp_S^c \xrightarrow \sim \htp^!_S, \qquad &  \htp_S^c \xrightarrow \sim \cohtp_{S!}, \\
\cohtp_S \xrightarrow \sim \htp^{c!}_S, \qquad &  \htp_S \xrightarrow \sim \cohtp^c_{S!},
\end{align*}
which at the level of a $1$-morphism $f:Y \rightarrow X$ is given by the commutative diagrams
$$
\xymatrix@R=14pt@C=38pt{
\cohtp^c_S(X,v_X)\ar_-{f^*}[d]\ar^-{\mathfrak p_{X/S}}[r] & \htp_S(X,v_X-\tau_X)\ar^{\htp^!_{S}(f)=f^!}[d]
&
\htp^c_S(X,v_X)\ar_-{f_*}[d]\ar^-{\mathfrak p^{-1}_{X/S}}[r] & \cohtp_S(X,v_X+\tau_X)\ar^{\htp_{S!}(f)=f_!}[d]
\\
\cohtp^c_S(Y,v_Y)\ar^-{\mathfrak p_{Y/S}}[r] & \htp_S(Y,v_Y-\tau_Y)
&
\htp^c_S(Y,v_Y)\ar^-{\mathfrak p^{-1}_{Y/S}}[r] & \cohtp_S(Y,v_Y+\tau_Y) \\
\cohtp_S(X,v_X)\ar_-{f^*}[d]\ar^-{\mathfrak p^{-1}_{X/S}}[r] & \htp^c_S(X,v_X-\tau_X)\ar^{\htp^!_{S}(f)=f^!}[d]
&
\htp_S(X,v_X)\ar_-{f_*}[d]\ar^-{\mathfrak p_{X/S}}[r] & \cohtp^c_S(X,v_X+\tau_X)\ar^{\htp_{S!}(f)=f_!}[d]
\\
\cohtp_S(Y,v_Y)\ar^-{\mathfrak p^{-1}_{Y/S}}[r] & \htp^c_S(Y,v_Y-\tau_Y)
&
\htp_S(Y,v_Y)\ar^-{\mathfrak p_{Y/S}}[r] & \cohtp^c_S(Y,v_Y+\tau_Y)
}
$$
where $f$ is proper in the first two diagrams,
 and $\mathfrak p_X$, $\mathfrak p_Y$ are induced by the purity isomorphisms (\Cref{eq:fdl_class}) of $X/S$, $Y/S$ respectively.
\end{prop}
\begin{proof}
Each case follows by applying \Cref{prop:replacement} respectively to the functors $\cohtp^c_S$,
 $\htp^c_S$ (restricted to $\hSmp$), $\cohtp_S$, $\htp_S$ (restricted to $\hSm$) and to the purity
 isomorphisms of \Cref{ex:purity&comparison}.
\end{proof}

\begin{rem}
Note that we can identify the Gysin map $\htp_S^!(f)$ obtained from the above proposition
 with the Gysin map of \Cref{num:Gysin}. Indeed, due to the last statement of \Cref{ex:purity&comparison},
 both maps are homotopy equivalent.
\end{rem}

%% file: normal-rev.tex
\subsection{Ordered Cech semi-simplicial scheme associated to a closed cover} 

\begin{num}
\label{num:ordeded_Cech}
\label{num:degeneracy}
Let $X$ be a noetherian scheme and consider a finite closed cover
of $X$,  i.e., a surjective map 
$$
p:X_\bullet=\sqcup_{i \in I} X_i \rightarrow X
$$
obtained from a finite collection of closed immersions $\nu_i:X_i \rightarrow X$, $i\in I$.
We let $\cap=\times_X$ be a shorthand for the fiber product of closed $X$-schemes. For every nonempty subset $J \subset I$ we set $X_J=\cap_{j \in J} X_j$ and denote by $\nu_J:X_J\rightarrow X$ the canonically induced closed immersion. For every pair of nonempty subsets $J \subset K$ of $I$, we let $\nu_K^J:X_K \rightarrow X_J$ be the canonically induced closed immersion so that we have $\nu_K=\nu_J\circ \nu_K^J$.

The \v Cech simplicial $X$-scheme $\cech_*(X_\bullet/X)$ associated with $p$ takes the form
\begin{equation}
\label{equation:cechsimplicial}
\cech_n(X_\bullet/X):= \bigsqcup_{(i_0,\hdots,i_n) \in I^{n+1}} X_{i_0} \cap \hdots \cap X_{i_n}
\end{equation}
with degeneracy morphisms $\delta^k_n:\cech_n(X_\bullet) \rightarrow \cech_{n-1}(X_\bullet)$, $k=0,\ldots, n$,
given by the sum of the canonical immersions
$$
X_{i_0}\cap \cdots \cap X_{i_k}\cap \cdots \cap X_{i_n} \to  
X_{i_0}\cap \cdots \cap \widehat{X_{i_k}}\cap \cdots \cap X_{i_n}
$$ 

The choice of a total ordering on $I$ induces a natural bijection between the set of subsets $J\subset I$ of cardinality 
$\sharp J=n+1$ and the set of $(n+1)$-tuples $(i_0,\ldots, i_n)\in I^{n+1}$ given by mapping a subset $J$ to the unique $(n+1)$-tuple 
$(i_0,\ldots, i_n)\in I^{n+1}$ such that $J=\{i_0,\ldots,i_n\}$ and $i_0<\cdots <i_n$.
In the following we fix such a total ordering and we set 
\begin{equation}
\label{equation:cechsemisimplicial}
\cecho_n(X_\bullet/X)
:=\bigsqcup_{\substack{(i_0,\ldots,i_n)\in I^{n+1} \\ i_0<\cdots <i_n}} X_{i_0}\cap \cdots \cap X_{i_n}=\bigsqcup_{J\subset I,\,\sharp J=n+1} X_J
\end{equation}
There is a canonical embedding $\cecho_*(X_\bullet/X) \subset \cech_*(X_\bullet/X)$ of $\NN$-graded $Z$-schemes given 
in degree $n$ by mapping each $X_{j_0}\cap \cdots \cap X_{j_n}$ to itself via the identity. 
The degeneracy morphisms $\delta_n^k$ in the simplicial structure on $\cech_*(X_\bullet/X)$ preserve $\cecho_*(X_\bullet/X)$
and induce degeneracy morphisms $$\delta_n^k=\bigsqcup_{J=\{i_0<\hdots<\widehat{i_k}<\hdots<i_n\} \subset K=\{i_0<\hdots<i_n\}  } \nu_K^J:\cecho_n(X_\bullet/X) \to \cecho_{n-1}(X_\bullet/X)$$
endowing $\cecho_*(X_\bullet/X)$ with the structure of a semi-simplicial 
$X$-scheme\footnote{Recall that a semi-simplicial object in a category $\mathscr C$ is a contravariant functor from 
$\Dinj\to \mathscr C$,  where $\Dinj$ denotes the category of finite ordered sets with injective maps as morphisms.}. 
We refer to the latter as the \emph{ordered \v Cech semi-simplicial $X$-scheme} associated to the finite closed cover $p:X_\bullet \rightarrow X$. 
\begin{rem}
By construction, the ordered \v Cech semi-simplicial scheme $\cecho_*(X_\bullet/X)$
is bounded by the cardinality $\sharp I$ of the index set $I$ in the sense that $\cecho_n(X_\bullet/X)=\varnothing$ for all $n>\sharp I$.
In particular, it is much smaller than $\cech_*(X_\bullet/X)$. 
\end{rem}
\end{num}

\subsection{Ordered hyperdescent for closed covers}
\begin{num}\label{num:6functors_extension}
We now use the $\infty$-categorical enhancement of the motivic category $\T$, 
and in particular the adjunction of $\infty$-functors $(f^*,f_*)$ and $(f_!,f^!)$.
Let us fix a base scheme $S$ and write $\Sch_S$ for the category of separated $S$-schemes. 
To any object $\E$ of $\T(S)$, we associate the covariant $\infty$-functor
$$
\htp_{S}(-;\E):\Sch_S \rightarrow \T(S), \; (f:X\rightarrow S) \mapsto f_!f^!(\E)
$$ 
and, dually, the contravariant $\infty$-functor
$$
\cohtp_S(-;\E):\Sch_S^{op} \rightarrow \T(S), \; (f:X\rightarrow S) \mapsto f_*f^*(\E)
$$
\end{num}

\begin{num}
Back to the  setup in \Cref{num:ordeded_Cech}, 
we assume in addition that $f:X\rightarrow S$ is a separated $S$-scheme.
For every nonempty subset $J \subset I$,  we let $f_J\colon X_J \rightarrow S$ be the composite of the closed immersion 
$\nu_J:X_J \rightarrow X$ with $f:Z\to S$.
To the ordered \v Cech semi-simplicial $X$-scheme $\cecho_n(X_\bullet/X)$ and any object $\E$ of $\T(S)$, we associate the functors
\begin{align*}
\left( (\Dinj)^{op} \rightarrow  \Sch_S \right) & \xrightarrow{\htp_{S}(-;\E)}  \T(S) \\
\left(\Dinj \rightarrow \Sch_S^{op} \right) & \xrightarrow{\cohtp_S(-;\E)}  \T(S)
\end{align*}
By using the augmentation map to $X$, 
we obtain canonical maps involving the limit and colimit of the preceding functors
\begin{align}
\label{eq:compute_closed_cover1}
\htp_{X_\bullet/X;\E}:&\colim_{n \in (\Dinj)^{op}} \left(\bigoplus_{J \subset I, \sharp J=n+1} \htp_{S}(X_J;\E)\right) 
\rightarrow \htp_{S}(X;\E) \\
\label{eq:compute_closed_cover2}
\cohtp_{X_\bullet/X;\E}:&\cohtp_S(X;\E) \rightarrow 
\underset{n \in \Dinj}\lim\left(\bigoplus_{J \subset I, \sharp J=n+1} \cohtp_S(X_J;\E)\right)
\end{align}
The next theorem interprets the colimit (resp.~limit) as the ``standard'' resolution of homology (resp.~cohomology) of 
$X/S$ with $\E$-coefficients.
\end{num}
\begin{thm}
\label{thm:compute_closed_cover}
For every finite closed cover $p:X_\bullet \rightarrow X$,  the maps 
 $\htp_{X_\bullet/X;\E}$ and $\cohtp_{X_\bullet/X;\E}$ are both isomorphisms in $\T(S)$.
\end{thm}
\begin{proof}
Using \Cref{ex:topological_invariance}, 
we can reduce to the case where $X$ and each $X_i$ are reduced.

Let us consider the case of $\htp_{X_\bullet/S;\E}$. 
For every nonempty subset $J \subset I$, there is an isomorphism $f_{J!}f_J^! \simeq f_!\nu_{J!}\nu_J^!f^!$.
So by replacing $\E$ with $f^!(\E)$, we are reduced to the case $S=X$. 
There is, see for example \cite[B.20]{DFJK}, a conservative family of functors
$$
i_z^!
\colon 
\T(X) \rightarrow \T\big(\Spec(\kappa(x))\big), x \in X
$$
Therefore, it suffices to show $i_x^!\big(\htp_{X_\bullet/X;\E}\big)$ is an isomorphism for all $x \in X$.
Given $J \subset I$, 
we consider the following cartesian square
$$
\xymatrix@=24pt{
X'_J\ar_{\nu'_J}[d]\ar^{i'_x}[r] & X_J\ar^{\nu_J}[d] \\
\{x\}\ar^{i_x}[r] & X
}
$$
By proper base change for the proper map $\nu_J$,
we have an isomorphism $i_x^!\nu_{J!}\nu_J^! \simeq \nu'_{J!}i_x^{\prime!}\nu_J^!$. 
Since, on the other hand, 
we have $\nu'_{J!}i_x^{\prime!}\nu_J^! \simeq \nu'_{J!}\nu_J^{\prime !}i_x^!$,
and because the pullback of the ordered \v Cech complex $\cecho_*(X_\bullet/X)$ along $\{x\} \rightarrow X$ 
corresponds to the ordered \v Cech complex $\cecho_*(X_\bullet \times_X \{x\}/\{x\})$,
we deduce the isomorphism
$$
i_x^!\big(\htp_{X_\bullet/X;\E}\big)
\simeq 
\htp_{X_\bullet \times_X \{x\}/\{x\};i_x^!\E}
$$
Since $X$ is reduced, we may therefore assume $X=\{x\}$ is the Zariski spectrum of a field.
In this case, 
the $X_i$'s are closed reduced subschemes of the reduced scheme $\{x\}$, 
and thus the closed cover $p':\sqcup_{i \in I} X'_i \rightarrow \{x\}$ is given by a sum of identity maps.
To conclude,  
one can then observe, for example, 
the existence of explicit homotopy contraction of the semi-simplicial augmented pointed $X$-scheme
$$
\cecho_*(X_\bullet/\{x\})_+ \rightarrow \{x\}_+
$$
The proof for the map $\cohtp_{X_\bullet/X;\E}$ is entirely analogous, using the conservative family of functors
$$
i_x^*:\T(X) \rightarrow \T\big(\Spec(\kappa(x))\big), x \in X
$$
of \cite[Proposition 4.3.17]{CD3}.
\end{proof}

\begin{rem}\label{rem:compute_closed_cover&reduction}
In formulas \eqref{eq:compute_closed_cover1} and \eqref{eq:compute_closed_cover2}, one can arbitrarily replace
 the closed subscheme $X_J$ of $X$ by its reduction according to \Cref{ex:topological_invariance}.
 In the followings, we will use that possibility without further warning.
\end{rem}

\begin{rem} 
\Cref{thm:compute_closed_cover} does not extend to arbitrary cdh-covers.
For instance, 
it does not work for the proper cdh-cover $\PP^1_k \rightarrow \Spec k$ for apparent reasons: 
for such a connected cover,
one needs the whole Cech complex to get a resolution of the point. 
Similarly, 
the ordered \v Cech complex associated with a nontrivial finite \'etale cover does not yield a resolution 
in the \'etale topology.
In the cdh-topology it is possible to generalize \Cref{thm:compute_closed_cover} by replacing closed covers 
$p:X_\bullet \rightarrow X$ by proper cdh-covers such that there exists a stratification of $X$ having the 
property that for every stratum $Y$, 
there exists a member of the covering family $X_i \rightarrow X$ for which 
$X_i \times_X Y \rightarrow Y$ is an isomorphism. 
The proof of \Cref{thm:compute_closed_cover} carries over to this setting by applying the proper base change theorem, 
and this generalization allows in particular to incorporate the elementary cdh-covers. 
A similar consideration applies to Nisnevich covers.
\end{rem}


\subsection{Schemes and subschemes with crossing singularities}
\label{sec:crossing_sing}

\begin{notation}
\label{num:normal_crossing}
\label{not:extended-nc}
Let $Z$ be a separated $S$-scheme with finitely many irreducible components $Z'_i$, $i\in I$. For every nonempty subset $J \subset I$, we let $Z'_J=(\cap_{j \in J} Z'_j)$, where $\cap=\times_X$,
 and $Z_J=(Z'_J)_{\mathrm{red}}$. We denote by $\nu_J$ the canonically induced closed immersion of $Z_J$ in $Z$. 
For every pair of nonempty subsets $J \subset K$ of $I$, we denote by 
$\nu_K^J:Z_K \rightarrow Z_J$
the naturally induced closed immersion. For a virtual vector bundle $v$ on $Z$ and a nonempty subset 
$J\subset I$, we let $v_J=\nu_J^{-1}v$. 

For a closed $S$-pair $(X,Z)$ corresponding to a closed subscheme $i:Z\rightarrow X$ 
with irreducible components $Z_i'$, $i\in I$, we extend the above notation by setting 
$$\bar{\nu}_J=i\circ \nu_J:Z_J\rightarrow Z \rightarrow X$$ 
For a virtual vector bundle $v$ on $X$, we let $v_J$ denote the pullback of $v$ to $Z_J$ by $\bar{\nu}_J$.

We fix the following terminology on normal crossing singularities in the rest of this paper. 
\end{notation}

\begin{df}
\label{df:normal_crossing}
With the notation above, 
we say that $Z$ has \emph{smooth} (resp.~\emph{regular}, \emph{h-smooth}) \emph{reduced crossing} over $S$ if, 
for any non-empty $J \subset I$, 
$Z_J$ is a smooth (resp.~regular, h-smooth) $S$-scheme.
\end{df}

With our conventions, 
the intersection of the irreducible components of $Z$ is allowed to have nontrivial multiplicity.
Note that h-smoothness is insensible to reduction; 
we will simply write \emph{h-smooth crossing}.

\begin{prop} 
\label{ex:sncexample1}
Let $Z/S$ be an h-smooth crossing scheme and let $v$ is a virtual vector bundle on $Z$. Then $\Pi_S(Z,v)$ is isomorphic to the colimit in the  underlying $\infty$-category of $\T(S)$ of the diagram   
\begin{equation}
\label{equation:snches2}
\htp_S(Z_I,v_I) 
\rightrightarrows
\bigoplus_{K \subset I, \sharp K=\sharp I-1} \htp_S(Z_K,v_K)\ 
\rightrightarrows\hdots
\bigoplus_{J \subset I, \sharp J=2} \htp_S(Z_J,v_J)
\rightrightarrows \bigoplus_{i \in I} \htp_S(Z_i,v_i)
\end{equation}
with degeneracy maps 
$$
(\delta_n^k)_*=
\sum_{J=\{i_0<\hdots<\widehat{i_k}<\hdots<i_n\} \subset K=\{i_0<\hdots<i_n\}  } (\nu_K^J)_*
$$
and with augmentation map 
$$\sum_{i\in I} \nu_{i*}:\bigoplus_{i \in I} \htp_S(Z_i,v_i) \rightarrow \htp_S(Z,v)$$ 

Dually, $\cohtp_S(Z,v)$ is isomorphic to the limit of the  diagram
\begin{equation}
\label{equation:cohom-limit}
\bigoplus_{i \in I} \cohtp_S(Z_i,v_i)
\rightrightarrows
\bigoplus_{J \subset I, \sharp J=2} 
\cohtp_S(Z_J,v_J)
\rightrightarrows \cdots
\bigoplus_{K \subset I, \sharp K=\sharp I-1}
 \cohtp_S(Z_K,v_K)
 \rightrightarrows
\cohtp_S(Z_I,v_I) 
\end{equation}
with co-degeneracy maps 
$$
(\delta_n^k)^*=\sum_{J=\{i_0<\hdots<\widehat{i_k}<\hdots<i_n\} \subset K=\{i_0<\hdots<i_n\}  } (\nu_K^J)^*
$$
and with co-augmentation map 
$$\sum_ {i\in I} \nu_i^*: \cohtp_S(Z,v)\rightarrow \bigoplus_{i \in I} \cohtp_S(Z_i,v_i)$$
\end{prop}
\begin{proof}
Consider the closed cover $Z_\bullet=\bigsqcup Z_i' \rightarrow Z$ of $Z$ by its irreducible components. Noting that by \Cref{ex:topological_invariance} we have, for every $J\subset I$, canonical isomorphisms $\htp_S(Z_J',v'_J)\simeq \htp_S(Z_J,v_J)$ and $\cohtp(Z_J',v'_J)\simeq \cohtp(Z_J,v_J)$, the assertion follows by appealing to \Cref{thm:compute_closed_cover} with $S=X=Z$, $X_\bullet=Z_\bullet$ and $\E=\Th(v)\otimes f^!(\un_S)$  (resp. $\E=\Th(v)$)  and then applying $f_!$ (resp. $f_*$) to the obtained resolution.
\end{proof}

%
%

\begin{ex}
\label{rem:compare}
 In the case $\T=\SH$, the $S$-scheme $Z$ in \Cref{ex:sncexample1} defines a sheaf of sets $\underline Z$ on $\Sm_{S}$. 
We claim the preceding computation yields an isomorphism $\htp_S(Z)\simeq \Sigma^{\infty}\underline Z_{+}$ in $\SH(S)$.
A proof uses the $\PP^1$-stable $\AA^1$-homotopy category $\underline \SH_\cdh(S)$ over $S$ for the big cdh site; 
i.e., 
the site of finite type $S$-schemes endowed with the $\cdh$-topology in the style of 
\cite[\textsection 6.1]{CD3}.
\Cref{thm:compute_closed_cover} holds in $\underline \SH_\cdh(S)$ due to $\cdh$-descent, 
so the comparison reduces to the smooth case, 
which holds by the general properties of an enlargement.
\end{ex}

Next, we show a K\"unneth formula for smooth crossings schemes.
\begin{prop}\label{prop:Kunneth-smooth-nc}
Suppose $Z$, $T$ are smooth crossings $S$-schemes, and $v$, $w$ are virtual bundles over $Z$ and $T$, respectively.
Then the canonical map \eqref{equation:etfps} is an isomorphism
$$
\htp_S(Z,w) \otimes \htp_S(T,w) \xrightarrow{\simeq} \htp_S(Z \times_S T,v \times_S w)
$$
\end{prop}
\begin{proof}
The case where $Z/S$ is smooth and $T/S$ is smooth crossing follows from \Cref{ex:sncexample1} and the fact
 $\otimes$ commutes with homotopy colimits (as a left adjoint).
 To treat the case where $Z/S$ has smooth crossings, we can therefore argue by induction on the number of irreducible
 components of $Z$. Let $Z'$ be an irreducible component of $Z$ and $Z''$ the union of the other irreducible components.
 The cdh-distinguished homotopy exact sequence associated with the cdh-cover $(Z',Z'')$ of $Z$ takes the form
\begin{equation}
\label{equation:cdhsequence}
\htp_S(Z' \times_Z Z'') \rightarrow \htp_S(Z') \oplus \htp_S(Z'') \rightarrow \htp_S(Z)
\end{equation}
By induction, the result holds for $Z'$ (resp. $Z''$ and $Z' \times_Z Z''$) and $T$. We conclude by tensoring
\eqref{equation:cdhsequence} with $\htp_S(T)$ and applying descent for the cdh-cover $(Z' \times_S T,Z'' \times_S T)$
 of $Z \times_S T$.
\end{proof}

As another corollary, the following computation explains the defect of absolute purity
 in the case of the immersion of a normal crossing divisor
 (and, in fact, in a slightly more general situation
  using our notion of h-smoothness).
\begin{cor}
\label{cor:comput_i^!_smc}
Let $i:Z \rightarrow X$ be a closed immersion such that $Z/X$ has h-smooth crossings.
Then $i^!(\un_X)$ is isomorphic to the homotopy colimit of the diagram
$$
\Th_Z(-N_I) 
\rightrightarrows
\bigoplus_{K \subset I, \sharp K=\sharp I-1} \Th_Z(-N_K)
\rightrightarrows\hdots
\bigoplus_{J \subset I, \sharp J=2} \Th_Z(-N_J)
\rightrightarrows \bigoplus_{i \in I} \Th_Z(-N_i)
$$
Here $N_J$ is the normal bundle of $Z_J$ in $Z$,
 $Th_Z(-N_J)$ is the associated Thom space (of the opposite), seen over $Z$.
 For any $J \subset K$, we consider the Gysin map of \Cref{prop:enhanced-Gysin}
$$
Th_Z(-N_J)=\cohtp_Z(Z_J,\twist{-N_J}) \xrightarrow{(\nu_K^J)_!=i^*\cohtp_{X!}(\nu_K^J)} \cohtp_Z(Z_K,\twist{-N_K})=Th_Z(-N_K)
$$
with the identification of the virtual cotangent bundle of $Z_J/X$ with
 the virtual bundle $\twist{-N_J}$. Then the degeneracy maps in the above diagram
 are given by the formulas:
$$
(\delta_n^k)_!=
\sum_{J=\{i_0<\hdots<\widehat{i_k}<\hdots<i_n\} \subset K=\{i_0<\hdots<i_n\}  } (\nu_K^J)_!
$$
\end{cor}
\begin{proof}
Applying \Cref{ex:sncexample1} to \( Z/X \) with \( v=0 \) results in the computation of \( \htp_X^c(Z) = \htp_X(Z) = i_! i^!(\un_X) \) as a colimit. By utilizing \Cref{prop:enhanced-Gysin} and the isomorphism of \( \infty \)-functors \( \htp^c_X \simeq \cohtp_{X!} \), we obtain an isomorphic diagram that still computes \( i_! i^!(\un_X) \), but consisting of objects of the form \( \cohtp_X(Z_J, -N_J) \). We conclude by applying the functor \( i^* \) and using the appropriate identifications.
\end{proof}

\begin{ex}\label{ex:comput_i^!_smc}
\Cref{cor:comput_i^!_smc}, applied 
to a strict normal crossing divisor in a regular scheme,
explains the failure of absolute purity for snc divisors and, 
more generally, for regular closed immersions that are h-smooth.
The augmentation map 
\begin{equation}\label{eq:snc_fdl_class}
\epsilon_i:\bigoplus_{i \in I} \Th_Z(-N_i) \rightarrow i^!(\un_X)
\end{equation}
coming form the above corollary can be seen as the ``best'' approximation of the fundamental class associated with $i$,
 in the spirit of \cite{DJK}.
\end{ex}


\begin{num}
\label{nota:normal-bundle-pair}
Consider a closed $S$-pair $(X,Z)$ such that $Z$ has h-smooth crossings over $S$ and such that for every nonempty subset 
$J\subset I$, $\overline{\nu}_J:Z_J\rightarrow X$ is an h-smooth closed immersion (see \Cref{not:extended-nc}). 
This holds, for instance, when $X$ is h-smooth in a Nisnevich neighborhood of $Z$. 
In such circumstances, $\bar{\nu}_J$ is, in particular, 
a regular immersion. We denote its associated normal bundle by $N_J$. Denote by $j: X-Z \to X$ the complementary open 
immersion. 
\end{num}

\begin{prop}
\label{cor:snccorollary2} Let $(X,Z)$ be a closed $S$-pair such that $Z$ has h-smooth crossings over $S$ and such that $X$ is h-smooth over $S$  in a Nisnevich neighborhood of $Z$.
 Let $v$ be a virtual vector bundle on $X$.  

Then the object $\Pi_{S}(X-Z,j^{-1}v)$ is isomorphic to the limit of the diagram
\begin{equation}
\label{equation:poles}
\htp_{S}(X,v) 
\xrightarrow \epsilon \bigoplus_{i \in I} \htp_{S}(Z_i,v_i+\twist{N_i}) 
\rightrightarrows \bigoplus_{J \subset I, \sharp J=2} \htp_{S}(Z_J,v_J+\twist{N_J})\rightrightarrows \cdots 
\rightrightarrows \htp_{S}(Z_I,v_I+\twist{N_I})
\end{equation}
given by the sums of the Gysin maps from \Cref{prop:enhanced-Gysin}
\begin{align*}
\epsilon=\sum_{i \in I} \htp_S^!(\bar{\nu}_i) & \; & 
(\delta^n_k)^!=\sum_{J=\{i_0<\hdots<\widehat{i_k}<\hdots<i_n\} \subset K=\{i_0<\hdots<i_n\}} \htp_S^!(\nu_{K}^J)
\end{align*}
associated to the closed immersions $\bar{\nu}_i:Z_i\rightarrow X$ and $\nu_K^J:Z_K\rightarrow Z_J$.

Dually, the object $\cohtp_S(X-Z,j^{-1}v)$ is isomorphic to the colimit  of the diagram 
\begin{equation}
\label{equation:poles-dual}
\cohtp_{S}(Z_I,v_I-\twist{N_I})
\rightrightarrows
\cdots
\rightrightarrows
\bigoplus_{J \subset I, \sharp J=2} \cohtp_{S}(Z_J,v_J-\twist{N_J})
\rightrightarrows
\bigoplus_{i \in I} \cohtp_{S}(Z_i,v_i-\twist{N_i})  \xrightarrow {\epsilon'}  \cohtp_{S}(X,v) 
\end{equation}
given by sums of the Gysin maps from \Cref{prop:enhanced-Gysin}
\begin{align*}
\epsilon'=\sum_{i \in I} \cohtp_{S!}(\bar{\nu}_{j}) & \ & 
(\delta_k^n)_!=\sum_{J=\{i_0<\hdots<\widehat{i_k}<\hdots<i_n\} \subset K=\{i_0<\hdots<i_n\}} \cohtp_{S!}(\nu_{K}^J)
\end{align*}
\end{prop}
\begin{proof}
With reference to \eqref{eq:loc-notation}, inserting $\mathbb{E}=\Th(v)\otimes f^!(\un_S)$ in the localization homotopy exact sequence $j_!j^! \rightarrow Id \rightarrow i_*i^*$ and applying $f_!$  yields the homotopy exact sequence
$$
\htp_S(X-Z,j^{-1}v)=f_!j_!j^!(\mathbb{E}) \rightarrow \htp_S(X,)=f_!(\mathbb{E}) \rightarrow f_!i_*i^*(\mathbb{E})
$$
By applying \Cref{thm:compute_closed_cover} to the closed cover $\bigsqcup Z_i'\to Z$ of $Z$ by its irreducible components, and then applying $f_!$ and arguing as in the proof of \Cref{ex:sncexample1}, we obtain the isomorphism 
$$f_!i_*i^*(\mathbb{E})\simeq \lim_{ n \in \Dinj} \left(\bigoplus_{J \subset I, \sharp J=n+1} f_!\bar{\nu}_{J*}\bar{\nu}_J^*(\mathbb{E})\right)$$


The object $f_!\bar{\nu}_{J*}\bar{\nu}_J^*(\mathbb{E})$ of $\T(S)$ depends only on a Nisnevich neighborhood 
of $Z$ in $X$.
Thus, under our hypotheses, 
we may replace $X$ by an h-smooth Nisnevich neighborhood of $Z$ in $X$ and assume that $f:X\rightarrow S$ 
itself is h-smooth, say with virtual relative tangent bundle $\tau_f$. 
We then have the purity isomorphism $\mathbb{E}\simeq \Th(v)\otimes \Th(\tau_f)$. 
Furthermore, 
under our assumptions, 
for every $J\subset I$,  
$\bar{\nu}_J:Z_J \rightarrow X$ and $f_J=f\circ \bar{\nu}_J:Z_J \rightarrow S$ are h-smooth morphisms. 
Since $\bar{\nu}_J^{-1}\tau_f=\tau_{f_J}+\twist{N_J}$, 
where $\tau_{f_J}$ is the virtual tangent bundle of the h-smooth morphism $f_J$ and 
$\Th(\tau_{f_J})\simeq f_J^!(\un_S)$ by purity, 
we obtain the isomorphisms 
\begin{align*}
f_!\bar{\nu}_{J*}\bar{\nu}_J^*(\mathbb{E})=f_{J!} \Th(\bar{\nu}_J^{-1}\tau_f)\otimes \Th(v_J)) & \simeq 
f_{J!}(\Th(\tau_{f_J})\otimes \Th(N_J)\otimes \Th(v_J)) \\ 
& \simeq f_{J!}((\Th(v_J)\otimes f_J^!(\un_S))\otimes \Th(N_J)) \\
& = \Pi_S(Z_J,v_J+\twist{N_J})
\end{align*}
In fact, applying the construction of \Cref{prop:enhanced-Gysin},
 we deduce that the above isomorphism can be turned into
an isomorphism of diagrams from the one obtained previously
 with the one considered in the statement, with the announced Gysin maps. 

The assertion for $\cohtp_S(X-Z,v)$ follows similarly by starting with the dual localization 
homotopy exact sequence $i_!i^!\rightarrow Id\rightarrow j_*j^*$.
We leave further details to the reader.
\end{proof}

\begin{rem}
The above result, 
in the dual case of $\cohtp_S(X-Z)$ and the torsion part of the motivic $\infty$-category
 $\DM_\et$, gives back the result of Fujiwara \cite[\textsection 8, third consequence]{Fujiwara}
 for torsion \'etale sheaves, deduced from the absolute purity theorem of Gabber.
\end{rem}

\begin{rem} Let us specialize the preceding result to the cases $\T=\DM, \DM_\et, \DM_\QQ$,
 and more specifically $\T=\DM_\QQ$ when considering Bondarko's weight structure (see \cite{BonWeights}).
Under the assumption and notations of \Cref{cor:snccorollary2},  
the motive $M_S(X-Z)$ is the limit of the augmented semi-simplical diagram
\begin{equation}
\label{eq:poles}
M_S(X) \xrightarrow \epsilon \bigoplus_{i \in I} M_S(Z_i)\twist{1}
\rightrightarrows \bigoplus_{J \subset I, \sharp J=2} M_S(Z_J)\twist{2} \hdots 
\rightarrow M_S(Z_I)\twist{c}
\end{equation}
with the same formulas as in \eqref{equation:poles} for the augmentation $\epsilon$ and the coface maps $\delta^n_k$.

In the case where $f:X\to S$ is smooth and proper, 
and $Z=D$ is a normal crossing divisor with irreducible components $D_i$, $i\in I$,
the formula for the motive $M_S(X-D)$ of the complement of a normal crossing divisor $D$ of $X/S$ is a relative motivic analog of the De Rham complex with logarithmic poles that Deligne used to define mixed Hodge structures. 
The motive of the non-proper  $S$-scheme $X-D$ is expressed as the ``complex" \eqref{eq:poles} whose terms $M_S(D_J)\twist{\sharp J}$ are  pure of weight $0$ for Bondarko's motivic weight structure. 
In particular,  it gives a canonical and functorial weight filtration for the motive 
$M_S(X-D)$ (recall that a pure object of weight $0$ shifted $n$ times has weight $n$).
We view this as a motivic analog of the fact that the weight filtration of the mixed Hodge structure on $X-D$ over $S=\Spec(\CC)$ arises from the naive filtration of the De Rham complex with logarithmic poles associated with $(X,D)$.

Dually, 
we can identify the Chow motive $h_S(X-D)$ with the colimit of the diagram
\begin{equation}
\label{equation:chowmotive}
h_S(D_I)\twist{-c} \rightarrow \hdots \bigoplus_{J \subset I, \sharp J=2} h_S(D_J)\twist{-2}
\rightrightarrows \bigoplus_{i \in I} h_S(D_i)\twist{-1} 
\xrightarrow 
{\epsilon'} h_S(X)
\end{equation}
When $S=\Spec(\CC)$, it follows from the identification of the orientation of the motivic spectrum representing algebraic De Rham cohomology given in \cite[Example 5.4.2(1)]{Deg12} that 
the De Rham realization of \eqref{equation:chowmotive}, 
see \cite[\textsection 3.1]{CD2}, 
can be canonically identified with the de Rham complex with logarithmic poles associated with
$(X, D)$.
\end{rem}

We finally derive 
the following generalization of a computation due to Rappoport and Zink, see \Cref{rem:RZS} for details. 
\begin{prop}
\label{cor:snccorollary}Let $(X,Z)$ be a closed $S$-pair corresponding to a closed immersion $i:Z\to X$ such that $Z$ has h-smooth crossings over $S$ and such that for every irreducible component $Z_i'$ of $Z$, the induced closed immersion $\bar{\nu}_i:Z_i\rightarrow X$ is h-smooth\footnote{This holds in particular when $X$ is h-smooth in a Nisnevich neighborhood of $Z$.}. For every $J\subset I$, let $N_J$ be the normal bundle of the induced regular closed immersion $\bar{\nu}_J:Z_J \rightarrow X$. 

Then the object $i^*j_*(\un_{X-Z})$ of $\mathscr T(Z)$ is isomorphic to the colimit  in the underlying $\infty$-category of the augmented semi-simplicial diagram of length $c+1$
$$
\cohtp_Z(Z_I,\twist{-N_I}) 
 \rightarrow \hdots \bigoplus_{J \subset I, \sharp J=2} \cohtp_Z(Z_j,\twist {-N_J})
\rightrightarrows
\bigoplus_{i \in I} \cohtp_Z(Z_i,\twist{-N_i})  \xrightarrow \epsilon \un_Z
$$
where the degeneracy maps are given (as in \Cref{cor:comput_i^!_smc}) by the formula
\begin{align*}
(\delta_n^k)_!=
\sum_{J=\{i_0<\hdots<\widehat{i_k}<\hdots<i_n\} \subset K=\{i_0<\hdots<i_n\} } i^*\cohtp_{X!}(\nu_{K}^J)
\end{align*}
using the ($\infty$-functorial) Gysin maps of \Cref{prop:enhanced-Gysin},
 associated to the regular closed immersions $\nu_K^J:Z_K\rightarrow Z_J$, $J\subset K$.
 The last map $\epsilon$ is obtained by composing \eqref{eq:snc_fdl_class} with the canonical map
 $i^!(\un_X) \rightarrow i^*(\un_X)=\un_Z$.

Dually, 
the object $i^!j_!(\un_{X-Z})$ in $\mathscr T(Z)$ is isomorphic to the limit of the following augmented 
semi-cosimplicial diagram of length $c+1$
$$
\un_Z 
\xrightarrow{\epsilon'} 
\bigoplus_{i \in I}\cohtp_Z^c(Z_i,\twist{N_i})
\rightrightarrows \bigoplus_{J \subset I, \sharp J=2} \cohtp_Z^c(Z_J,\twist{N_J}) \hdots
\rightarrow 
\cohtp_Z^c(Z_I,\twist{N_I})
$$
with degeneracy maps
\begin{align*}
(\delta_n^k)'_!= \sum_{J=\{i_0<\hdots<\widehat{i_k}<\hdots<i_n\} \subset K=\{i_0<\hdots<i_n\}} i^!\cohtp^c_{X_!}(\nu_{K}^J)
\end{align*}
\end{prop}
\begin{proof}
The first assertion immediately follows by applying $i^*$ to the localization triangle
 $$i_!i ^!(\un_X) \rightarrow \un_X \rightarrow j_*j^*(\un_X)=j_*(\un_{X-Z})$$
 and using the computation of \Cref{cor:comput_i^!_smc}. The other assertion is obtained
 similarly, starting from the dual localization triangle and applying $i^!$. 
%
%
\end{proof}

\begin{rem}
\label{rem:RZS}
Let $\T$ be a motivic $\infty$-category with a realization functor from $\DM_\et$ as in \eqref{eq:motivic_cat}.
Assume that $X$ is regular and that $Z=D$ is a normal crossing divisor in $X$ with irreducible components $D_i$, $i\in I$.
The above formula shows that the motive $i^*j_*(\un_{X-Z})$  is the colimit in the underlying $\infty$-category of the diagram
\begin{equation}
\label{equation:RZgeneralization}
\nu_{I*}(\un_{D_I})(c)[2c] \xrightarrow{d_{c-2}} \hdots 
\xrightarrow{d_1}
\bigoplus_{J \subset I, \sharp J=2} \nu_{J*}(\un_{D_J})(2)[4] \xrightarrow{d_0}
\bigoplus_{i \in I} \nu_{i*}(\un_{D_i})(1)[2] \xrightarrow \epsilon \un_D
\end{equation}
Here, 
$d_n=\sum_k (-1)^k (\delta_n^k)_!$
 is the alternate sum of Gysin maps associated with the relevant closed immersions
  (see \ref{num:Gysin}, given that $\nu_{J*}(\un_{D_J})=\cohtp_D(D_J)$).
The computation for \eqref{equation:RZgeneralization} specializes under $\ell$-adic realization
 to the Rapoport-Zink formula \cite[Lemma 2.5]{RapZink},
 which was inspired by analogous computations of Steenbrink in Hodge theory \cite{Steenb}.
 The lemma of Rapoport and Zink is used to obtain the so-called
 \emph{weight spectral sequence} (see \cite[Satz 2.10]{RapZink}) which has been used to deduce
 various cases of Deligne's weight monodromy conjecture (see the introduction of \cite{Ito}).
 Similarly, one can deduce from our computation a motivic version of the Rapoport-Zink
 and Steenbrink weight spectral sequences, which naturally specializes by realization to both versions.
\end{rem}

\subsection{Explicit models in the $\ZZ$-linear case}
\label{sec:explicit_models}

\begin{num}\label{num:Z-linear_yoneda_motivic}
We now assume that $\T$ is an $\HZ$-linear motivic $\infty$-category.
 Thus, for any scheme $S$, $\T(S)$ is a presentable $\HZ$-linear $\infty$-category,
 and this implies that the given functor $\htp_S$ admits a right Kan extension \cite[\textsection 4.3]{LurieHTT}
 along the $\HZ$-linear Yoneda embedding $\ZZ_S$\footnote{Actually, one can even get a functor $\ehtp_S$
 from the stable $\AA^1$-derived $\infty$-category $\Der_{\AA^1}$ by the $\HZ$-linear analog
 of the universality theorem of Drew and Gallauer \cite{DrewGall}.}
$$
\xymatrix@R=10pt@C=20pt{
\Sm_S\ar^{\htp_S}[rr]\ar_-{\ZZ_S}[rd] && \T(S) \\
& \Der\big(\Sh(\Sm_S,\ZZ)\big)\ar_-{\ehtp_S}[ru]\ar@{=>}[u] &
}
$$
Let us also consider the inclusion $\rho:\Sm_S \rightarrow \Sch_S$
 of Nisnevich sites. In this situation, one has an adjunction
 of $\HZ$-linear $\infty$-categories (see \cite[\textsection 6.1, Ex. 6.1.13]{CD3})
$$
\rho_!:\Sh(\Sm_S,\ZZ) \leftrightarrows \Sh(\Sch_S,\ZZ):\rho^*
$$
 such that $\rho^*$ is the restriction functor and $\rho_!$ is fully faithful.
 This, together with the fact that $\T$ satisfies $\cdh$-descent (\cite[cdh-descent]{CD3}) 
 implies that the functor $\bar \Pi_S$ admits a left Kan extension
$$
\xymatrix@=10pt{
\Der\big(\Sh(\Sm_S,\ZZ)\big)\ar^{\ehtp_S}[rr]\ar_/-4pt/{\rho^\cdh_!}[rd] &\ar@{=>}[d]& \T(S) \\
& \Der\big(\Sh_\cdh(\Sch_S,\ZZ)\big)\ar_/12pt/{\eehtp_S}[ru]
}
$$
where $\rho_!^\cdh=a_\cdh \rho_!$ is composite with the associated $\cdh$-sheaf functor.
Given an $S$-scheme $X$ of finite type, we set
$$
\uZZ_S(X,\T)=\eehtp_S(\ZZ_S^\cdh(X))
$$
where $\ZZ_S^\cdh(X)$ is the $\cdh$-sheaf of abelian groups represented by $X$.
 In this way, one has defined a covariant ($\infty$-)functor
 $\uZZ_S^{\T}:\Sch_S \rightarrow \T(S)$. As $\eehtp_S$ is obtained by a right Kan extension,
 one also gets the formula for any morphism $p:X \rightarrow S$ of finite type:
$$
\uZZ_S(X,\T) \simeq \varprojlim_{V/X} p_{V!}p_V^!(\un_S)
$$
where the limit runs over the $S$-morphisms $V \rightarrow X$ with $V$ a smooth $S$-scheme.
In particular, one gets a canonical map
\begin{equation}\label{eq:compare_mot_ext}
\htp_S(X;\un_S)=p_!p^!(\un_S) \longrightarrow \varprojlim_{V/X} p_{V!}p_V^!(\un_S)\simeq \uZZ_S(X,\T) 
\end{equation}
with the notation of \Cref{num:6functors_extension}. 
One should be cautious that this map is an isomorphism when $X/S$ is smooth, 
but not necessarily in general (see, however, \Cref{cor:model-smcrossings} below).

Using both extensions, we will show that the computations obtained 
 in the previous paragraph can be enhanced by giving models
 in terms of explicit complexes of (pre)sheaves.
\end{num}

\begin{num}\label{num:Z-linear_Cech_smcrossings}
Let us consider the notation and assumptions of the previous paragraph.
 We also consider an $S$-scheme $Z$ of finite type with reduced smooth crossings,
 and the finite closed cover
 $p:Z_\bullet=\sqcup_{i \in I} Z_i \rightarrow Z$
 associated with its integral components as in \Cref{num:normal_crossing}.
 We let $c=\#I$ be the number of integral components of $Z$.
 Then we can consider the complex $\cC^\ord_*(X_\bullet/X,\ZZ)$
 of abelian sheaves in $\Sh(\Sm_S,\ZZ)$ associated with the ordered \v Cech complex
$$
\cC^\ord_n(Z_\bullet/Z,\ZZ)=\sum_{J \subset I, \sharp J=n+1} \ZZ_S(Z_J) \\
$$
where we recall that $Z_J=(Z'_J)_{red}$, and with differentials
\begin{equation}\label{eq:HZ-diff-smcrossings}
d^n=\sum_{K \subset I, \sharp K=n+1} \sum_{k=0}^n (-1)^k.(\nu_{K}^{K\backslash k})_*    
\end{equation}
where we have denoted by $K\backslash k$ the set $K$ minus its $k$-th element,
 for the order on $K$ induced by that of $I$.
 We can view this complex in the big category of $\cdh$-sheaves by applying the functor
 $\rho_!^\cdh$. Then it becomes an augmented complex in $\Sh_\cdh(\Sch_S,\ZZ)$ 
$$
\rho_!^\cdh\cC^\ord_*(Z_\bullet/Z,\ZZ) \xrightarrow{\epsilon_{Z_\bullet/Z}} \ZZ_S^\cdh(Z) 
$$
Using the same idea as in \Cref{thm:compute_closed_cover}, we get the following lemma:
\end{num}
\begin{lm}
Consider the above assumptions. Then the augmented \v Cech ordered complex is
 acyclic \emph{i.e.} the map $\epsilon_{Z_\bullet/Z}$ is a quasi-isomorphism
 of complexes of $\Sh_\cdh(\Sch_S,\ZZ)$.
\end{lm}
\begin{proof}
As stated, this is analogous to the proof \Cref{thm:compute_closed_cover}.
 We can assume that $X=S$ using the existence of the functor $p_{Z\sharp}$ for the projection map $p_Z:Z \rightarrow S$,
 as we work with the big $\cdh$-site. As $(p_i:Z_i \rightarrow Z)$ is a $\cdh$-cover,
 it suffices to check that $\epsilon_{Z_\bullet/Z}$ is a quasi-isomorphism after pullback
 along $p_i:Z_i \rightarrow Z$. Then the closed cover $p$ becomes split and the lemma
 follows. 
\end{proof}

\begin{cor}\label{cor:model-smcrossings}
Consider the above assumptions (\Cref{num:Z-linear_yoneda_motivic}
and \Cref{num:Z-linear_Cech_smcrossings}).
There are isomorphisms in $\T(S)$
$$
\eehtp_S\cC_*^\ord(Z_\bullet/Z,\ZZ) \xrightarrow \simeq \uZZ_S(Z,\T)
 \xleftarrow \simeq \htp_S(Z)
$$
The first isomorphism is obtained by applying $\eehtp_S$ to
 the augmented ordered \v Cech complex, and the second one is defined 
 in \eqref{eq:compare_mot_ext}.
\end{cor}
\begin{proof}
The first isomorphism is obvious from the above lemma,
 and the second one follows either by using the fact both objects admit
 a finite resolution by objects associated with smooth $S$-schemes
 (or by induction on the number of integral components of $Z$).    
\end{proof}

\begin{rem}
The preceding corollary can be viewed as a method for computing the (homotopy) colimit described in \Cref{cor:snccorollary2}.
More precisely, it provides a way to identify a suitable model for this homotopy colimit.
\end{rem}

\begin{ex}\label{ex:sncexample2}
\begin{enumerate}
\item Assume $\T=\DAx t$ is the $t$-local stable $\AA^1$-derived motivic $\infty$-category,
 for the topology $t=\nis, \et, \h$
 (see e.g. \cite[Ex. 5.3.31]{CD3} for the first two,
  and \cite{CD4} for the last one). Then $\eehtp_S\cC_*^\ord(Z_\bullet/Z,\ZZ)$ is nothing else than
  the infinite suspension of the $\AA^1$-localization of the complex
\begin{equation}
\label{eq:alternate-complex}
\ZZ^t_S(Z_{I}) 
\xrightarrow{d_{c-2}} \bigoplus_{J \subset I, \sharp J=c-1} \ZZ_S^t(Z_{J}) \rightarrow
\hdots 
\xrightarrow{\ d_1\ }
\bigoplus_{J \subset I, \sharp J=2} \ZZ_S^t(Z_{J})
\xrightarrow{\ d_0\ } \bigoplus_{i \in I} \ZZ_S^t(Z_{i})
\end{equation}
with representable $\ZZ$-linear $t$-sheaves over $\Sm_S$ as indicated,
 and with differentials given by the alternating sum of formula \eqref{eq:HZ-diff-smcrossings}.
 This gives an explicit model for the ``$t$-local $\AA^1$-motive'' $\htp_S(Z,\DA_t)$ associated
 with the smooth reduced crossing $S$-scheme $Z$. In fact, the latter object is also modeled
 by the $\ZZ$-linear $t$-sheaf $\ZZ_S(Z)$ on $\Sm_S$ represented by $Z$,
 and the isomorphism with the above complex is then given by the natural augmentation map.
\item Assume $\T=\DM_\Lambda$ is the motivic $\infty$-category of $\Lambda$-linear motives.
 We assume either that $S$ is regular and defined over a field of characteristic exponent $p$ and $p \in \Lambda^\times$,
 or that $S$ geometrically unibranch and $\Lambda=\QQ$. Then a model for the motive $M_S(Z)_\Lambda$ is given
 by considering the complex 
 \begin{equation}
\label{eq:alternate-complex-cor}
[Z_{I}] \xrightarrow{d_{c-2}} \bigoplus_{J \subset I, \sharp J=c-1} [Z_{J}]
\xrightarrow{\ d_{c-3}\ }
\hdots 
\xrightarrow{\ d_1\ } \bigoplus_{J \subset I, \sharp J=2} [Z_{J}]
\xrightarrow{\ d_0\ } \bigoplus_{i \in I} [Z_{i}]
\end{equation}
in the additive category $\Sm^{\mathrm{cor}}_S$ of smooth $S$-schemes with finite correspondences,
 taking its image in $\DM^{\mathrm{eff}}(S,\Lambda)$ and then taking its infinite suspension.
 Another possible model is the analog complex but made with the corresponding Nisnevich $\Lambda$-linear sheaves with transfers.
It is obtained by applying the associated free sheaf with transfers functor $\ZZ^{tr}_S$.
\end{enumerate}
\end{ex}

\begin{rem}
The preceding formulas are the motivic relative version of the classical computation of the homology of a normal crossing scheme. 
 It actually gives back the known formulas by realization of motives (Betti, \'etale, etc...). 

A dual formula holds for computing the relative Chow motive $h_S(Z)=f_*f^*(\un_S)$.
To that end, we consider the isomorphism $\mathrm{h}(Z_\bullet/Z,\un_S)$ of Theorem \ref{thm:compute_closed_cover}: 
$h_S(Z)$ is quasi-isomorphic to the image of the complex  \eqref{eq:alternate-complex} under the (derived) internal Hom functor 
$\derR \uHom(-,\un_S)$, see also \Cref{thm:strong_duality}.
\end{rem}

\begin{num}\label{num:Z-linear_coCech_smcrossings}
We use the notation of \Cref{num:Z-linear_yoneda_motivic} and assume 
$(X,Z)$ is a closed $s$-pair such that $X$ is $S$-smooth (see \Cref{num:closed_pairs}).
 Let us denote by $\ZZ_S(X/X-Z)$ the cokernel of the canonical map
 $\ZZ_S(X-Z) \rightarrow \ZZ_S(X)$ in the abelian category $\Sh(\Sm_X,\ZZ)$.
 This cokernel is covariant with respect to morphisms of closed pairs,
 and in particular contravariant in $Z$ with respect to closed immersions.
 
Let $p:Z_\bullet=\sqcup_{i \in I} Z_i \rightarrow Z$ be a finite closed cover.
 Using again the notation of \eqref{num:ordeded_Cech}, we can define
 an ordered \v Cech complexes of sheaves in $\Sh(\Sm_X,\ZZ)$,
 with the cohomological convention
$$
\cC_\ord^n(X/X-Z_\bullet,\ZZ)=\bigoplus_{J \subset I, \sharp J=n+1} \ZZ_S(X/X-Z_J)
$$
and differentials
\begin{equation}\label{eq:HZ-codiff-smcrossings}
d^n=\sum_{K \subset I, \sharp K=n+1} \sum_{k=0}^n (-1)^k.(\nu_{K}^{K\backslash k})^*
\end{equation}
using notation as in \Cref{num:Z-linear_Cech_smcrossings}
(again $K\backslash k$ is the set $K$ minus its $k$-th element).
 This is a co-augmented complex in $\Sh(\Sm_S,\ZZ)$ 
$$
\ZZ_S(X/X-Z) \xrightarrow{\epsilon'_{X/X-Z_\bullet}} \cC^*_\ord(X/X-Z_\bullet,\ZZ)
$$
The following lemma is a particular case  of \Cref{thm:compute_closed_cover}.
\end{num}
\begin{lm}
Under the above assumptions, the co-augmentation $\epsilon'_{X/X-Z_\bullet}$
 is a quasi-isomorphism of complexes of Zariski (and a fortiori Nisnevich) sheaves.
\end{lm}
\begin{proof}
One reduces to the case where $X=S$, using the (derived or $\infty$) functor $p_\sharp$,
$p:X \rightarrow S$, and to the small Zariski site $X_\zar$.
Moreover, it suffices to check the statement on fibers along points $x$ of the scheme $X$.
Now the result reduces to an exercise in homological algebra using that 
$$
\ZZ_X(X/X-Z_J)_x=
\begin{cases} 
\ZZ & x \in Z_J  \\ 
0 & x \notin Z_J 
\end{cases}
$$
\end{proof}

\begin{rem}
In fact, the above statement is equivalent (and the proof is the same)
 to a higher version of the classical
 Mayer-Vietoris triangle, stating that the augmented complex:
$$
\ZZ_S(X-Z)  \xrightarrow{\epsilon'} \bigoplus_{i \in I} \ZZ_S(X-Z_i)
 \rightarrow \hdots \rightarrow \bigoplus_{J \subset I, \sharp J=n+1} \ZZ_S(X-Z_J)
 \rightarrow \hdots \rightarrow \ZZ_S(X-Z_I)
$$
is exact for the Zariski topology, the differentials being alternated sums as above.
 We could not find a reference in the literature
 for this rather obvious generalization of the Mayer-Vietoris triangle.
\end{rem}

\begin{cor}\label{cor:model-cosmcrossings}
Under the above assumptions, there are isomorphisms in $\T(S)$
$$
p_{Z!}i^*p_X^!(\un_S) \simeq \ZZ_S(X/X-Z,\T) \xrightarrow{\epsilon'_{X/X-Z_\bullet}} \ehtp_S \cC^*_\ord(X/X-Z_\bullet,\ZZ)
$$
where $p_Z$, and $p_X$ are the canonical projections.
The first isomorphism follows from the localization property.
\end{cor}

\begin{ex}\label{ex:sncexample2bis}
We can consider again the settings of \Cref{ex:sncexample2}, $\T=\DAx t, \DM_\Lambda$.
 These motivic $\infty$-categories are all defined through ($\AA^1$-)localization and ($\PP^1$-)stabilization
 of a derived category of $\Lambda$-linear $t$-sheaves with/without transfers.\footnote{To be 
 precise, one must consier an intermediary abelian category of symmetric $\GG$-spectra in order to get the $\PP^1$-stable category: see \cite[\textsection 5.3.C]{CD3}. The reader as the choice of applying the natural
 suspension functor at the level of abelian cateories (\emph{loc. cit.} (5.3.16.1)) to the next resolution in order to get a model in those terms.} If one denotes by
 $\ZZ^\epsilon_S(X)$ the corresponding free sheaf, with the expected properties, represented by a smooth $S$-scheme $X$. Extending the definitions of \Cref{num:Z-linear_coCech_smcrossings},
  one denotes by $\ZZ^\epsilon_S(X/X-Z)$ the cokernel of the canonical map $\ZZ^\epsilon_S(X-Z) \rightarrow \ZZ^\epsilon_S(X)$ for any closed pair $(X,Z)$ with $X/S$ smooth.
  Then one can consider the complex
\begin{equation}
\label{eq:alternate-cocomplex}
\bigoplus_{i \in I} \ZZ_S^\epsilon(X/X-Z_{i})
 \xrightarrow{\ d^0\ } \bigoplus_{J \subset I, \sharp J=2} \ZZ_S^\epsilon(X/X-Z_{J}) 
 \xrightarrow{\ d^1\ } \hdots 
 \rightarrow \bigoplus_{J \subset I, \sharp J=c-1} \ZZ_S^\epsilon(X/X-Z_{J}) 
 \xrightarrow{d^{c-2}} \ZZ^\epsilon_S(X/X-Z_{I})
\end{equation}
with differentials given by formula \eqref{eq:HZ-codiff-smcrossings}.
 In the above, we have used a cohomological convention, so that the complex is concentrated
 in degree $[0,c-1]$. There is a natural augmentation map,
 which makes the above complex into a (cohomological) resolution of $\ZZ_S^\epsilon(X/X-Z)$
 once viewed in the category $\T(S)$ (that is, after $\AA^1$-localization and $\PP^1$-stabilization).

Note that later, it will be convenient to use homological conventions for the preceding
 complex. Then it is concentrated in homological degrees $[-c+1,0]$. 
\end{ex}

\subsection{Application to strong duality}
\label{subsection:asdr}

Next, we deduce some applications of the computations of Section \ref{sec:crossing_sing} towards strong duality results.

\begin{prop}\label{thm:strong_duality}
Let $Z/S$ be a proper $S$-scheme with smooth crossings, 
and let  $v$ be a virtual bundle over $Z$.
Then $\htp_S(Z,v)$ is rigid with dual $\cohtp_S(Z,-v)$ isomorphic to limit of the diagram 
$$
\xymatrix@=20pt{
	\bigoplus_{i \in I} \htp_S(Z_i,-v_i-\twist{T_i})\ar@<2pt>[r]\ar@<-2pt>[r]
	& \bigoplus_{J \subset I, \sharp J=2} \htp_S(Z_J,-v_J-\twist{T_J})
	\ar[r] & \hdots\ar[r]
	& \htp_S(Z_I,-v_I-\twist{T_I})
}
$$
where for every $J\subset I$, $T_J$ denotes the  tangent bundle of $Z_J/S$.
\end{prop}
\begin{proof}
According to \eqref{equation:snches2},  $\htp_S(Z,v)$ is isomorphic to the colimit of the finite diagram
$$
\xymatrix@=20pt{
	\htp_S(Z_I,v_I)\ar[r] &
	\hdots \ar[r]
	& \bigoplus_{J \subset I, \sharp J=2} \htp_S(Z_J,v_J)\ar@<2pt>[r]\ar@<-2pt>[r]
	& \bigoplus_{i \in I} \htp_S(Z_i,v_i)
}
$$
whose components are spectra of smooth proper schemes, hence rigid spectra. This implies $\htp_S(Z,v)$ is rigid. The fact that its dual is $\cohtp_S(Z,-v)$ follows from \Cref{prop:weak-dual}(2). On the other hand, by \eqref{equation:cohom-limit}, $\cohtp_S(Z,-v)$ isomorphic to the colimit of the diagram  $$
\xymatrix@=20pt{
	\cohtp_S(Z_I,-v_I)\ar[r] & 
	\hdots \ar[r]
	& \bigoplus_{J \subset I, \sharp J=2} \cohtp_S(Z_J,-v_J)\ar@<2pt>[r]\ar@<-2pt>[r]
	& \bigoplus_{i \in I} \cohtp_S(Z_i,v_i)
}
$$
whose components are isomorphic to $\htp_S(Z_J,-v_J-\twist{T_J})$ by combining \Cref{ex:duality} and \Cref{prop:weak-dual}(2).
\end{proof}

\begin{thm}\label{thm:strong_duality-pair} 
 Let $(X,Z)$ be a closed $S$-pair such that $X/S$ is smooth and proper, with tangent bundle $T$, and such that $Z/S$ has smooth crossings.  Let $v$ be a virtual vector bundle on $X$. 

Then $\htp_S(X-Z,j^{-1}v)$ and $\cohtp_S(X-Z,j^{-1}v)$ are rigid with duals $\htp^c_S(X-Z,-j^{-1}(v+\twist{T}))$ and $\cohtp^c_S(X-Z,-j^{-1}(v-\twist{T}))$, respectively.
\end{thm}

\begin{proof} One first appeals to 
\Cref{cor:snccorollary2} to conclude that $\htp_S(X-Z,j^{-1}v)$ (resp.\ $\cohtp_S(X-Z,j^{-1}v)$) 
is rigid as a limit (resp.\ colimit) of a finite diagram whose components are rigid spectra due to 
the assumption that $X$, and hence all the $Z_J$, $J\subset I$, are smooth proper $S$-schemes. 
The given expressions for the dual then follow from \Cref{prop:weak-dual}.
\end{proof}
Finally, we deduce an improvement of \Cref{thm:generalizedhomotopypurity}.
\begin{thm}
\label{thm:generalizedhomotopypurity2}
Let $(X,Z)$ be a closed $S$-pair 
such that $Z/S$ is proper with smooth crossing over $S$ and such that 
$X$ is smooth in a Nisnevich neighborhood of $Z$.  

Then, for every virtual vector bundle $v$ on $X$, 
 $\htp_S(X/X-Z,v)$ and $\cohtp_S(X/X-Z,v)$ are rigid with duals $\htp_S(Z,-i^{-1}v-i^{-1}\tau_{X/S})$ and $\cohtp_S(Z,-i^{-1}v+i^{-1}\tau_{X/S})$, respectively.
 \end{thm}
\begin{proof}
This is a direct combination of  \Cref{thm:generalizedhomotopypurity} and \Cref{thm:strong_duality}. 
\end{proof}
In other words, under the stated hypothesis, one gets a canonical (generalized) \emph{purity isomorphism} of the form:
\begin{equation}\label{eq:generalizedhomotopypurity2}
\htp_S(X/X-Z,v) \simeq \htp_S(Z,-i^{-1}v+i^{-1}\tau_{X/S})^\vee
 \simeq \cohtp_S(Z,i^{-1}\tau_{X/S}-i^{-1}v)
\end{equation}
Note that this isomorphism is natural in $X$ with respect to pullbacks,
 and in $Z$ with respect to inclusions $T \rightarrow Z$.
 It can also be checked that, whenever $Z$ is smooth over $S$, it coincides
 with the purity isomorphism of Morel and Voevodsky
 (see e.g. \Cref{thm:generalizedhomotopypurity}(2))
 composed with the inverse of the Poincar\'e duality isomorphism of \Cref{ex:duality}.


\subsection{Complements of stably contractible arrangements}
\label{subsection:smhtaioha}
To illustrate the preceding results, we determine the stable homotopy types of complements of normal crossing $S$-schemes with stably $\AA^1$-contractible components. 

\begin{num} \label{def:Stab-A1-cont-arrang}A \emph{stably $\AA^1$-contractible arrangement over $S$} is a closed $S$-pair $(X,Z)$ consisting of a smooth  stably $\AA^1$-contractible $S$-scheme $X$ and a closed subscheme $Z\subsetneq X$ with smooth crossing over $S$ that satisfies the following assumptions (see \Cref{num:normal_crossing}).
\begin{enumerate}
\item For any $J \subset I$,
every connected component of $Z_J$ is stably $\AA^1$-contractible over $S$.
\item For any 
$K \subsetneq J \subset I$, $Z_K$ is nowhere dense in $Z_J$.
\end{enumerate}
For a subset $J \subset I$, we set $n_J=\sharp J$, 
and for any generic point $x$ of $Z_J$ we let $c_x$ denote the codimension of $x$ in $X$.
\end{num}
\begin{ex} 
A basic example of a stably $\AA^1$-contractible arrangement consists of an arrangement of affine hyperplanes 
in affine space $\AA^d_S$ over $S$. 
\end{ex}

\begin{prop}
\label{prop:complementdivisor}
Let $S$ be a smooth stably $\AA^1$-contractible scheme over a field $k$ and let $(X,Z)$ be stably $\AA^1$-contractible arrangement over $S$. 
Then there exists a canonical isomorphism
$$
\Pi_S(X-Z) \simeq \bigoplus_{J \subset I, x \in Z_J^{(0)}} \un_S\big(c_x\big)\big[2c_x-n_J\big]
$$
In addition, 
if $Z$ is a normal crossing subscheme of $X$, then the isomorphism takes the form
$$
\Pi_S(X-Z) \simeq \bigoplus_{n=0}^d m(n)\un_S(n)[n]
$$
Here $d$ is the relative dimension of $X$ over $S$ and $m(n)$ denotes the sum of the number of 
connected components of all codimension $n$ subschemes $Z_J$ of $X$. 

\end{prop}
\begin{proof}
According to \Cref{cor:snccorollary2}
one obtains that $\Pi_S(X-Z)$ is the homotopy limit of the augmented semi-simplicial diagram 
\begin{equation}
\label{equation:finitetower}
\Pi_S(X) \rightarrow  \bigoplus_{i \in I} \Pi_S(Z_i,N_i)
\rightrightarrows \cdots \rightrightarrows \bigoplus_{J \subset I, \sharp J=n} \Pi_S(Z_J,N_J) \rightrightarrows \cdots
\end{equation}
Let $x$ is a generic point of $Z_J$, for $J \subset I$, 
and write $Z_J(x)$ for the associated connected component.
By assumption, $Z_J(x)$ is smooth and stably $\AA^1$-contractible over $S$, hence over $k$. It follows from \Cref{lm:stably_A1-contract_virtual_vb} that the rank $c_x$ vector bundle $N_J|_{Z_J(x)}$ is stably trivial, 
and hence
$$
\Pi_S(Z_J,N_J) 
\simeq 
\bigoplus_x \Pi_S(Z_J(x),N_J|_{Z_J(x)}) 
\simeq 
\bigoplus_x \un_S(c_x)[2c_x]
$$
To deduce the first assertion, it suffices to show that the morphisms in \eqref{equation:finitetower} are zero.
Recall that these maps are sums of Gysin morphisms $(\nu_K^{J})^!$ for $J,K \subset I$, $K=J \cup\{k\}$,
$\nu_K^J:Z_K \rightarrow Z_J$. 
We are reduced to consider maps of the form
\begin{equation}
\label{equation:finitemap}
\un_S(c_x)[2c_x] \rightarrow \un_S(c_y)[2c_y]
\end{equation}
Here, 
$x$ (resp.~$y$) is a generic point of $Z_J$ (resp.~$Z_K$). 
Since $Z_K$ is nowhere dense in $Z_J$,
all such maps belong to some stable cohomotopy group $\pi^{2r,r}(S)$ for $r>0$.
The assumption that $S$ is stably $\AA^1$-contractible over $k$ implies $\pi^{2r,r}(S) \simeq \pi^{2r,r}(k)$.
Morel's $\AA^1$-connectivity theorem shows the latter group is trivial. 
It follows that the map \eqref{equation:finitemap} is zero.

For the second assertion, 
it suffices to note that if $Z$ is a normal crossing subscheme,
then for any $J \subset I$, $Z_J$ has pure codimension $n_J$ in $X$.
\end{proof}

Using \Cref{prop:weak-dual}(3), we obtain the following rigidity result.

\begin{cor}
\label{cor:complementdivisor_duality}
With the notation and assumptions of  \Cref{prop:complementdivisor}, 
$\Pi_S(X-Z)$ is rigid with dual
$$
\Pi_S^c(X-Z)(-d)[-2d]\simeq  
\bigoplus_{K \subset I, x \in Z_K^{(0)}} \un_S(-c_x)[-2c_x+n_K]
$$
\end{cor}

%% file: shtpinfty.tex
\subsection{Punctured tubular neighborhoods}

\begin{df}\label{df:PuncturedTN}
Let $(X,Z)$ be a closed $S$-pair and let $v$ be a virtual vector bundle on $X$. The  \emph{punctured tubular $\T$-neighborhood} 
$\TN_S(X,Z,v)$ of $Z$ in $X$ relative to $S$ twisted by $v$ is the homotopy fiber in $\T(S)$ of the composite
$$
\beta_{X,Z}:\htp_S(Z,i^{-1}v) \xrightarrow{\nu_*} \htp_S(X,v) \xrightarrow{\nu^*} \htp_S(X/X-Z,v)
$$
Here the first map is induced by the immersion $i:Z\rightarrow X$, 
and the second one is defined in \Cref{df:quotient}. 
In the case of a trivial twist, we use the notation $\TN_S(X,Z)$.  
\end{df}


It is straightforward to verify that \(\TN_S(X,Z)\) is functorial for morphisms of closed pairs. Additionally, the functor \(\TN_S\) maps excisive morphisms to isomorphisms. Notably, the punctured tubular neighborhood depends solely on a Nisnevich neighborhood of \(Z\) in \(X\). In \Cref{cor:cdh-invariance} below, we will demonstrate an even more useful cdh-excision property. 

\begin{rem}
Our definition is motivated by the notion of the \emph{link} of a point on a hypersurface, as discussed by Brauner, Zariski, Milnor, and Mumford (see \cite{Milnor}, \cite{mumfordihes}). Following Mumford's work, we can interpret \(\beta_{X,Z}\) as a tubular neighborhood of \(Z\) in \(X\), and the homotopy cofiber corresponds to the pointed tubular neighborhood, drawing an analogy with the Gysin sequence (see the next example).

Extending this analogy, we can show that the complex realization of our definition, when \(Z\) is a point on a complex hypersurface in affine space, is indeed the link described above. This relationship will be clearly illustrated in our examples.
\end{rem}

\begin{ex}
\label{example:TN_toy_example}
Let $(V,X)$ be the closed $S$-pair corresponding to the zero section $s:X\rightarrow V$ of a vector bundle 
$V$ on a separated $S$-scheme $X$. 
Then, 
by definition, 
one obtains the homotopy exact sequence (see \ref{num:Euler_seq} for notation)
$$
\TN_S(V,X) \rightarrow \htp_S(X) \xrightarrow{e_S(V)} \Th_S(V) \label{eq:vb-homotopy-seq}
$$
In particular, 
$\TN_S(V,X)\simeq \htp_S(V^\times)$,
where $V^\times$ denotes the complement of the image of $s$. 
Hence $\TN_S(V,X)$ is the extension of $\htp_S(X)$ by $\Th_S(V)[-1]$ classified by the Euler class $e_S(V)$. 
The vanishing of $e_S(V)$ is, 
by definition, 
equivalent to the existence of a splitting 
$$ 
\TN_S(V,X) \simeq \htp_S(X) \oplus \Th_S(V)[-1]
$$
\end{ex}
\begin{rem}
\label{rem:sftmotivicrealization}
Assume that $S$ is the spectrum of a perfect field $k$ of characteristic exponent $p$. 
\Cref{example:TN_toy_example} implies that for the closed $S$-pair $(V,X)$ corresponding to the zero section 
$s:X\rightarrow V$ of a vector bundle $V$ of rank $r$ on a separated $S$-scheme $X$, 
$\TN_S(V,X)$ is a strictly finer invariant than its motivic realization. 
Indeed, the realization in $\DM(k)[1/p]$ of $\TN_S(V,X)$ is the extension of $M(X)$ by $M(X)(r)[2r-1]$ classified 
by the map $\tilde c_r(V): M(X) \to  M(X)(r)[2r-1]$ induced by multiplication with the top Chern class 
$c_r(V) \in \mathrm{CH}^r(X) \simeq  \Hom(M(X),\un(r)[2r])$. 
In particular, 
the sequence splits if $c_r(V)=0$.

However, the vanishing of the homotopy Euler class $e(V)$, 
which implies the vanishing of the Euler class in Chow-Witt groups, 
is a strictly stronger condition than the vanishing of the top Chern class $\tilde c_r(V)$. 
For the smooth affine quadric $5$-fold $X\colon x_1y_1+x_2y_2+x_3y_3=1$ in $\AA^{6}$, 
the kernel of the surjection $(x_1,x_2,x_3): k[Q]^{3}\rightarrow k[Q]$ defines a nontrivial and 
stably trivial vector bundle $V$ of rank $2$ on $X$. 
While $V$'s Chern classes are trivial, 
$V$'s Euler class in $\CHt^{2}(X)=\mathrm{K}_{-1}^{\mathrm{MW}}(k)$ equals $\eta$, 
see the case $n=2$ in \cite[Lemma 3.5]{zbMATH06349725}. 
\end{rem}

\Cref{example:TN_toy_example} admits the following generalization.
\begin{prop}
\label{prop:TN_weakly_smooth}
Let $(X,Z)$ be a weakly h-smooth closed $S$-pair (see \Cref{df:closedSpairs}) with normal bundle $N_{Z/X}$.
Then, there exists a homotopy exact sequence
$$
\TN_S(X,Z) \longrightarrow \htp_S(Z) \xrightarrow{e_S(N_{Z/X})} \Th_S(N_{Z/X})
$$
In other words, 
$\TN_S(X,Z)\simeq\htp_S(N_{Z/X}^\times)$. 
Moreover, if the Euler class of $N_{Z/X}$ vanishes, then 
$$
\TN_S(X,Z)\simeq \htp_S(Z) \oplus \Th_S(N_{Z/X})[-1]
$$
\end{prop}
\begin{proof}
One can assume that $X$ and $Z$ are h-smooth over $S$ by excision.
By appealing to the purity isomorphism of \Cref{thm:generalizedhomotopypurity}, 
one deduces the commutative diagram
$$
\xymatrix{
\htp_S(Z)\ar_{i_*}[rd]\ar^{\beta_{X,Z}}[rr] & & \htp_S(X/X-Z)\ar^\simeq[r]\ar@{}|/-1pt/{(1)}[d] & \htp_S(Z,N_{Z/X}) \\
& \htp_S(X)\ar@/_10pt/_{i^!}[rru]\ar[ru] & &
}
$$
Indeed, the commutativity of part (1) follows from the definitions of the Gysin map, 
the purity isomorphism, 
and the associativity formula for fundamental classes in \cite[Theorem 3.3.2]{DJK}.
Then, 
the homotopy exact sequence follows from the excess intersection formula of \cite[Proposition 3.3.4]{DJK}.
The remaining assertions follow as in the previous example.
\end{proof}

The following result presents a motivic version of a classical computation of topological punctured tubular neighborhoods, which arises from the octahedron axiom.

\begin{prop}
\label{prop:octahedron_TN}
Let $(X,Z)$ be a closed $S$-pair and let $v$ be a virtual vector bundle on $X$. 
Then, 
the columns and rows of the following diagram are homotopy exact
\begin{equation}
\label{eq:ocatedron}
\xymatrix@R=20pt@C=30pt
{
0\ar[r]\ar[d] & \htp_{S}(X-Z,j^{-1}v)\ar^{j_*}[d] \ar@{=}[r]\ar@{}|{(1)}[rd] & \htp_{S}(X-Z,j^{-1}v)\ar^{\alpha_{X,Z}}[d] \\
\htp_{S}(Z,i^{-1}v)\ar^{i_*}[r]\ar@{=}[d]\ar@{}|{(2)}[rd] & \htp_{S}(X,v)\ar[d]\ar[r] & \htp_{S}(X/Z,v)\ar[d] \\
\htp_{S}(Z,i^{-1}v)\ar^-{\beta_{X,Z}}[r] & \htp_{S}(X/X-Z,v)\ar[r] & \TN_S(X,Z,v)[1]
}
\end{equation}
\end{prop}
\begin{proof}
Indeed, 
the middle column (resp.~row) follows from \Cref{df:quotient}, 
the commutativity of (1) follows from the definition,
and that of (2) from the definition of $\beta_{X,Z}$. 
The lower-right corner of the diagram is just the formulation of the octahedron axiom.
\end{proof}

\begin{rem}
In more classical terms for cohomology with coefficients in a ring spectrum $\E$,
one obtains long exact sequences involving the punctured tubular neighborhood
\begin{align*}
\hdots \rightarrow \E^{n,i}_Z(X) \rightarrow \E^{n,i}(Z) \rightarrow &\E^{n,i}(\TN_S(X,Z)) \rightarrow 
\E^{n+1,i}_Z(X) \rightarrow \hdots \\
\hdots  \rightarrow \E^{n,i}(X,Z) \rightarrow \E^{n,i}(X-Z) \rightarrow &\E^{n,i}(\TN_S(X,Z)) \rightarrow 
\E^{n+1,i}(X,Z) \rightarrow \hdots 
\end{align*}
Here $\E^{**}_Z(X)$ (resp.~$\E^{**}(X,Z)$) is the cohomology with support (resp.~relative cohomology).
\end{rem}

One gets the following practical way of computing punctured tubular neighborhoods by using resolution of singularities:

\begin{cor}
\label{cor:cdh-invariance}
Let $f:(Y,T) \rightarrow (X,Z)$ be a \cdh-excisive morphism of closed $S$-pairs 
and let $v$ be a virtual vector bundle on $X$.  
Then, 
the induced map
$$
\TN_S(Y,T,f^{-1}v) \rightarrow \TN_S(X,Z,v)
$$
is an equivalence.
\end{cor}
\begin{proof}
Indeed, 
according to \Cref{prop:octahedron_TN}, 
one obtains a commutative diagram whose rows are homotopy exact sequences
$$
\xymatrix@R=10pt@C=30pt{
\TN_S(Y,T,f^{-1}v)\ar[r]\ar[d] & \htp_{S}(Y-T,f^{-1}(v)|_{Y-T})\ar[r]\ar[d]
& \htp_{S}(Y/T,f^{-1}v)\ar[d] \\
\TN_S(X,Z,v)\ar[r] & \htp_{S}(X-Z,v|_{X-Z})\ar[r] & \htp_{S}(X/Z,v)
}
$$
By assumption, the middle vertical map, induced by the restriction of $f$,
 is an equivalence. Moreover, the right-most vertical map is an equivalence
 according to the \cdh-descent property of $\T$ (see \cite[3.3.10]{CD3}).
\end{proof}


In particular, one can use any suitable resolution of singularities of a pair $(X,Z)$ to compute the punctured tubular neighborhood of $(X,Z)$.
 More precisely, if we can find a cdh-excisive morphism $(Y,T) \rightarrow (X,Z)$ such that $(Y,T)$ is smooth over the base $S$, then applying \Cref{prop:TN_weakly_smooth} and \Cref{cor:cdh-invariance}, we get 
 $\TN_S(X,Z) \simeq \htp_S(N_{T/Y}^\times)$.
We obtain several examples from singularity theory in this way --- 
$S$ can be any base, the spectrum of a field $k$ or even of $\ZZ$.
\vspace{0.1in}

\begin{ex}

Let $\mathbb{P}=\mathbb{P}^1_S$ be the projective line and $\cO(-1)=\mathbb{V}(\mathcal{O}_{\mathbb{P}}(1))$ be its tautological line bundle.
Consider the relative quadratic cone $X=V(xy-z^2)$ in $\AA^3_S$.
 Then, by blowing-up the ordinary double point at the origin $o_S$, one gets a resolution $Y \rightarrow X$ whose exceptional divisor is $\mathbb{P}$,
 with normal bundle $\cO(-2)=\cO(-1)^{\otimes 2}$.
 Therefore, we have $$\TN_S(V(xy-z^2),0_S) \simeq \htp_S(\cO(-2)^\times)$$ 
 For $S=\mathrm{Spec}(\mathbb{C})$, the underlying topological manifold of the complex realization of $\cO(-2)^{\times}$ is homotopy equivalent to the total space of unit tangent bundle $\mathrm{U}TS^2$ of the sphere $S^2=\mathbb{CP}^1$. As a topological manifold, $\mathrm{U}TS^2$ is homeomorphic to $\mathbb{RP}^3\cong \mathrm{SO}(3)$. Our computation thus recovers the stable homotopy type of the link of the germ of complex of hypersurface singularity $$(V=\{u^2+v^2-z^2=0\},0)\subset (\mathbb{C}^3,0)$$ defined in \cite[Chapter 2]{Milnor} as the intersection of $V$ 
 with a real $5$-sphere $S^5_{\varepsilon}\subset \mathbb{C}^3=\mathbb{R}^6$ of sufficiently small radius 
 $\varepsilon>0$ centered at origin.
Our computation also accounts for the real case: the underlying topological manifold of the real realization of $\cO(-2)^{\times}$ is homotopy equivalent to the unit tangent bundle of the circle $S^1=\mathbb{RP}^1$, hence to two disjoint copies of $S^1$. The latter equals the link of the real germ of isolated singularity  $(V=\{u^2-v^2-z^2=0\},0)\subset (\mathbb{R}^3,0)$. 
\end{ex}

\begin{ex}
\label{ex:doublepointon3fold}
Next we consider an ordinary double point in a 3-fold: say $X=V(xt-yz)$ in $\AA^4_S$, which is singular at the origin $o_S$. A resolution of the singularity is given by the blow-up $\tilde{X}\to X$ of $o_S$ with exceptional divisor $\mathbb{P} \times \mathbb{P}$, whose normal bundle is  $\cO_{\mathbb{P} \times \mathbb{P}}(-1,-1)=\mathrm{p_1}^*\cO(-1)\otimes \mathrm{p_2}^*\cO(-1)$. Another resolution $X^{-}\to X$ is given the blow-up of $X$ with center at the the Weil non-Cartier divisor $V(x,y)$. The exceptional locus of $X^{-}\to X$ is isomorphic to $\mathbb{P}$ and its normal bundle in $X^{-}$ is equal to $\cO_{\mathbb{P}}(-1)\oplus \cO_{\mathbb{P}}(-1)$. This yields two models of the punctured tubular neighborhood
$$
\TN_S(V(xt-yz),o_S) \simeq \htp_S([\cO_{\mathbb{P} \times \mathbb{P}}(-1,-1)]^\times) \simeq \htp_S([\cO_{\mathbb{P}}(-1) \oplus \cO_{\mathbb{P}}(-1)]^\times)
$$
The $S$-schemes  $[\cO_{\mathbb{P} \times \mathbb{P}}(-1,-1)]^\times$ and $[\cO_{\mathbb{P}}(-1) \oplus \cO_{\mathbb{P}}(-1)]^\times$ are actually both isomorphic to $V-\{o_S\}$. For $S=\mathrm{Spec}(\mathbb{C})$ the underlying topological manifolds of the complex realizations of these schemes are homotopy equivalent to the $S^1$-bundle over $S^2\times S^2$ with Euler class $(1,1)\in H^2(S^2\times S^2,\mathbb{Z})\cong \mathbb{Z}^2$ and to the trivial $S^3$-bundle over $S^2$, respectively. Again, our descriptions recover the (stable) homotopy of the link of the germ  of complex of hypersurface singularity $$(V=\{x_1^2+x_2^2+x_3^2+x_4^2=0\},0)\subset (\mathbb{C}^4,0),$$ this link being homotopy equivalent the unit tangent bundle $\mathrm{U}TS^3 \cong S^2\times S^3$.  

\end{ex}

\begin{rem}
The reader will find in \Cref{thm:maincomputation} a way of computing punctured tubular neighborhoods
 when dealing with resolution of singularities whose exceptional locus is snc.
 This was our main motivation for \Cref{section:NCboundary}.
\end{rem}

\begin{num} 
One can further interpret \Cref{prop:octahedron_TN} in terms of the six functors formalism. For the closed $S$-pair $(X,Z)$, consider the commutative diagram
$$
\xymatrix@R=8pt@C=40pt{
Z\ar@{^(->}^i[r]\ar_p[rd] & X\ar^{f}[d] & X-Z\ar@{_(->}_/4pt/j[l]\ar^q[ld] \\
& S &
}
$$
of \eqref{eq:loc-notation}. 
By combining the two localization triangles one gets, 
as a functorial enhancement of \eqref{eq:ocatedron}, 
the following commutative diagram of natural transformations of $\T(X)$
\begin{equation}
\label{eq:functorial_tub_Neighb}
\begin{split}
\xymatrix@R=16pt@C=24pt
{
 0\ar[r]\ar[d] & j_!j^!\ar_{ad'_{j_!j^!}}[d]\ar@{=}[r] & j_!j^!\ar^{\alpha_j}[d] \\
 i_!i^!\ar[r]\ar@{=}[d] & Id\ar[d]\ar^-{ad_{j^*,j_*}}[r] & j_*j^*\ar[d] \\
 i_!i^!\ar^{\beta_i}[r] & i_*i^*\ar[r] & i_*i^*j_*j^*
}
\end{split}
\end{equation}
Each arrow in \eqref{eq:functorial_tub_Neighb} is a unit or counit for one of the adjunctions 
$(k^*,k_*)$ or $(k_!,k^!)$, $k=i,j$.
The second and third rows (resp.~columns) are localization triangles, 
expressed in terms of natural transformations.
In particular, 
each row and column of \eqref{eq:functorial_tub_Neighb} is exact homotopy; 
specifically, it gives rise to a homotopy exact sequence in \(\T(X)\) when evaluated at any object. 
 
Note, moreover, that $\alpha_j$ is given by the map $j_! \rightarrow j_*$  "forgetting the support." 
The map $\beta_i$ corresponds to the natural transformation $\beta_i:i^! \rightarrow i^*$, 
which is specific to the case of (closed) 
immersions.\footnote{It can also be derived from the exchange transformation $i^!Id^* \rightarrow i^*Id^!$.}
Finally, using the identification of functors 
$ i^!j_! = i^*j_*[-1]$
obtained by applying the localization triangles (middle row of the previous diagram) 
and post-composing with $j_!j^!$, yields the homotopy exact sequence
$$
i_!i^!j_!j^! \rightarrow j_!j^! \rightarrow j_*j^*j_!j^!=j_*j^*
$$
Since the last arrow identifies with $\alpha_j$, one gets $i_!i^!j_!j^!=i_*i^*j_*j^*$, 
which gives the result since $i_*=i_!$ (resp.~$j^*=j^!$) is right invertible.

We thus obtain the following expression for the punctured tubular neighborhood.
\end{num}
\begin{prop}
There is a canonical equivalence
$$
\TN_S(X,Z) 
\simeq 
p_!i^!j_!q^!(\un_S)
=
p_!i^*j_*q^!(\un_S)[-1]
$$
\end{prop}

This relation explains the close connection between punctured tubular neighborhoods and nearby cycles. 
In this line of thought, we extend  \cite[Theorem 5.1]{Wild1} and \cite[1.4.6]{Deg12} to our context.

\begin{thm}
\label{thm:analytic}
Let $S$ be an excellent scheme, and let $(X,Z)$, $(Y,T)$ be closed $S$-pairs.
Assume that there exists an isomorphism $f:T \rightarrow Z$, 
which extends to an isomorphism of the respective formal completions
$\mathfrak f:\hat Y_T \rightarrow \hat X_Z$. Then, there exists a canonical equivalence
$$
\mathfrak f^*:\TN_S(Y,T) 
\xrightarrow{\simeq} 
\TN_S(X,Z)
$$
which is compatible with composition in $\mathfrak f$.
\end{thm}
\begin{proof}
We can assume that $Z=T$ and that $Z$ is reduced.
It suffices to show there is an equivalence 
$$
\tilde{\mathfrak f}^*:\TN_S(Y,T) \rightarrow \TN_S(X,Z)
$$
and a commutative diagram
$$
\xymatrix@C=40pt@R=0pt{
 & \Pi_S(Y/Y-Z)\ar^{\tilde{\mathfrak f^*}}[dd] \\
Z\ar[ru]\ar[rd] & \\
 & \Pi_S(X/X-Z)
}
$$
We can utilize the strategy outlined in the proof of \cite[Theorem 1.4.6]{Deg12} by applying Artin's approximation theorem at the points of \( Z \). This approach is valid under the assumption that \( S \) is excellent. Additionally, we can use Zariski hypercovers to globalize the situation. Importantly, we do not need to extend our motivic category to include diagrams of base schemes. The proof proceeds directly with the simplicial schemes corresponding to the Zariski hypercoverings within the \(\infty\)-category \(\mathcal{T}(S)\).
\end{proof}

\subsection{Punctured tubular neighborhood of subschemes with crossing singularities}

\Cref{thm:compute_closed_cover} allows us to derive our main computation of punctured 
tubular neighborhoods of h-smooth crossing subschemes (\Cref{df:normal_crossing}). 
We adopt the notation of \ref{not:extended-nc} and \ref{nota:normal-bundle-pair}.


\begin{thm}
\label{thm:maincomputation}
Let $(X,Z)$ be a closed $S$-pair such that $Z/S$ has h-smooth crossings over $S$
 and $X/S$ is h-smooth in a Nisnevich neighborhood of $Z$ and let $v$ be a virtual vector bundle on $X$. 
Then, 
$\TN_S(X,Z,v)$ is canonically isomorphic to the 
homotopy fiber of a map
$$
\colim_{n \in (\Dinj)^{op}} \left(\bigoplus_{J \subset I, \sharp J=n+1} \Pi_{S}(Z_J,v_J)\right)
\xrightarrow{\partial } \underset{n \in \Dinj}\lim
\left(\bigoplus_{J \subset I, \sharp J=m+1} \Pi_{S}(Z_J,v_J+\twist{N_J})\right)
$$
Here the direct images define the face maps 
$$
(\delta_n^k)_*=\sum_{K=\{i_0<\hdots<i_n\}, J=\{i_0<\hdots<\not{i_k}<\hdots<i_n\}} (\nu_{K}^J)_*
$$
in the source, and the Gysin maps define the coface maps 
$$
(\tilde{\delta}^m_l)^!=\sum_{K=\{i_0<\hdots<i_m\}, J=\{i_0<\hdots<\not{i_l}<\hdots<i_m\}} (\nu_{K}^J)^!
$$
in the target.
Moreover, the canonical map $\partial_0^0$ 
induced by $\partial$
between the $0$-th degree terms of both sides has the following description:
\begin{equation}\label{eq:abstract_mumford}
\partial_0^0=(\delta_{ij}=\bar{\nu}_j^!\bar{\nu}_{i*})_{i,j\in I}
\colon \bigoplus_{i \in I} \Pi_{S}(Z_i,v_i) 
\longrightarrow \bigoplus_{j \in I} \Pi_{S}(Z_j,v_j+\twist{N_j})
\end{equation}
Finally, using the Euler class $e(N_i):\un_{Z_i} \rightarrow \Th(N_i)$ (see paragraph \ref{num:Euler_seq})
 of the normal bundle $N_i$,
 one can compute the diagonal coefficients of this matrix as
$$
\delta_{ii}=p_{i!}\big( e(N_i) \otimes \Th(\tau_i+v_i)\big)
$$
where $p_i:Z_i \rightarrow S$ is the (h-smooth) projection,
 with virtual tangent bundle $\tau_i$.
\end{thm}
\begin{proof}
According to \Cref{df:PuncturedTN}, 
we have to compute the homotopy fiber of the map $$\beta_{X,Z}:\htp_S(Z,v) \rightarrow \htp_S(X/X-Z,v)$$
\Cref{ex:sncexample1} identifies $\beta_{X,Z}$'s source with the desired colimit whereas \Cref{cor:snccorollary2} identifies its target with the desired limit. The computation of the (co)face maps and of $\partial_0^0$ follows from these two propositions.
The final remark follows from the definition of $\tilde \delta_{ii}=\overline{\nu}^!_i\overline{\nu}_{i*}$,
 the excess intersection formula \cite[Proposition 3.2.8]{DJK}, 
 and $p_i^!(\un_S) \simeq \Th(\tau_i)$
 since $p_i$ is h-smooth by assumption.
\end{proof}

One can suggestively summarize the computation in \Cref{thm:maincomputation} with the diagram
\begin{gather*}
\left(
\vcenter{
	\xymatrix@R=10pt{
		\ar@{..}[d] & \\
		\bigoplus_{i_1<i_2} \Pi_{S}(Z_{i_1i_2})\ar@<4pt>[d]\ar@<-4pt>[d] & \\
		\bigoplus_{i \in I} \Pi_{S}(Z_i)\ar^-{\partial}[r] 
		& \bigoplus_{j \in I} \Pi_{S}(Z_j,\twist{N_j})\ar@<4pt>[d]\ar@<-4pt>[d] \\
		& \bigoplus_{j_1<j_2} \Pi_{S}(Z_{j_1j_2},\twist{N_{j_1j_2}})\ar@{..}[d] \\
		& 
	}
}
\right)
\end{gather*}

Typically, computing a punctured tubular neighborhood involves determining the homotopy colimit (or limit) of the left (or right) column, followed by calculating the map induced by the boundary operator, denoted as $\partial$. Building on this idea, we can provide an explicit model of our motivic punctured tubular neighborhood within this framework, provided that $\T$ is $\HZ$-linear. This model concretely realizes the aforementioned picture.

\begin{prop}\label{prop:mainHZcomputation}
Let us consider the assumptions of the above proposition,
 and assume that $\T$ is $\HZ$-linear as in \Cref{num:Z-linear_yoneda_motivic}.
 Then, the punctured tubular neighborhood $\TN_S(X,Z)$ is the image under the functor $\ehtp_S$
 of the following complex of Nisnevich sheaves
\begin{equation}\label{eq::mainHZcomputation}
\begin{split}
&\ZZ_S(Z_{I}) 
\xrightarrow{d_{c-2}} \bigoplus_{J \subset I, \sharp J=c-1} \ZZ_S(Z_{J}) \rightarrow
\hdots 
\xrightarrow{\ d_1\ }
\bigoplus_{J \subset I, \sharp J=2} \ZZ_S(Z_{J})
\xrightarrow{\ d_0\ } \bigoplus_{i \in I} \ZZ_S(Z_{i}) \\
\xrightarrow{\nu^*\nu_*}
&\bigoplus_{j \in I} \ZZ_S(X/X-Z_{j})
 \xrightarrow{\ d^0\ } \bigoplus_{K \subset I, \sharp K=2} \ZZ_S(X/X-Z_{K}) 
 \xrightarrow{\ d^1\ } \hdots 
 \xrightarrow{d^{c-2}} \ZZ_S(X/X-Z_{K})
\end{split}
\end{equation}
The source of $\nu^*\nu_*$ is placed in degree $0$.
\end{prop}
\begin{proof}
We apply Corollaries \ref{cor:model-smcrossings} and \ref{cor:model-cosmcrossings}, using the fact that one obtains a model of the homotopy cofiber in $\Der(\Sh(\Sm_S,\ZZ))$ by taking the (desuspended) cone. The result follows since $\ehtp_S$ is exact.
\end{proof}

\begin{rem}\label{rem:mainHZcomputation}
\begin{enumerate}
\item When working with different types of sheaves (\'etale/\h-sheaves, including those with transfers), one can always substitute the free sheaf functor $\ZZ_S$ with the one appropriate for the respective context.
\item In the specific case of the Nisnevich-local motivic $\infty$-category $\DA$, it is necessary to apply the $\AA^1$-localization functor to the aforementioned model to obtain $\AA^1$-local objects. This process introduces many higher homotopies obscured by the map $\beta_{X,Z} = \nu^* \nu_*$.
\item For instance, let us consider the situation over a field $S = \Spec(k)$, focusing on the category $\DM(k)[1/p]$. We first note that Voevodsky's cancellation theorem establishes that the infinite-suspension functor $\DM^{eff}(k)[1/p] \rightarrow \DM(k)[1/p]$ is fully faithful. Given a pair $(X,Z)$ over $k$, as specified in the previous proposition, one can examine the complex \eqref{eq::mainHZcomputation}, replacing the sheaves $\ZZ_k(Y)$ with the equivalent sheaf that includes transfers.  By applying Suslin's singular chain complex functor $C_*^{Sus}$ and deriving the total complex, we can model the punctured tubular neighborhood motive $M(\TN_k(X,Z))$. In a certain sense, the resulting double complex encapsulates the higher homotopies referenced earlier. We thank the referee for highlighting this observation for us; it will be further illustrated in an explicit example later (see \Cref{rem:TN-higher-htp}).
\item The formula of the preceding proposition is mentioned in \cite[4.7.2]{BGV}.
\end{enumerate}
\end{rem}

For a closed $S$-pair $(X,Z)$ such that $X$ is smooth over $S$ in a Nisnevich neighborhood of $Z$, 
$\tau_{X/S}$ is a well-defined virtual vector bundle on a suitable Nisnevich neighborhood of $Z$, 
and its restriction  $i^{-1}\tau_{X/S}$ to $Z$ is a well-defined virtual vector bundle on $Z$, 
see \Cref{num:res-cot-cplx}. 
Since the twisted punctured tubular neighborhood of $Z$ in $X$ depends only on a Nisnevich neighborhood of $Z$ in $X$, 
the object $\TN_S(X,Z,-v-\tau_{X/S})$ is well-defined for every virtual vector bundle $v$ on 
(a Nisnevich neighborhood of $Z$ in) $X$. 
One derives from \Cref{thm:generalizedhomotopypurity2} the following strong duality result.

\begin{thm}
\label{thm:PTN_dual}
Let $(X,Z)$ be a closed $S$-pair such that $X$ is smooth in a Nisnevich neighbordhood of $Z$ and such $Z/S$ is proper with smooth crossings over $S$. Then, for every virtual vector bundle $v$ on $X$, $\TN_S(X,Z,v)$ is rigid with dual $\TN_S(X,Z,-v-\tau_{X/S})[-1]$.

In particular, under the stated hypothesis, the punctured tubular neighborhood $\TN(X,Z)$ is auto-dual, up to twist and shift. 
\end{thm}

\subsection{Stable homotopy at infinity and boundary motives}
\label{subsection:shatinfinityabm}

As explained in the next examples, 
the following definition is rooted in both classical topology, see \cite{HR}, 
and in Wildeshaus' theory of boundary motives \cite{Wild2}.
\begin{df}
\label{df:smhtatinfinitydef}
The \emph{homotopy at infinity} of a separated $S$-scheme $X/S$ is the homotopy fiber computed in $\T(S)$ of the map $\alpha_{X/S}:\Pi_S(X)\to \Pi_S^c(X)$ in 
\eqref{eq:forget_supp} so that there is a homotopy exact sequence
$$
\htp_S^\infty(X) \longrightarrow \htp_S(X) \xrightarrow{\alpha_{X/S}} \htp_S^c(X)
$$
\end{df}

Owing to \eqref{eq:motivic_cat}, 
the main case is $\T=\SH$.
We refer to the spectrum $\htp^\infty_S(X)$ in $\SH(S)$ as the \emph{stable homotopy at infinity} of $X$ relative to $S$. 

\begin{ex}
\label{ex:shtp_infty_vb}
Let $p:V\rightarrow S$ be a vector bundle and
 consider the closed pair $(V,S)$ given by the zero section $s:S\to V$. Then, using purity isomorphisms,  one gets the commutative diagram
$$
\xymatrix@R=10pt@C=20pt{
& \htp_S(V)\ar^{\alpha_{V/S}}[r]\ar_{p_*}^\sim[ld]\ar^{can}[dd] & \htp_S^c(V)\ar@{=}[r]
 & p_*p^!(\un_S)\ar^{\mathfrak p_p}_\sim[d] & \\
\un_S\ar_-{e(V)}[rd] & & & p_*(\Th_V(p^{-1}V))\ar@{=}[r] & p_*p^*(\Th_S(V)) \\
& \htp_S(V/V-Z)\ar^{\mathfrak p_{V,S}}_\sim[rr] && \Th_S(V)\ar_-{ad_p}^/-15pt/\sim[ru] &
}
$$
The isomorphisms \( p_* \) and the unit \( ad_p \) are a consequence of \( \AA^1 \)-homotopy invariance. The purity isomorphism \( \mathfrak{p}_p \) exists because \( p \) is smooth, while \( \mathfrak{p}_{V,S} \) serves as the (tautological) purity isomorphism. The commutativity of the right-hand side can be established by applying \cite[Lemma 3.3.1]{DJK} with \( f = p \), \( i = s \), and \( i' = Id_V \). Meanwhile, the commutativity of the left-hand side follows from the definition of the Euler class \( e(V) \) (see \ref{num:Euler_seq}).
From this, we deduce the homotopy exact sequence
$$
\htp_S^\infty(V) \rightarrow \un_S \xrightarrow{e(V)} \Th(V)
$$ 
In other words, 
$\htp_S^\infty(V)\simeq \htp_S(V^\times)$ and, if $e(V)=0$ 
then $\htp_S^\infty(V)\simeq \un_S \oplus \Th(V)[-1]$ 
\end{ex}

It follows from the discussion in Section \ref{sec:notations}
that $\htp^\infty_S(X)$ realizes to the analogous definition for the other motivic $\infty$-categories of 
\eqref{eq:motivic_cat}.

\begin{ex}\textit{Motivic realization}. 
\label{ex:realization}
Let $S$ be the spectrum of a perfect field $k$ of characteristic exponent $p$ and let $X$ be a separated $k$-scheme. Then,  the motivic realization functor (see also \cite{mhskpao}, \cite{orpaomodules} in this case)
\begin{equation}
\label{equation:motivicrealization}
\SH(k) \rightarrow \DM(k)[1/p]
\end{equation} sends $\htp_k(X)$ to Voevodsky's homological motive $M(X)$ of $X$ (\cite[\textsection 8.7]{CD5}),
and it sends $\htp_k^c(X)$ to $M^c(X)$, 
Voevodsky's homological motive of $X$ with compact support (\cite[Proposition 8.10]{CD5}).
It follows that the motivic realization functor sends $\htp_k^\infty(X)$ to \emph{the boundary motive} 
$\partial M(X)$ of $X$ (see Wildeshaus \cite{Wild1}).
We generalize the above discussion to arbitrary base schemes in \Cref{section:Mumford}.
\end{ex}

\begin{rem}
The boundary motive is an essential part of Wildeshaus’ theory of \emph{interior motives}, which aims at fulfilling the motivic part of the Langlands program: attaching pure motives to certain automorphic forms. We refer the reader to \cite[Th. 4.3 and Def. 4.9]{Wild3} for the construction of the «e-part» of the interior motive attached to $X$ (a smooth $k$-scheme, $k$ a base field admitting resolution of singularities). This construction is obtained from the «e-part» of the boundary motive $\partial M(X)^e$ (see the proof of Theorem 2.4 of \emph{loc. cit.}), under an assumption on the weight filtration of $\partial M(X)^e$: namely, it «avoids weights -1 and 0» (\emph{loc. cit.}  Assumption 4.2). We refer the reader to \cite{Wild4}, Section 5 for applications to the motivic Langlands program.
\end{rem}

\begin{ex}\textit{Betti Realization}. 
\label{ex:realization2}
Let \( S \) be the spectrum of a field \( k \) that admits a complex embedding \( \sigma \). We consider the Betti realization functor (see Section \ref{sec:notations}) given by 
\begin{equation}
\label{equation:bettirealization}
\SH(k) \rightarrow \DB(k) = \mathrm{D}(\ZZ)
\end{equation}
Thanks to Ayoub's enhancement of this functor to an arbitrary base scheme using the technique of analytical sheaves \cite{Ayoub3}, we find that for any separated \( k \)-scheme \( X \), the spectrum \( \htp_k(X) \) corresponds to the singular chain complex \( S_*(X^\sigma) \) of the analytification \( X^\sigma \) of \( X \). Meanwhile, the spectrum \( \htp_k^c(X) \) corresponds to the Borel-Moore singular chain complex \( S_*^{BM}(X^\sigma) \).

Since \( X^\sigma \) is locally contractible and \( \sigma \)-compact, the latter complex is quasi-isomorphic to the complex \( S_*^{lf}(W) \) of locally finite singular chains (see \cite[Chapter 3]{HR}). Therefore, the stable homotopy type at infinity \( \htp_S^\infty(X) \) realizes to the singular complex at infinity \( S_*^\infty(X^\sigma) \) (see Definition \cite{HR}), which is defined by the distinguished triangle of chain complexes of abelian groups
\begin{equation}
\label{eq:top_infty_hlg}
S^\infty_*(X^\sigma) 
\rightarrow 
S_*(X^\sigma) 
\xrightarrow{\alpha_{X^\sigma}} 
S_*^{lf}(X^\sigma)
\rightarrow 
S^\infty_*(X^\sigma)[1]
\end{equation}
\end{ex}

As a corollary of \Cref{th:Kunneth-infty}, we get the following computations:
\begin{prop}\label{cor:Kunneth-infty}
In the setting of \Cref{th:Kunneth-infty} assume that either i) or ii) holds and that $Y/S$ is proper. Then,  there is a canonical isomorphism
$$
\htp_S^\infty(X \times_S Y) 
\simeq 
\htp_S^\infty(X) \otimes \htp_S(Y)
$$
\end{prop}


\begin{prop} \label{cor:Kunneth-A1-Cont}
In the setting of \Cref{th:Kunneth-infty}  assume that $g:Y\to S$ is smooth and stably $\AA^1$-contractible over $S$ with relative tangent bundle $T_g$ stably constant over $S$  and let $v_0$ be a virtual vector bundle over $S$ such that 
$\twist{T_g}=g^*v_0$ in $K_0(Y)$. 
Then,  there exists a homotopy exact sequence
$$
\htp^\infty_S(X \times_S Y) \longrightarrow 
\htp_S(X) \xrightarrow{\alpha_X \otimes \alpha_Y} 
\htp_S^c(X) \otimes \Th(v_0)
$$
In particular, 
if $T_g$ is the pullback of a vector bundle $V$ over $S$ with a trivial Euler class,
then 
$$
\htp^\infty_S(X \times_S Y)\simeq 
\htp_S(X) \oplus \htp_S^c(X) \otimes \Th_S(V)[-1]
$$
\end{prop}
Note that the splitting uses \Cref{ex:shtp_infty_vb}.

\begin{ex}
Let $X$ be a smooth stably $\AA^1$-contractible variety of dimension $d$ over a field $k$. 
\Cref{cor:Kunneth-A1-Cont} implies that 
$$
\htp^\infty_k(X)
\simeq
\un_k \oplus \un_k(d)[2d-1] 
\simeq 
\htp^\infty_k(\AA^d_k)
$$
In other words, 
stable homotopy at infinity cannot distinguish between $X$ and affine space $\AA^d_k$, 
as one would expect from topology (see \cite{asokostvar}). 
A theory of unstable motivic homotopy at infinity, however, 
is expected to provide a finer invariant, 
which will distinguish between $X$ and $\AA^d_k$.

Similarly, the situation for smooth morphisms $f:X\rightarrow S$ with stably $\AA^1$-contractible fibers over a general base $S$ is entirely described by their stable tangent bundles. In particular,
if $T_f$ is constant over $S$, equal to $f^*V$ for some vector bundle $V$ on $S$, then the stable homotopy type at infinity of $X$ is the same as that of the vector bundle $V$. It is thus essentially described by the Euler class of $V$ as explained in \Cref{ex:shtp_infty_vb}.
\end{ex}

\begin{rem}
In general, 
one can interpret $\htp^\infty_S(X)$ as an extension of $\htp_S(X)$ by $\htp_S^c(X)$. 
This viewpoint is prominent in Wildeshaus' work on boundary motives;
a motivic realization, 
where weight considerations are at stake.  
In topology, 
it is well-known that forming a product with Euclidean space $\mathbf{R}^{n}$ kills the fundamental group at infinity.
In our stable context, 
taking a product with affine space $\AA^n$, 
or more generally, any smooth stably $\AA^1$-contractible $S$-scheme $f:Y\to S$ of relative dimension $n$ with a trivial relative tangent bundle splits the extension in the sense that
$$
\htp^\infty_S(X \times Y)
\simeq 
\htp_S(X) \oplus \htp_S^c(X)(n)[2n-1]
$$
\end{rem}

As an application of the results and techniques above, we can now wholly determine the homotopy at infinity of complements of stably $\AA^1$-contractible arrangements in smooth stably $\AA^1$-contractible schemes over a field (see  \Cref{def:Stab-A1-cont-arrang}). 

\begin{prop}
Let $S$ be a smooth stably $\AA^1$-contractible scheme over a field $k$ and let $(X,Z)$ be a stably $\AA^1$-contractible arrangement over $S$ such that $Z$ is a normal crossing closed subscheme of $X$. Then,  there exists a canonical isomorphism
$$
\Pi^\infty_S(X-Z) 
\simeq 
\bigoplus_{i=0}^d m(i)\un_S(i)[i] \oplus \bigoplus_{j=0}^d m(j)\un_S(d-j)[2d-j-1]
$$
where $d$ is the dimension of $X$ over $S$ and where $m(n)$ denotes the sum of the number of connected components of all codimension $n$ subschemes $Z_J$ of $X$. 

\end{prop}
\begin{proof}
Indeed, applying \Cref{prop:complementdivisor} and \Cref{cor:complementdivisor_duality}, 
we deduce the homotopy exact sequence
$$
\htp_S^\infty(X) 
\longrightarrow 
\bigoplus_{i=0}^d m(i)\un_S(i)[i]
\longrightarrow 
\bigoplus_{j=0}^d m(j)\un_S(d-j)[2d-j]
$$
To conclude, 
it suffices to prove that the second map is zero.
Since $S$ is stably $\AA^1$-contractible over the field $k$, it is given by a sum of elements of the groups 
$\pi^{2d-i-j,d-i-j}(k)$. Since $d>0$, these groups are all trivial by Morel's stable $\AA^1$-connectivity theorem.
\end{proof}

\subsection{Stable homotopy type at infinity via punctured tubular neighborhoods}
\label{subsection:shtaivptn}

\begin{num} 
\label{compactification}
Recall that a \emph{compactification} of a separated morphism of finite type $f\colon X\to S$ consists of 
an open immersion $j:X\hookrightarrow \bar{X}$ into a proper $S$-scheme $\bar{f}\colon \bar{ X} \to S$. 
The closed subscheme $\partial X=(\bar X-X)_{red}$ of $\bar{X}$ is called the \emph{boundary} of the compactification $j$.  
We denote by $i:\partial X\hookrightarrow \bar{X}$ the corresponding closed immersion and set 
$\partial \bar{f}=\bar{f}\circ i:\partial X\to S$ in the commutative diagram
$$
\xymatrix@R=14pt@C=30pt{
X\ar@{^(->}^j[r]\ar_f[rd] & \bar X\ar^{\bar f}[d]
& \partial X\ar@{_(->}_i[l]\ar^{\partial \bar{f}}[ld] \\
& S &
}
$$
The following result gives our main tool for computing stable homotopy types at infinity.
For specializations to topology and motives, 
see \cite{HR} and \cite[Theorem 1.6]{Wild2}, 
respectively.
\end{num}


%

\begin{prop}
\label{prop:basic_comput_thp-infty}
Let $(\bar{X},\partial X)$ be the closed $S$-pair associated with a compactification of a separated $S$-scheme of 
finite type. Then,  there exists a canonical isomorphism
$$
\htp^\infty_S(X) \simeq \TN_S(\bar X,\partial X)
$$
which is natural in $(\bar X,X,\partial X)$, 
covariantly functorial with respect to proper maps, 
and contravariantly functorial with respect to \'etale maps.
\end{prop}
\begin{proof}
Given the six functors formalism, this is a direct application of \Cref{prop:octahedron_TN}. More precisely, with the notation of \Cref{compactification}, 
one reduces to the commutative diagram
$$
\xymatrix@R=10pt@C=35pt{
f_!f^!(\un_S)\ar^{\alpha_f}[rr]\ar_\sim[d]
 && f_*f^!(\un_S)\ar^\sim[d] \\
\bar f_*j_!j^!\bar f^!(\un_S)\ar^-{ad'_{j_!,j^!}}[r]
 & \bar f_*\bar f^!(\un_S)\ar^-{ad_{j_*,j^*}}[r] & \bar f_*j_*j^*\bar f^!(\un_S)
}
$$
and exactness of the rows and columns of \eqref{eq:functorial_tub_Neighb}.
\end{proof}

\begin{rem}
The above result has the following geometric interpretations. First, using the notations of \Cref{prop:octahedron_TN} for the closed $S$-pair $(\bar X,\partial X)$ and that of 
\Cref{df:smhtatinfinitydef}, 
the commutative diagram in the proof of \Cref{prop:basic_comput_thp-infty} can be recast as
$$
\xymatrix@R=12pt@C=30pt{
\htp_S(X)\ar^{\alpha_X}[r]\ar@{=}[d] & \htp^c_S(X)\ar^-\sim[d] \\
\htp_S(\bar X-\partial X)\ar^-{\alpha_{\bar X,\partial X}}[r] & \htp_S(\bar X/\partial X)
}
$$
In particular, 
considering the Borel-Moore homotopy $\Pi_S^c(X)$ of $X$ naturally leads to considering the object $\bar{X}/\partial X$ 
obtained by identifying the boundary $\partial X$ of any compactification $\bar{X}$ with a point. 
The latter can be viewed as a motivic model for the one-point compactification in topology.

Second, $\htp_{S}^\infty(X)$ can be canonically identified with the homotopy fiber of the canonical map
\begin{equation}
\label{equation:smhthf}
\htp_{S}(\partial X) \oplus \htp_{S}(X) \xrightarrow{i_*+j_*} \htp_{S}(\bar X)
\end{equation}
Under motivic realization, 
\eqref{equation:smhthf} becomes the formula for the boundary motive given in \cite[Proposition 2.4]{Wild1}.
\end{rem}

A reformulation of \Cref{prop:basic_comput_thp-infty} yields the following invariance result for the  punctured tubular neighborhood of a closed subscheme $Z$ of a proper $S$-scheme $X$: 

\begin{cor}
\label{cor:independance_TN}
Let $(X,Z)$ be a closed $S$-pair such that $X/S$ is proper.
Then,  the punctured tubular neighborhood $\TN_S(X,Z)$ is isomorphic to $\htp^\infty_S(X-Z)$,
and therefore it depends only on the open subscheme $X-Z$.
\end{cor}

By combining 
\Cref{prop:TN_weakly_smooth} and \Cref{prop:basic_comput_thp-infty}, we obtain the following result.
\begin{cor} \label{cor:irreducible-smooth-boundary}
Let $(\bar{X},\partial X)$ be the closed $S$-pair associated to a compactification of a separated $S$-scheme $X$. 
Assume that $(\bar{X},\partial X)$ is weakly h-smooth with normal  bundle $N=N_{\partial X/\bar{X}}$. 
Then,  there is a homotopy exact sequence
\begin{equation}
\label{eq:boundary-homotopy-seq}
\htp_S^\infty(X) \longrightarrow \htp_S(\partial X) \xrightarrow{e(N)} \Th_S(N)
\end{equation}
Here, $e(N)$ is induced by the Euler class of $N$ (see \ref{num:Euler_seq}). 
In particular, 
$\htp_S^\infty(X)\simeq\Pi_S(N^\times)$, 
and when $e(N)$ vanishes, 
there is a splitting $\htp_S^\infty(X)\simeq\htp_S(\partial X) \oplus \Th_S(N)[-1]$.
\end{cor}

%
%

\begin{rem}
\label{rem:sftmotivicrealization-1}Assume that $S$ is the spectrum of a perfect field $k$ of characteristic exponent $p$. Then,  the realization in $\DM(k)[1/p]$
of the homotopy exact sequence \eqref{eq:boundary-homotopy-seq}
is the homotopy exact sequence 
$$
\partial M(X) \longrightarrow M(\partial X)
\xrightarrow{\tilde c_r(N)} M(\partial X)(r)[2r]
$$
where  $\partial M(X)$ is the boundary motive of $X$ in Example \ref{ex:realization}, $r$ is the rank of the normal bundle  $N$ of $\partial_X$ in $\bar{X}$ and the map $\tilde c_r(N)$ 
is induced by multiplication with the top Chern class 
$c_r(N) \in \mathrm{CH}^r(\partial X) \simeq  \Hom(M(\partial X),\un(r)[2r])$. \Cref{cor:irreducible-smooth-boundary} implies
that $\htp_k^{\infty}(X)$ is a strictly finer invariant than $\partial M(X)$, see  \Cref{rem:sftmotivicrealization}.
\end{rem}


\subsection{Interpretation in terms of fundamental classes}
In what follows, 
we observe connections between stable homotopy at infinity 
and more generally punctured tubular neighborhoods and certain fundamental classes.
\begin{prop}
\label{prop:diagonal&shtpinfty}
Let $f:X \rightarrow S$ be a smooth morphism with relative tangent bundle $T_f$.
Then,  the map $\alpha'_{X/S}$ obtained by adjunction from the composite
$$
\htp_S(X) \xrightarrow{\alpha_{X/S}} \htp^c_S(X) \simeq \uHom\big(\htp_S(X,-T_f),\un_S\big),
$$
where the isomorphism uses \Cref{prop:weak-dual}(4),
fits into the commutative diagram
$$
\xymatrix@=20pt{
\htp_S(X) \otimes \htp_S(X,-T_f)\ar^-{\alpha'_{X/S}}[r]\ar_\simeq[d] & \un_S \\
\htp_S(X \times_S X,-p_j^{-1}T_f)\ar^-{\delta^!}[r] & \htp_S(X)\ar_{f_*}[u]
}
$$
The left vertical map is the K\"unneth isomorphism \eqref{eq:kunneth2} and $\delta^!$ is the Gysin map 
(\Cref{num:Gysin}) associated with the diagonal immersion $\delta:X\rightarrow X\times_S X$.
\end{prop}
In other words, the map $\alpha_{X/S}$, whose homotopy cofiber  is the stable homotopy at infinity of $X/S$,
 can be computed under the canonical isomorphisms
\begin{align*}
[\htp_S(X),\htp^c_S(X)] &\simeq [\htp_S(X) \otimes \htp_S(X,-T_f),\un_S] \\
 &\simeq  [\htp_S(X \times_S X),\Th(p_j^{-1}T_f)] = H_{\T}^0(X \times_S X,p_j^{-1}(T_f))
\end{align*}
as the twisted fundamental class $[\Delta_{X/S}]_{X\times X}^j$ of the diagonal, with respect to
 the $\delta$-parallelization corresponding to the smooth retraction $p_j$ of $\delta$, see \Cref{ex:i-parallelization}.


\begin{proof}
For notational convenience, 
let $p_1:X \times_S X \rightarrow X$ be the projection on the first factor. The associativity formula in 
\cite[Theorem 3.3.2]{DJK} shows the equality of fundamental classes $\eta_\delta.\eta_{p_1}=1$.  
The assumption that $f$ is smooth implies the cartesian square
$$
\xymatrix@=14pt{
X \times_S X\ar^-{p_1}[r]\ar_{p_2}[d]\ar@{}|\Delta[rd] & X\ar^f[d] \\
X\ar_f[r] & S
}
$$
is Tor-independent. Thus the transversal base change formula in \cite[Theorem 3.3.2]{DJK} implies the equality $\Delta^*(\eta_f)=\eta_{p_1}$
from which the commutativity of the square follows. 
\end{proof}

\begin{rem}
Computing fundamental classes of the diagonal is a famous problem, at the center of the Chow-K\"unneth conjecture, for example.
 The previous proposition shows the link between determining the stable homotopy type at infinity,
 or the boundary motive, of $X/S$ and computing the (twisted) fundamental class of its diagonal. The main difference with the Chow-K\"unneth conjecture
 is that we are interested mainly in the non-proper case.
\end{rem}

Similarly, one gets the following link between punctured tubular neighborhoods and another fundamental class.
\begin{prop}
\label{prop:graph&PTN}
Let $(X,Z)$ be a closed $S$-pair such that  $X/S$ is smooth with relative tangent bundle $T_{X/S}$ and such that $Z/S$ is proper and has smooth crossings (see \Cref{df:normal_crossing}). 
Then,  the map $\beta'_{X,Z}$ obtained by adjunction from
$$
\htp_S(Z) \xrightarrow{\beta_{X,Z}} \htp_S(X/X-Z)
\simeq \htp_S(Z,-\twist{i^{-1}T_{X/S}})^\vee,
$$
where the isomorphism follows from \Cref{thm:PTN_dual}, fits into the commutative diagram
$$
\xymatrix@=20pt{
\htp_S(Z) \otimes \htp_S(Z,-\twist{i^{-1}T_{X/S}})\ar^-{\beta'_{X,Z}}[rr]\ar_{Id \otimes i_*}[d] && \un_S \\
\htp_S(Z) \otimes \htp_S(X,-\twist{T_{X/S}})\ar_\simeq^{(*)}[r] & 
\htp_S(Z \times_S X,-\twist{p_j^{-1}T_{X/S}})\ar^-{\gamma_i^!}[r] &
\htp_S(Z)\ar_{q_*}[u]
}
$$
where  $\gamma_i^!$ is the Gysin morphism associated to the graph immersion $\gamma_i=Id\times i:Z\rightarrow Z\times_S X$.
\end{prop}
In other words, the map $\beta_{X,Z}$, whose cone is the punctured tubular neighborhood $\TN_S(X,Z)$ of the pair $(X,Z)$,
 can be computed under the canonical isomorphisms
\begin{align*}
[\htp_S(Z),\htp_S(X/X-Z)] & \simeq [\htp_S(Z),\htp_S(Z,-\twist{i^{-1}T_{X/S}})^\vee] 
 \simeq [\htp_S(Z) \otimes \htp_S(Z,-\twist{i^{-1}T_{X/S}}),\un_S] \\
& \simeq [\htp_S(Z\times_S X,-\twist{p_j^{-1} T_{X/S}},\un_S]
 \simeq H_{\T}^0(Z\times_S X,p_j^{-1}T_{X/S})
\end{align*}
as the twisted fundamental class $[\Gamma_i]_{Z\times X}^{can}$ of the graph $\gamma_i$
 of the closed immersion $i:Z\to X$, with obvious $\gamma_i$-parallelization
 $N_{\gamma_i} \simeq \gamma_i^{-1}(p_j^{-1} T_{X/S})$.
\begin{proof}
First, let us note that $\gamma_i$ is a section of the smooth separated morphism 
$Z \times_S X \rightarrow Z$.
So it is a regular closed immersion whose normal bundle is isomorphic to the relative tangent bundle $p_j^*T_{X/S}$ of $Z \times_S X$ over $Z$. This justifies the existence of the Gysin map $\gamma_i^!$ using \Cref{num:Gysin}.
Secondly,  the isomorphism ($*$) follows from the K{\"u}nneth isomorphism of \Cref{prop:Kunneth-smooth-nc}.
A routine check using the definitions of the maps shows that the diagram commutes.
\end{proof}

\begin{num}
Pushing the idea from the preceding result, one obtains a method of computation
 for the decomposition of punctured tubular neighborhoods obtained in \Cref{thm:maincomputation}.
 We use the notations of \emph{op. cit.}: $(X,Z)$ is a closed $S$-pair, $Z=\cup_{i \in I} Z_i$.
Furthermore, we make the following assumptions.
\begin{enumerate}
\item $X/S$ is smooth with relative tangent bundle $T_{X/S}$.
\item $Z/S$ is proper and has smooth crossings.
\end{enumerate}
In fact, as $Z_i/S$ is smooth and proper, one deduces from \Cref{ex:duality} that
  $\htp_S(Z_i,N_i)$, where $N_i$ denotes the normal bundle of $Z_i$ in $X$ is rigid with dual $\htp_S(Z_i,-\twist{T_{X/S}^i})$, where we denote by $T_{X/S}^i$
 the restriction of $T_{X/S}$ to $Z_i$ and use the isomorphism of virtual vector bundles
 $\twist{T_{X/S}^i}=\twist{N_i}+\twist{T_{Z_i/S}}$.
Combined with the K\"unneth formula \eqref{eq:kunneth2},
 one gets a canonical isomorphism
\begin{equation}\label{eq:iso_compute_pre_Mumford}
\varphi:[\htp_S(Z_i),\htp_S(Z_j,N_j)] \xrightarrow{\ \cong \ } H_{\T}^0(Z_i \times_S Z_j,p_2^{-1}T_{X/S}^j)
\end{equation}
\end{num}
\begin{prop}\label{prop:iso_compute_pre_Mumford}
Consider the above assumptions and the cartesian square of closed immersions
\begin{equation}
\label{equation:bigdiagram}
\xymatrix@R=14pt@C=28pt{
Z'_{ij}\ar[r]\ar_{\nu'_{ij}}[d] & X\ar^\delta[d] \\
Z_i\times_S Z_j\ar^{\bar \nu_i \times_S \bar \nu_j}[r] & X\times_S X
}
\end{equation}
Let $\delta_{ij}:\htp_S(Z_i) \rightarrow \htp_S(Z_j,N_j)$ be the map appearing in \Cref{thm:maincomputation}.
\begin{enumerate}
\item Through the isomorphism \eqref{eq:iso_compute_pre_Mumford}, we have 
$$
\delta_{ij}=(\bar \nu_i \times_S \bar \nu_j)^*\big([\Delta_{X/S}]^2_{X\times X}\big)
$$
The right-hand side is the \emph{second} twisted fundamental class of 
 the diagonal of $X/S$ (see \Cref{ex:i-parallelization}).
\item If $i=j$, $\nu'_{ii}$ is the diagonal $\delta_i$ of $Z_i/S$. 
We consider the map
$$
H^0_{\T}(X,N_i) \xrightarrow{\epsilon_{2*}} H^0_{\T}(Z_i,N_{\delta_i}+\delta_i^{-1}p_2^{-1}T_{X/S})
 \xrightarrow{\delta_{i!}} H^0_{\T}(Z_i \times_S Z_i,p_2^{-1}T_{X/S})
$$
where the first map is induced by the canonical isomorphism of virtual bundles
$$
\epsilon_2:\twist{N_i} \simeq \twist{N_i}-\twist{T_{X/S}^i}+\twist{\delta_i^{-1}p_2^{-1}T_{X/S}} \simeq \twist{N_{\delta_i}}+\twist{\delta_i^{-1}p_2^{-1}T_{X/S}}
$$
over $Z_i$ and $\delta_{i!}$ is the Gysin map in cohomotopy (see \Cref{num:Gysin}).
Let also $e(N_i)$ be the Euler class of the normal bundle $N_i$ of $Z_i/X$ (see \Cref{num:Euler_seq}).
Then,  through the isomorphism \eqref{eq:iso_compute_pre_Mumford}, we have
$$
\delta_{ii}=\delta_{i!}(\epsilon_{2*}e(N_i))
$$
\item Assume furthermore that \eqref{equation:bigdiagram} is transversal: 
$\nu'_{ij}$ is regular with normal bundle isomorphic to the restriction of $T_X$ to $Z_{ij}$,
i.e., 
it is of proper codimension.
Then,  $\delta_{ij}$ can be computed through the isomorphism \eqref{eq:iso_compute_pre_Mumford} as
$$
\delta_{ij}=[Z'_{ij}]_{Z_i\times Z_j}^2
$$
Here, 
$[Z'_{ij}]_{Z_i\times Z_j}^2 \in H_{\T}^0(Z_i \times_S Z_j,p_2^{-1}\twist{T_X^j})$ is the twisted fundamental class
of $\nu'_{ij}$ with respect to the obvious $\nu'_{ij}$-parallelization. 
\end{enumerate}
\end{prop}
\begin{proof}
The first statement follows from the definition of the explicit duality pairing given in \Cref{ex:duality},
and the properties of fundamental classes. 
For compatibility with composition and transversal base change formula for closed immersions, 
see \cite[Lemma 3.2.13, Ex. 3.2.9(i)]{DJK}.
 The second (resp. third) computation follows from the first one and the excess intersection (resp. transversal base change) formula for the above cartesian square.
\end{proof}

\begin{ex}
When $\T$ is an oriented motivic category, 
i.e., 
one of the categories under $\DM$ in \eqref{eq:motivic_cat}, 
and we assume that the second condition of the proposition holds,
then $\delta_{ij}=[Z'_{ij}]_{Z_i\times Z_j}$ is the image of the usual cycle class of the natural 
diagonal immersion of $Z'_{ij}$ by the cycle class map
$$
\CH^d(Z_i \times_S Z_j) \rightarrow H_{\T}^{2d,d}(Z_i \times_S Z_j)
$$
where $d$ is the dimension of $X/S$.
 In particular, we get $\delta_{ij}=\delta_{ji}$ after making the identification
 $\CH^d(Z_i \times_S Z_j)=\CH^d(Z_j \times_S Z_i)$.
That is, the matrix in \Cref{thm:maincomputation} is \emph{symmetric}.
In the non-oriented case, this will no longer be true in general,
as we will illustrate in the forthcoming section.
\end{ex}

%

%% file: Mumford-rev.tex
This section describes how to compute punctured tubular neighborhoods in the two-dimensional case. We focus on the computation of the neighborhood at infinity of an arbitrary surface \(X_0\), after compactifying it to \(X\) with a normal crossing boundary \(D=\partial X_0\) (cf., e.g.,  \Cref{prop:basic_comput_thp-infty}). This process also applies to the punctured tubular neighborhood of singularities of normal surfaces over a perfect field. By taking a suitable resolution of singularities, we can reduce the situation to a log-pair \((X,D)\) and reference \Cref{cor:cdh-invariance}. In particular, for rational singularities, we will demonstrate that our framework enables us to provide a motivic version of Mumford's plumbing construction, as discussed in \cite{mumfordihes}.

Let us establish some notation for this section. Except in \Cref{sec:KthNCD}, we work over a base field \(k\) and within a motivic \(\infty\)-category \(\T\) (see \Cref{sec:notations}). We denote \(\htp=\htp_k\) following the notation in \Cref{df:internal_theories}. Our primary cases are \(\T=\SH\) and \(\T=\DM\). Recall that for a smooth \(k\)-scheme \(X\), \(\htp(X)=\Sigma^\infty X_+\) and \(\htp(X)=M(X)\).

\subsection{K-theory and Picard groups of normal crossing divisors}\label{sec:KthNCD}

\begin{num}
Given an arbitrary scheme $X$, one can define its Thomason-Trobaugh K-theory spectrum $K(X)$
 and this defines presheaf of $S^1$-spectrum on the category of qcqs schemes $\Sch$.
 According to \cite{Thomason}, it satisfies Nisnevich descent and therefore defines
 an object $K \in \Sh_\nis(\Sch,\Sp)$ where $\Sp$ is the $\infty$-category of $S^1$-spectra.\footnote{This
 is also the \emph{stabilization} of the big Nisnevich $\infty$-topos with base site $\Sch$.}
 According to \cite[Th. 6.3]{KST}, Weibel's homotopy invariant K-theory $KH$
 can be defined as the $\cdh$-localisation
 of the sheaf $K$, and we will put $K^{\cdh}=KH:=L_{\cdh} K$.

One then considers the adjunction
$$
L_{\cdh}:\Sh_\nis(\Sch) \rightleftarrows \SH_\cdh(\Sch):\cO_{\cdh}
$$
where $L_{\cdh}$ is the $\infty$-categorical associated $\cdh$-sheaf functor.
 Both the above homotopy categories are equipped with standard $t$-structures,
 whose heart are made respectively of Nisnevich and $\cdh$ sheaves of abelian groups,
 and whose towers of truncations, with homological conventions,
$$
\hdots \rightarrow \tau_{\leq n} \rightarrow \tau_{\leq n+1} \rightarrow \hdots
$$
correspond to the ($S^1$-stable) Postnikov tower. Associated with this tower
 applied to $K$ or $KH$, and using the cohomological functor
$\pi_{-p-q}(\Map(\Sigma^\infty X_+,.))$
for a scheme $X$, we get the $t$-descent spectral sequences, $t=\nis, \cdh$
$$
E_{2,t}^{p,q}=H^p_t(X,a_t\pi_{-q}(K)) \Rightarrow K^t_{-p-q}(X)
$$
This is classical (see, e.g., \cite{BG, KST}). Note that the form
 of the $E_2$-term in the $\cdh$-local case follows as
 the functor $L_{\cdh}$ is $t$-exact. Using this fact again, one gets a canonical morphism of towers, induced
 by the unit map of the adjunction $(\cO_{\cdh},L_{\cdh})$
$$
\tau_{\leq -p-q} K \rightarrow \cO_{\cdh}\big(\tau_{\leq -p-q} KH\big)
$$
This gives a canonical morphism of spectral sequences induced by the canonical morphism deduced from $\cdh$-sheafification
$$
E_{2,\nis}^{p,q}=H^p_\nis(X,a_\nis\pi_{-q}(K)) \rightarrow H^p_\cdh(X,a_\cdh\pi_{-q}(K))=E_{2,\cdh}^{p,q}
$$
Given these considerations,
 we will define the $\cdh$-local Picard group of $X$ as the isomorphism group of
 $\cdh$-locally trivial torsors over $X$ under the group $\GG$
$$
\Pic_{\cdh}(X)=H^1_{\cdh}(X,\GG)
$$
\end{num}
\begin{prop}
Let $X$ be a one-dimensional scheme, 
and $\pi_0(X)$ be the (finite) set of its connected components.
Then, 
there exists a commutative diagram of abelian groups, 
in which each horizontal line is an exact sequence
$$
\xymatrix@C=20pt@R=10pt{
0\ar[r] & \Pic(X)\ar[r]\ar[d] & K_0(X)\ar^{\rk}[r]\ar[d] & \ZZ^{\pi_0(X)}\ar[r]\ar@{=}[d] & 0 \\
0\ar[r] & \Pic_{\cdh}(X)\ar[r] & KH_0(X)\ar^{\rk}[r] & \ZZ^{\pi_0(X)}\ar[r] & 0
}
$$
Both exact sequences are split by the determinant functors $K_0(X) \xrightarrow{\det} \Pic(X)$
 and $KH_0(X) \xrightarrow{\det_{\cdh}} \Pic_{\cdh}(X)$ respectively.
\end{prop}
\begin{proof}
We apply the $t$-descent spectral sequences mentioned earlier. 
Since both the Nisnevich and $\cdh$-topologies on $X$ have cohomological dimensions 
less than or equal to $\dim X = 1$, both spectral sequences are concentrated in the 
lines $p = 0$ and $p = 1$. 
In particular, they degenerate at $E_2$ and induce two short exact sequences that are 
functorially related.
Next, we use the identification of the Nisnevich 
sheafification 
$$
 a_\nis \pi_i K = 
\begin{cases}
\ZZ & \text{if } i = 0 \\
\GG & \text{if } i = 1
\end{cases}
$$
This result relies on the observation that for a local ring $R$, 
we have $K_0(R) = \ZZ$ and $K_1(R) = R^\times$ (see \cite[III, Lemma 1.4]{Weibel} 
for the latter statement). 
Additionally, 
the Nisnevich local sheaf represented by $\ZZ$ on $\text{Sch}$ is also a $\cdh$-sheaf, 
which allows us to correctly represent the diagram as stated in the proposition.
Finally, we recall that the determinant is induced by the canonical map 
$$
\tilde \det: K \rightarrow \mathrm{B} \GG
$$
which is a morphism of Nisnevich sheaves of $S^1$-spectra on $\text{Sch}$. 
The $t$-descent spectral sequences are functorial with respect to this morphism, 
demonstrating that $\tilde \det$ induces the desired splitting.
\end{proof}

According to \Cref{thm:SH-orientations}, one derives the following key result
 (see also \cite{RonTheta} for the first occurrence of this kind of fact).

\begin{cor}\label{cor:SL^c-orientability-curves}
Let $X$ be a scheme of dimension one.
 Then the Thom space $\Th_X(v) \in \ho \SH(X)$ of a virtual vector bundle $v$ depends only
 on the rank and determinant of $v$.
 In particular, an orientation\footnote{See \Cref{subsection:ovbaqi} in the Appendix.} $\epsilon \in \Or_X(v)=\Or_X(\det v)$ induces a canonical isomorphism
$$
\epsilon_*:\Th_X(v) \xrightarrow \simeq \un_X(r)[2r], r=\rk v
$$
\end{cor}

\begin{rem}
To put it differently, we discover the surprising fact that when we restrict ourselves to virtual vector bundles over one-dimensional schemes, motivic ring spectra are always canonically $\SL^c$-oriented as defined by Panin and Walter (see \cite{PW1}).
\end{rem}

Note that the above corollary holds for possibly singular schemes.
 The next result will help us understand orientations of line bundles
 in the case of normal crossing singularities.

\begin{thm}\label{thm:Pic_NCD}
Let $D$ be a reduced scheme with finitely many irreducible components $D=\cup_{i \in I} D_i$.
 We use the notation of \ref{num:normal_crossing}, for $Z=D=S$. In particular,
 we assume the set of indices $I$ is linearly ordered.
 Assume that for all $J \subset I$,
 $D_J=(D'_J)_{red}$ is $0$-dimensional when $\sharp J=2$,
 and empty when $\sharp J>2$.

Then, there is a commutative diagram with exact rows of the form
$$
\xymatrix@R=14pt@C=18pt{
0\ar[r] & \GG(D)\ar^-{\sum_i \nu_i^*}[r]\ar[d]
 & \bigoplus_{i \in I} \GG(D_i)\ar^-{\phi}[r]\ar_{(2)}[d]
 & \bigoplus_{i<j} \GG(D'_{ij})\ar[r]\ar^{(3)}[d]
 & \Pic(D)\ar^-{\sum_i \nu_i^*}[r]\ar[d]
 &\bigoplus_{i\in I} \Pic(D_i)\ar[r]\ar^{(5)}[d]
 & 0 \\
0\ar[r] & H^0_\cdh(D,\GG)\ar^-{\sum_i \nu_i^*}[r] & \bigoplus_{i \in I} H^0_\cdh(D_i,\GG)\ar[r]
 & \bigoplus_{i<j, x} \kappa_{ij}(x)^\times\ar[r]
 & \Pic_\cdh(D)\ar^-{\sum_i \nu_i^*}[r]
 & \bigoplus_{i\in I}\Pic_\cdh(D_i)\ar[r]
 & 0
}
$$
where $x$ runs over the points of the $0$-dimensional scheme $D_{ij}$,
 $\kappa_{ij}(x)$ being the associated residue field,
 the vertical maps are the natural arrows obtained via $\cdh$-sheafification,
 and we define
$$
\phi:(u_i)_{i \in I} \mapsto \sum_{i<j} u_i|_{D'_{ij}}.(u_j|_{D'_{ij}})^{-1}
$$
Moreover, if all the $D_i$ are regular schemes, the maps $(2)$ and $(5)$
 are isomorphisms.
\end{thm}
\begin{proof}
We start with the following lemma.
\begin{lm}\label{lm:exact-NCD-Gm}
Under the assumptions of the previous theorem,
 the following sequence of Zariski sheaves on $D$ is exact
$$
\xymatrix@C=34pt{
0\ar[r] & \GGx D\ar^-{\sum_i \nu_i^*}[r]
 & \bigoplus_{i \in I} \nu_{i*}(\GGx{D_i})\ar^-{\phi}[r]
 & \bigoplus_{i<j} \nu_{ij*}\left(\GGx{D'_{ij}}\right)\ar[r]
 & 0
}
$$
where $\GGx D$ denotes the Zariski sheaf on $D$ obtained by restriction,
 and $\phi$ is defined as in the statement of the theorem.
\end{lm}

We demonstrate the exactness on stalks at a point \( x \in X \). If \( x \) does not belong to any of the \( D_{ij} \), then it belongs to a single component \( D_i \), and the exactness is evident. If \( x \in D_{ij} \), we consider the case of a local reduced ring \( A \) with \( D = \Spec(A) \), and two integral components \( D_1 = \Spec(A/I) \) and \( D_2 = \Spec(A/J) \) — in particular, \( I \cap J = 0 \).

We reduce the problem to demonstrating the exactness of the sequence
$$
0 \rightarrow A^\times \rightarrow (A/I)^\times \oplus (A/J)^\times \rightarrow (A/(I+J))^\times \rightarrow 0
$$
which is now an exercise in commutative algebra.

The lemma immediately produces the top exact sequence from the diagram stated in the theorem, given that for any \( 0 \)-dimensional scheme \( X \), it holds that \( \Pic(X) = 0 \).

To obtain the complete diagram, we consider the embedding of Zariski sites: \( \rho: D_\zar \rightarrow \Sch^{pf}_D \), where \( \Sch^{pf}_D \) is the category of finitely presented \( D \)-schemes. This induces an adjunction between the respective categories of abelian Zariski sheaves
$$
\rho_\sharp: \Sh_\zar(D, \ZZ) \rightleftarrows \Sh_\zar(\Sch^{pf}_D, \ZZ): \rho^*
$$
where \( \rho^* \) is the restriction functor. Recall from sheaf theory 
(see e.g., \cite[VII, \textsection 4.0]{SGA4II}) that \( \rho_\sharp \) is fully faithful and exact, while \( \rho^* \) is exact.

We denote by \( \uGGx D \) the sheaf represented by \( \GG \) on the big Zariski site \( \Sch^{pf}_D \), such that \( \rho^* \uGGx D = \GGx D \). Then, the exactness of the sequence from the above lemma is equivalent to the exactness of the following sequence
$$
\xymatrix@C=34pt{
0 \ar[r] & \rho^*(\uGGx D) \ar^-{\sum_i \nu_i^*}[r] & \bigoplus_{i \in I} \rho^*\nu_{i*}(\uGGx{D_i}) \ar^-{\phi}[r] & \bigoplus_{i<j} \rho^*\nu_{ij*}\left(\uGGx{D'_{ij}}\right) \ar[r] & 0
}
$$

We are working within the derived $\infty$-category $\Der(\Sh_\zar(\Sch_D^{pf},\ZZ))$. 
There are adjunctions of $\infty$-functors given by
\[
\rho_\sharp: \Der(\Sh_\zar(D, \ZZ)) \rightleftarrows \Der(\Sh_\zar(\Sch^{pf}_D, \ZZ)): \rho^*
\]
\[
a_{\cdh}: \Der(\Sh_\zar(\Sch^{pf}_D, \ZZ)) \rightleftarrows \Der(\Sh_\cdh(\Sch^{pf}_D, \ZZ)): \cO_\cdh
\]
It is important to note that all the preceding functors are either left or right derived functors, which is particularly relevant for $\cO_\cdh$. Next, we will consider the following diagram in the $\infty$-category $\Der(\Sh(\Sch_D^{pf}, \ZZ))$
$$
\xymatrix@C=34pt@R=14pt{
\rho_\sharp\rho^*(\uGGx D)\ar^-{\sum_i \nu_i^*}[r]\ar[d]
 & \bigoplus_{i \in I} \rho_\sharp\rho^*\nu_{i*}(\uGGx{D_i})\ar^-{\phi}[r]\ar[d]
 & \bigoplus_{i<j} \rho_\sharp\rho^*\nu_{ij*}\left(\uGGx{D'_{ij}}\right)\ar[d] \\
(\uGGx D)\ar^-{\sum_i \nu_i^*}[r]\ar[d]
 & \bigoplus_{i \in I} \nu_{i*}(\uGGx{D_i})\ar^-{\phi}[r]\ar[d]
 & \bigoplus_{i<j} \nu_{ij*}\left(\uGGx{D'_{ij}}\right)\ar[d] \\
\cO_{\cdh} a_\cdh(\uGGx D)\ar^-{\sum_i \nu_i^*}[r]
 & \bigoplus_{i \in I} \cO_{\cdh} a_\cdh\nu_{i*}(\uGGx{D_i})\ar^-{\phi}[r]
 & \bigoplus_{i<j} \cO_{\cdh} a_\cdh\nu_{ij*}\left(\uGGx{D'_{ij}}\right)
}
$$
Here, the vertical maps between the first and second rows represent the obvious counit map, while the vertical maps between the second and third rows correspond to the unit map. Consequently, the diagram is commutative.
Based on what we have just discussed, the top row is homotopy exact. By applying $\cdh$-descent, 
we conclude that the bottom row is also homotopy exact. 
The result follows from applying the functor \( H_0\Map(\mathbb{Z}(D), -) \).
In particular, the $\cdh$-topology does not detect nilpotents, 
which leads to the specific form of the map's target mentioned in point (3).
\end{proof}

\begin{rem}
In characteristic \(0\), the bottom exact sequence mentioned in the previous statement can be matched with the exact sequence in \((\cdh\)-local) motivic cohomology derived from \Cref{ex:sncexample1}. The same is true in characteristic \(p\) after tensoring with \(\Lambda = \mathbb{Z}[1/p]\), and for arbitrary schemes \(D\) as described above, after inverting \(\Lambda = \mathbb{Q}\). This result can be explained either by the representability of motivic cohomology within motivic stable homotopy theory, or by the existence of the motivic \(\infty\)-category \(\DM(-, \Lambda)\). Furthermore, it holds that \((\cdh\)-local) motivic cohomology satisfies the 
relation 
\[
H_M^{i,1}(D,\Lambda) = H^{i-1}_\cdh(D,\GG) \otimes \Lambda
\]
under the appropriate assumptions on \(D\) and \(\Lambda\).
\end{rem}

Here is a simple application relevant to the study
 of singularities of normal surfaces.
\begin{cor}
Let $D$ be a simple normal crossing divisor
 in a regular $2$-dimensional scheme $X$.
 Then, the maps induced by $\cdh$-sheafification
\begin{align*}
\GG(D) & \rightarrow H^0_\cdh(D,\GG) \\
\Pic(D) & \rightarrow \Pic_\cdh(D) \\
K_0(D) & \rightarrow KH_0(D)
\end{align*}
are isomorphisms. Moreover, the following sequence is exact
$$
0 \rightarrow \GG(D)\xrightarrow{\sum_i \nu_i^*}
 \bigoplus_{i \in I} \GG(D_i)\xrightarrow{\sum_{i<j} \nu_i^*(\nu_j^*)^{-1}}
 \bigoplus_{i<j} \GG(D_{ij})\rightarrow
 \Pic(D)\xrightarrow{\sum_{i} \nu_i^*}
 \Pic(D_i) \rightarrow 0
$$
\end{cor}

\begin{notation}\label{not:dual-graph}
Recall that the \emph{dual graph} $\Delta$ of a proper simple normal crossing divisor $D$
 in a smooth algebraic $k$-surface is the (finite) cell complex with a vertex $x_i$
 for each irreducible components $D_i$ of $D$,
 and with a cell of dimension $1$ attached at $x_i$ and $x_j$ for each point
 of $D_{ij}$.
\end{notation}
\begin{cor}\label{cor:Pic-proper-SNC}
Let $D$ be a proper simple normal crossing divisor
 in a smooth $2$-dimensional scheme $X$ over a field $k$.
 Assume that the intersections of the $D_i$ are $k$-rational points
 and let $\Delta$ be the dual graph of $D$.

Then, there exists an isomorphism
$$
H^0_\cdh(D,\GG)=\GG(D) \simeq H^0(\Delta,k^\times) \simeq (k^\times)^{\pi_0(D)}
$$
and a short exact sequence
$$
0 \rightarrow H^1(\Delta,k^\times)
 \rightarrow \Pic(D) \rightarrow \oplus_{i \in I} \Pic(D_i) \rightarrow 0
$$
In particular, if $\Delta$ is simply connected, the restrictions
 to the branches $D_i$ of $D$ induce an isomorphism
$$
\Pic(D) \xrightarrow \simeq \oplus_{i \in I} \Pic(D_i).
$$ 
\end{cor}
\begin{proof}
Indeed, the assumptions imply that the cell cohomological complex
 $C^*(\Delta,k^\times)$ associated with the obvious cellular
 structure of $\Delta$ is isomorphic to the complex 
$$
\bigoplus_{i \in I} \GG(D_i)\xrightarrow{\sum_{i<j} \nu_i^*(\nu_j^*)^{-1}}
 \bigoplus_{i<j} \GG(D_{ij})
$$
concentrated in cohomological degree $[0,1]$.
\end{proof}
 
\begin{ex}
\begin{enumerate}
\item \textit{Multiplicities}. Let
 $D=(D_1 \cup D_2) \subset \PP^2_k$, with homogeneous coordinates $x,y,z$
 where $D_1$ is the projective line $V(y)$ and $D_2$ is the irreducible conic $V(yz-x^2)$.
 In other words, $D$ is the union
 of two rational curves with a single intersection of multiplicity $2$ at the point $[0:0:1]$.
 Then one gets from \Cref{thm:Pic_NCD} that
 $\Pic(D)=k \oplus (\ZZ \times \ZZ)$, and $\Pic_\cdh(D)=\ZZ \times \ZZ$.
%
\item \textit{Non-rational intersections}. Let $D=(D_1 \cup D_2) \subset \PP^2_\RR$ such that
 $D_1=V(x)$, $D_2=V(x^2+y^2+z^2)$.
 Then, one gets from the first corollary a (split) short exact sequence
$$
0 \rightarrow \CC^\times/\RR^\times \rightarrow \Pic(D) \rightarrow \ZZ \times  \ZZ\rightarrow 0
$$
Moreover, $\Pic(D) \simeq \Pic_\cdh(D)$. In conclusion, we deduce that both groups
 account for the non-real intersections of the branches of $D$.
\end{enumerate}
\end{ex}

An important application of the preceding results concerns orientations  of line bundles over normal crossing divisors on a surface. We begin with the following: 
\begin{prop}
\label{thm:orientationncdsurfaces}
Consider the assumptions of \Cref{thm:Pic_NCD}.
 Let $\cL$ be an invertible sheaf on $D$, and write
 $\cL_i=\cL |_{D_i}$, $\cL_{ij}=\cL|{D'_{ij}}$.
 Then in the following diagram of sets 
 \begin{equation}\label{prop:glueing-or1}
\xymatrix@=30pt{
\Or_D(\cL)\ar^-{\prod_i \nu_i^*}[r]
 & \prod_{i \in I} \Or_{D_i}(\cL_i)\ar@<2pt>^-{\prod_{i<j} \nu_{ij}^{i*}}[rr]
\ar@<-2pt>_-{\prod_{i<j} \nu_{ij}^{j*}}[rr]
 && \prod_{i<j} \Or_{D'_{ij}}(\cL_{ij})
}
\end{equation}
the first map is surjective on the equalizer of the last two maps.

Moreover, assuming that the hypothesis in \Cref{cor:Pic-proper-SNC} holds true, and that the dual graph \(\Delta\) of \(D\) in \(X\) is simply connected, then diagram \eqref{prop:glueing-or1} is exact. 

\end{prop}
\begin{proof}
 Recall \Cref{subsection:ovbaqi} that when $\Or_{D}(\cL)$ is nonempty, the morphism $\prod_{i\in I} \nu_i^*:\Or_{D}(\cL)\to \prod_{i\in I} \Or_{D_i}(\cL_i)$ is defined by mapping an orientation class of $\cL$
represented by an isomorphism $\epsilon:\cL\to \cM^{\otimes 2}$ for some invertible sheaf $\cM$ on $D$, to the product of the classes in $\Or_{D_i}(\cL_i)$ of the isomorphisms  $\epsilon_i:\cL_i\to \cM|_{D_i}^{\otimes 2}$ induced by the restrictions of $\epsilon$. The two  right-hand side arrows in the statement are defined in a similar way. Since the restrictions $\epsilon_{ij}:\cL_{ij}\to \cM|_{D'_{ij}}^{\otimes 2}$ of $\epsilon$ satisfy the identities $ \epsilon_i|_{D'_{ij}}=\epsilon_j|_{D'_{ij}}$ for all $i<j$,
it follows that the map $\prod_{i\in I}\nu_i^*$ factors through the equalizer $E(\cL)\subset \prod_{i\in I} \Or_{D_i}(\cL_i)$ of the two right-hand side arrows. 

Now assume given an element $e$ of $E(\cL)$, represented by a collection of isomorphisms $\epsilon_i:\cL_i\to \cM_i^{\otimes 2}$ for some invertible sheaves $\cM_i$ on $D_i$ such that for all $i<j$ the induced orientations $\epsilon_i|_{D'_{ij}}:\cL_{ij}\to \cM_i|_{D'_{ij}}^{\otimes 2}$ and $\epsilon_j|_{D'_{ij}}:\cL_{ij}\to \cM_j|_{D'_{ij}}^{\otimes 2}$ are equivalent. In view of \Cref{thm:Pic_NCD}, up to replacing all the orientations $\epsilon_i$ by equivalent ones, we can assume without loss of generality from the very beginning that $\cM_i=\cM|_{D_i}$ for some invertible sheaf $\cM$ on $D$. 
The assumption that the orientations $\epsilon_i|_{D'_{ij}}$ and $\epsilon_j|_{D'_{ij}}$ are equivalent then determines a collection of elements $$u_{ij}((\epsilon_i)_{i\in I})\in \mathrm{Isom}_{D'_{ij}}(\cM|_{D'_{ij}},\cM|_{D'_{ij}})\cong \GG(D'_{ij}),\quad i<j.$$ Applying \Cref{thm:Pic_NCD} again, this collection determines an invertible sheaf $\cN$ on $D$ with isomorphisms $\alpha_i:\cO_{D_i}\to \cN_i=\cN|_{D_i}$ for every $i\in I$ such that $\alpha_j|_{D'_{ij}}\circ \alpha_i^{-1}|_{D'_{ij}}$ is the multiplication by $u_{ij}((\epsilon_{i})_{i\in I})$. 
Let $\cM'=\cN^\vee\otimes \cM$ and let $\cM'_i=\cM'_{D_i}$. Then the collection of orientations  $\epsilon_i'=(^t\alpha_i^{-1})^{\otimes 2}\circ \epsilon_i:\cL_i\to \cM_i^{\otimes 2}\to  (\cM'_i)^{\otimes 2}$ is equivalent to the collection $(\epsilon_i)_{i\in I}$, whence represents the element $e$, and satisfies $u_{ij}((\epsilon'_{i})_{i\in I})=\Id_{\cM'|_{D'_{ij}}}$. The latter property means that the isomorphisms $\epsilon'_i$ coincide on the intersections $D'_{ij}$, whence glue an isomorphism $\epsilon':\cL\to (\cM')^{\otimes 2}$ of invertible sheaves on $D$ whose restriction on each $D_i$ equals $\epsilon'_i$. This shows that the map $\prod_{i\in I}\nu_i^*:\Or_D(\cL)\to E(\cL)$ is surjective, as required.

We now prove the second assertion. It amounts to verify that under the additional assumptions, the map $\prod_i \nu_i^*: \Or_D(\cL)
  \prod_{i \in I} \Or_{D_i}(\cL_i)$ is injective, with image $E(\cL)$. The property is immediate when $\Or_D(\cL)=\empty$. So assume that $\Or_D(\cL)$ is nonempty, whence that  $\Or_{D_i}(\cL_i)\neq \emptyset$ for every $i\in I$. Since these sets are then principal homogeneous under the action of the groups $\Or_D(\cO_D)\cong H^1(D,\mu_2)$ and  $`\Or_{D_i}(\cO_{D_i})\cong H^i(D_i,\mu_2)$, we are reduced to the case where $\cL=\cO_D$ for which the assertion follows
from the long exact sequence 
{\small
$$
0\to H^0(D,\mu_2)\to \bigoplus_{i\in I} H^0(D_i,\mu_2)\to \bigoplus_{i<j} H^0(D'_{ij},\mu_2)\to H^1(D,\mu_2)\to \bigoplus_{i\in I} H^1(D_i,\mu_2)\to \bigoplus_{i<j} H^1(D_{ij},\mu_2)\to \cdots$$}
analogous to that in \Cref{thm:Pic_NCD}, which can be deduced from \Cref{lm:exact-NCD-Gm} and the identification of the kernel of the map $H^1(D,\mu_2)\to \bigoplus_{i\in I} H^1(D_i,\mu_2)$ with $H^1(\Delta, \mathbb{Z}/2\mathbb{Z})$ as in the proof  \Cref{cor:Pic-proper-SNC}.
\end{proof}

\begin{cor}
\label{thm:orientationncdsurfaces-bis}
Consider the assumptions of \Cref {thm:orientationncdsurfaces} and assume that the hypothesis in \Cref{cor:Pic-proper-SNC} holds true. Then an invertible sheaf $\cL$ on $D$ is orientable if and only if its restrictions $\cL_i=\cL_{D_i}$ are orientable for every $i\in I$. 
\end{cor}
\begin{proof}
One direction is immediate since every  orientation $o$ of $\cL$ induces by restriction orientations $\nu_i^*o$ of $\cL_i=\cL|_{D_i}$. 
Conversely, assume that $\cL_i$ is orientable for every $i\in I$. Then $\Or_{D_i}(\cL_i)$ becomes a principal homogeneous space under the action of $\Or{D_i}(\cO_{D_i})$, the choice of orientation classes $\epsilon_i\in \Or_{D_i}(\cL_i)$ gives isomorphisms $\psi_i:\Or_{D_i}(\cO_{D_i})\to \Or_{D_i}(\cL_i)$, $o_i \mapsto o_i\cdot \epsilon_i$. In particular, there exists a unique collection of orientation classes $\epsilon_{i}'\in \Or_{D_i}(\cL_i)$ such that for every $i\in I$, $\psi_i^{-1}(\epsilon_i')$ equal the neutral element of  $\Or_{D_i}(\cO_{D_i})$ (the class of the inverse of the multiplication map $m_i:\cO_{D_i}\otimes \cO_{D_i}\to \cO_{D_i}$). It is straightforward to check that the so-defined orientation classes $\epsilon_i'$ have the property that $\nu_{ij}^{i*}\epsilon_i'=\nu_ij^{j*}\epsilon_j'$ for all $i<j$. The conclusion then follows from \Cref{thm:orientationncdsurfaces}.
\end{proof}

\subsection{Theta characteristic of curves and homotopy type of NCD on surfaces}
\label{subsection:thetacharacteristic}

\begin{notation}\label{num:geom-orientations}
Let $X$ be a quasi-projective $k$-scheme with canonical sheaf
 $\omega_X=\det(\cL_{X/k})$.
 We will say that $X$ is \emph{orientable} if $\omega_X$ (or what amount to the same: the virtual bundle associated with its cotangent complex)
 is orientable in the sense of \Cref{df:orient_virtual}. In other words, the set of orientations $\Or_X(\omega_X)$ is not empty.
 When specializing this notion to a smooth projective curve $X=C$, it can be linked to the theory
 of Theta characteristics of $C$ (see \cite{MumTheta, AtiTheta}).
 In fact, a Theta characteristic of $C$ is precisely an orientation of $C$, 
i.e., a ``square-root'' of the canonical sheaf $\omega_C$.
 If we denote (as in \emph{op. cit}) by $\mathcal S(C)$ the set of Theta characteristics (up to isomorphisms), 
 we obtain the equality
$$
\mathcal S(C)=\Or(\omega_C)
$$
 
The following result is a slightly more precise version of a theorem due to 
R\"ondigs (see \cite{RonTheta}).
\end{notation}

\begin{prop}\label{prop:splitting-curves}
Consider a smooth projective curve $p:C \rightarrow \Spec(k)$ over the field $k$, 
with a rational point $x \in C(k)$.
We let $\calC_x$ be the conormal sheaf of $x$ in $X$, and let $\Theta:\calC_x \rightarrow \omega_C|_x$
be the canonical isomorphism. Then the following homotopy exact sequence in $\SH(k)$
$$
\htp(C-\{x\}) \xrightarrow{j_*} \htp(C) \xrightarrow{x^!} \htp(k,\langle \calC_x \rangle)
$$
is split if and only if $C$ is orientable, 
i.e., 
$C$ admits a Theta characteristic. 
Moreover, 
if $C$ is orientable, 
one gets a splitting by choosing a quadratic pre-isomorphism of invertible sheaves over $C$
$$
\Upsilon:p^{-1}(\calC_x) \qiso \omega_C
$$
such that $\Upsilon|_x$ is quadratically equivalent to $\Theta$. 
The following composite gives the splitting
$$
p^!_\Upsilon:\htp(k,\langle \calC_x \rangle) \xrightarrow{p^!} 
\htp(C,\langle p^{-1}\calC_x\rangle-\langle \omega_C\rangle)
\xrightarrow{\Upsilon_*} \htp(C)
$$
where we have identified $\Upsilon$ with the orientation class in $\Or_C(p^{-1}\calC_x \otimes \omega_C^\vee)$
obviously associated (see \Cref{ex:quad_iso}, \Cref{num:orientation_classes}),
and the isomorphism $\Upsilon_*$ follows from \Cref{cor:SL^c-orientability-curves}.
\end{prop}
\begin{proof}
Given the current advancements in motivic homotopy technology, 
we can provide a shorter proof than that presented in \cite{RonTheta}. 
For the "if" part, 
we leverage the compatibility of Gysin maps (\Cref{num:Gysin}) with compositions. 
We have the following identity
$$
\htp(k,\langle \mathcal{C}_x\rangle) \xrightarrow{p^!} \htp(D,\langle p^{-1}\mathcal{C}_x\rangle - \langle \omega_C) \rangle) \xrightarrow{x^!} 
\htp(k,(\langle \mathcal{C}_x\rangle -\langle \omega_{C,x}\rangle) + \langle \mathcal{C}_x\rangle) \xrightarrow{\varphi_*} \htp(k,\mathcal{C}_x)
$$
In this identity, the last isomorphism is induced by the functoriality with respect to isomorphisms 
of virtual bundles. 
The conclusion follows from the fact that $\omega_C$ being orientable is equivalent to the existence 
of a quadratic pre-isomorphism $\Upsilon: p^{-1}(\mathcal{C}_x) \cong \omega_C$. 
The condition on $\Upsilon|_x$ translates to the requirement that $x^! \circ p^!_\Upsilon = \text{Id}$.

For the "only if" part, 
we deduce from the assumption that the map $x_*: \GW(k,\mathcal{C}_x) \rightarrow \GW(C)$ 
is a split monomorphism. 
We can examine the $\mathcal{C}_x$-twisted symmetric bilinear form 
$\varphi: k \otimes_k k \rightarrow \mathcal{C}_x$, 
which is obtained by choosing an arbitrary trivialization of $\mathcal{C}_x$. 
The image $x_*(k,\varphi)$ yields a nontrivial symmetric bilinear form on $\omega_C^\vee$, 
as indicated by the identity $[x_*(\mathcal{O}_k)] = [\omega_C^\vee]$ in $K_0(C)$. 
\end{proof}

\begin{rem}
Consider an arbitrary smooth projective curve $C$ over $k$,
and suppose we are given two distinct rational points $x,x' \in C(k)$.
Then $x'_*:\un \rightarrow \htp(C-\{x\})$ is a direct factor, 
split by the projection so that one gets a decomposition
$$
\htp(C-\{x\})
\simeq
\un \oplus \mathrm A_{x,x'}(C)[1]
$$
One can call the stable homotopy type
 $\mathrm A_{x,x'}(C)$ the \emph{Albanese stable homotopy type} of $(C,x,x')$.
 Indeed, its realization via the motive functor $\SH(k) \rightarrow \DM(k)$
 is the homological Voevodsky motive $\underline{\mathrm{Alb}}(C)$,
 associated with the Albanese scheme of $C$ (seen here as the dual of the Jacobian of the pointed curve $(C,x)$).
It is important to note that this object exists even if the curve \( C \) is not oriented. However, if \( C \) is oriented, we can obtain a canonical decomposition 
\[
\Pi(C) \simeq \un \oplus \mathrm{A}_{x,x'}(C)[1] \oplus \un(1)[2]
\]
by first choosing a trivialization \( \mathcal{C}_x \simeq k \) and applying the previous proposition.
This decomposition maps to the homological Chow-Künneth decomposition of \( M(C) \) in \( \DM(k) \), 
as mentioned previously in \cite{RonTheta}. Compared to the aforementioned reference, 
we have only pointed out that the condition of orientation is not necessary to define 
the homotopical version of the (dual) Jacobian of the curve \( C \).
\end{rem}

\begin{notation}\label{num:normalization-splitting}
To specify a method for selecting the quadratic pre-isomorphism \(\Upsilon\) referenced in this proposition, we can proceed as follows: We start by choosing a uniformizer \(\pi_x\) for the point \(x\) in the scheme \(X\), that is a generator $\pi_x$ of the maximal ideal $\mathfrak{m}$ of the discrete valuation ring $\cO_{X,x}$.  This uniformizer determines an isomorphism of $k$-vector spaces \(\overline{\pi}_x: k \rightarrow \calC_x=\mathfrak{m}/\mathfrak{m}^2\) defined by mapping $1$ to the residue class of $\pi_x$. The selection of \(\Upsilon\) thus corresponds to choosing an orientation class \(\tau \in \Or_C(\omega_C)\) such that the restriction \(\tau|_x \in \Or_k(\omega_C|_x)\) is mapped to \(1\) by the following composite isomorphism
\begin{equation*}
\Or_k(\omega_C|_x) \xrightarrow{\Theta^{-1}} \Or_k(\calC_x) \xrightarrow{\overline{\pi}_x^{-1}} \Or_k(k) = Q(k^\times)
\end{equation*}
Here, the latter group represents the set of quadratic classes of units of \(k\) (see \Cref{prop:orient&Pic-mod2}). When an orientation class \(\tau\) satisfies this condition, we call it \emph{$\pi_x$-normalized}. It is important to note that if \(C\) is orientable, a \(\pi_x\)-normalized orientation class \(\tau\) can always be chosen, since the group \(Q(k)\) acts on \(\Or_C(\omega_C)\).

Once such a normalized orientation class \(\tau\) has been selected, we construct \(\Upsilon\) as follows. Specifically, \(\tau\) 
is represented an isomorphism $\tau: \omega_C\to \cL^{\otimes 2}$.
We can then derive \(\Upsilon^{-1}\) as the quadratic class of the following composite isomorphism

$$
\omega_C \xrightarrow{\tau} \cL^{\otimes 2} \xrightarrow{\Id \otimes p^*(\overline{\pi}_x)} \cL^{\otimes 2} \otimes p^{*}\calC_x
$$
\end{notation}

\begin{ex}
Consider \(D = \PP^1_k = \mathrm{Proj}(k[u,v])\), and let \(x\) be the rational point $[0:1]$. We choose \((u/v\)) as a uniformizer for \(x\) in \(D\).  In this case, we have a canonical isomorphism \(\omega_D = \cO_D(-2)\), and the obvious orientation $\tau$ given by the inverse of the canonical morphism \( \cO_D(-1)^{\otimes 2} \xrightarrow{\sim} \cO_D(-2)\) is \((u/v)\)-normalized in the sense described above.
\end{ex}

\begin{notation}\label{num:hyp_Mumford1}
For the next proposition,
we consider a proper curve $D$ with smooth reduced crossings over $k$ in the sense of \Cref{df:normal_crossing}.
We will use the same notation as in \Cref{num:normal_crossing}: 
$(D_i)_{i \in I}$ are the irreducible components of $D$,
 $D'_{ij}=D_i \times_X D_j$, $D_{ij}=(D'_{ij})_\text{red}$. We let
$$
\bar \nu_i:D_i \rightarrow X, \nu_{ij}^l:D_{ij} \rightarrow D_l, l=i,j
$$
be the obvious inclusions.
We assume that $D_i$ admits a rational point $x_i \in D_i(k)$ that will play the role of the point at infinity,
 disjoint of the other components:
 $x_i \notin \cup_{j \neq i} D_{ij}$.
We let $\omega_i$ be the canonical sheaf of the curve $D_i/k$, 
and $\calC_{x_i}$ be the conormal sheaf of the points $x_i$ in $D_i$.
For normalization purposes, 
it will be convenient to choose a uniformizer $\pi_i$ for the point $x_i\in D_i$ with associated isomorphism $\overline{\pi}_i:k \xrightarrow \simeq \calC_{x_i}$.
\end{notation}

\begin{prop}\label{prop:Mumford_source}
Consider the above notation.
 We let $\mathcal D$ be the homotopy cokernel of the double arrows
\begin{equation}
\label{eq:top_part_Mumford}
\xymatrix@=60pt{
\bigoplus_{i<j} \htp(D_{ij})\ar@<2pt>^-{\sum_{i<j} (\nu_{ij}^i)_*}[r]\ar@<-2pt>_-{\sum_{i<j} (\nu_{ij}^j)_*}[r]
 & \bigoplus_{i \in I} \htp(D_i-\{x_i\})
}
\end{equation}
Then there is a canonical homotopy exact sequence
$$
\htp(\mathcal D) \xrightarrow \alpha \htp(D) \xrightarrow \beta \bigoplus_{i \in I} \un(1)[2]
$$
whose right-hand side depends only on the choice of the uniformizers $(\pi_i)_{i \in I}$.

If $\T$ is orientable, the sequence does not depend on such a choice and admits the following (homotopy) splitting
$$
\bigoplus_i \un(1)[2] \xrightarrow{\sum_i p_i^!} \bigoplus_i \htp(D_i) \xrightarrow{\sum_i \nu_{i*}} \htp(D)
$$
where $p_i^!:\un(1)[2] \rightarrow \htp(D_i)$ is the (oriented) Gysin map (\Cref{num:Gysin}).
 In the general case, the sequence admits a splitting
 if each curve $D_i$ is orientable. Moreover, a choice of $\pi_i$-normalized orientations $\tau_i \in \Or_{D_i}(\omega_i)$
 (as defined in \Cref{num:normalization-splitting})
 induces a canonical (homotopy) splitting
$$
\bigoplus_i \un(1)[2] \xrightarrow{\sum_i p_i^!} \bigoplus_i \htp(D_i,1-\langle \omega_i \rangle) \xrightarrow{\sum_i \tau_{i*}} \bigoplus_i \htp(D_i)
 \xrightarrow{\sum_i \nu_{i*}} \htp(D)
$$
where $p_i^!$ is the Gysin map and $\tau_{i*}$ is the isomorphism deduced from \Cref{cor:SL^c-orientability-curves}.
\end{prop}
\begin{proof}
We first build the homotopy exact sequence by considering the following diagram in the $\infty$-category $\T$
\begin{equation}\label{eq:Mumford_source}
\xymatrix@C=54pt@R=28pt{
\bigoplus_{i<j} \htp(D_{ij})\ar@<2pt>^-{\sum_{i<j} (\bar \nu_{ij}^i)_*}[r]\ar@<-2pt>_-{\sum_{i<j} (\bar \nu_{ij}^j)_*}[r]\ar@{=}[d]
 & \bigoplus_{i \in I} \htp(D_i-\{x_i\})\ar^-\delta[r]\ar^{\sum_i j_{i*}}[d] & \htp(\mathcal D)\ar^\alpha[d] \\
\bigoplus_{i<j} \htp(D_{ij})\ar@<2pt>^-{\sum_{i<j} (\nu_{ij}^i)_*}[r]\ar@<-2pt>_-{\sum_{i<j} (\nu_{ij}^j)_*}[r]
 & \bigoplus_{i \in I} \htp(D_i)\ar^{\sum_i \overline{\pi}^{-1}_{i*}x_i^!}[d]\ar^-{\sum_i \nu_{i*}}[r]\ar@{}|{(3)}[rd] & \htp(D)\ar^\beta[d] \\
& \bigoplus_{i \in I} \un(1)[2]\ar@{=}[r] & \bigoplus_{i \in I} \un(1)[2]
}
\end{equation}
in which all rows and columns are exact, and we have used \Cref{ex:sncexample1} for the exactness of the middle row.
 Note that the left top square is well-defined because of the assumption on the $x_i$.
 The map $\overline{\pi}^{-1}_{i*}$ refers to the isomorphism $\htp(k,\langle \calC_{x_i}\rangle)=\Th(\calC_{x_i}) \rightarrow \un(1)[2]$ inferred from $\overline{\pi}_i$. Consequently, 
 the assertion regarding the splitting follows from \Cref{prop:splitting-curves}.
\end{proof}

\begin{notation}\label{num:Mumford_source}
In this section, 
we will clarify how to derive explicit isomorphisms from the previous proposition and simplify the notation.
If the motivic \( \infty \)-category \( \T \) is not orientable, we make the following choices
\begin{itemize}
\item A uniformizer $\pi_i$ of $x_i$ in $D_i$ with induced trivialization \( \bar{\pi}_i: k \xrightarrow{\sim} \mathcal{C}_{x_i} \).
\item An orientation class \( \tau_i \in \Or_{D_i}(\omega_i) \) that is \( \pi_i \)-normalized, 
as defined in \Cref{num:normalization-splitting}.
\end{itemize}
With these choices established, 
we will use the following definitions for the Gysin maps for any index 
\( i \in I \)
\begin{align*}
x_i^!: & \ \htp(D_i) \rightarrow \htp(k, \mathcal{C}_{x_i}) \xrightarrow{\bar{\pi}^{-1}_{i*}} \un(1)[2] \\
p_i^!: & \ \un(1)[2] \rightarrow \htp(D_i, 1 - \omega_i) \xrightarrow{\tau^{-1}_{i*}} \htp(D_i)
\end{align*}
In each composite, the first map is the true (twisted) Gysin map.

When the space \(\T\) is oriented, 
both maps involved are well-defined and canonical, as they are normalized by the choice 
of orientation of \(\T\). We will now examine the maps \(\alpha\) and \(\beta\) as defined 
by diagram \eqref{eq:Mumford_source}. 
It is important to note that \(\beta\) is uniquely defined (up to homotopy) by the relations 
for all \(i \in I\)
\[
\beta \circ \nu_{i*} = x_i^!
\]
This follows from Part (3) of the diagram above and our preceding convention. 
Additionally, the previous proposition provides a splitting of \(\beta\) through the map
\[
\delta = \sum_{i \in I} \nu_{i*} p_i^! : \sum_i \un(1)[2] \rightarrow \htp(D)
\]
In particular, 
\(\beta \circ \delta\) is a homotopy idempotent in \(\htp(D)\). 
We will also consider the map
\[
\gamma = \big[1 - (\beta \circ \delta)\big]|^\cD : \htp(D) \rightarrow \cD
\]
Finally, 
we obtain canonical reciprocal isomorphisms in the homotopy category of \(\T(k)\)
$$
\xymatrix@=30pt{
\htp(D) \ar@<-2pt>_-{\begin{pmatrix}\gamma \\ \beta\end{pmatrix}}[r] & \cD \bigoplus \oplus_{i \in I} \un(1)[2]\ar@<-2pt>_-{(\alpha,\delta)}[l]
}
$$
\end{notation}

\begin{rem}\label{rem:compute-coker-double}
Being stable, the $\infty$-category $\T(k)$ is automatically additive (see \cite[Lem. 1.1.2.9]{LurieHA}).
 Therefore, considering two arrows
 $\xymatrix@C=25pt{M\ar@<2pt>^-{f}[r]\ar@<-2pt>_-{g}[r] & N}$ as in the previous statement,
 one can define the new morphism $(f-g):M \rightarrow N$.
 Moreover, it follows from their respective universal properties that one has an
 identification (up to a contractible set of choices) $\mathrm{coKer}(f,g)=\mathrm{coKer}(f-g)$,
 between the homotopy cokernel of the double arrows (equivalently the homotopy pullback)
 and the homotopy cofiber of their difference.
 Coming back to the assumptions of the above proposition, the above remark shows that
  the object $\mathcal D$ is the homotopy cofiber of the map:
$$
q=\sum_{i<j} \big((\nu_{ij}^i)_*-(\nu_{ij}^j)_*):\bigoplus_{i<j} \htp(D_{ij}\big)
 \rightarrow \bigoplus_{i \in I} \htp(D_i-\{x_i\}).
$$
\end{rem}

\begin{rem}\label{rem:Mumford_Artin_part}
If we assume that all the \( D_i \) are rational curves, then \( \mathcal{D} \) is an Artin-Tate object. 
In the more general case, 
by adding an additional rational point \( x'_i \) to each \( D_i \), \( \mathcal{D} \) 
will include a component reflecting the homotopy type of a dual Jacobian part. 
More precisely, 
\( \mathcal{D} \) can be described as the homotopy cokernel of a double arrow 
 (or the homotopy cofiber of their differences according to the previous remark)
 of the following form
\[
\xymatrix@=40pt{
\bigoplus_{i<j} \htp(D_{ij})\ar@<2pt>[r]\ar@<-2pt>[r]
& \big(\oplus_i \un\big) \oplus \left(\oplus_i \mathrm{A}_{x_i,x'_i}[1]\right)
}
\]
We note that both arrows in this diagram are explicitly computable 
within the framework of \( \text{SH}(k) \).
\end{rem}

\subsection{Punctured tubular neighborhoods and quadratic Mumford matrices}
\label{sec:general_Mumford}

\begin{notation}\label{num:hyp_Mumford2}
Consider a closed pair \((X, D)\) consisting of a smooth surface \(X\) over a field \(k\), along with a normal crossing divisor \(D\) in \(X\) that is proper over \(k\). 

We will refer to this pair as a \emph{log-pair over \(k\)}. Additionally, as stated in \Cref{num:hyp_Mumford1}, we assume that for all \(i \in I\), the component \(D_i\) has a rational point \(x_i \in D_i(k)\) that does not belong to any other components of \(D\). This assumption is not necessary for the next lemma, but it will be crucial for the subsequent theorem.

We denote by \(T_X = \mathbb{V}(\Omega_X)\) the tangent bundle of \(X\) and by \(\omega_X = \det(\Omega_X)\) the canonical sheaf of \(X\). For each \(i \in I\), we denote the conormal sheaf of \(D_i\) in \(X\) by \(\mathcal{C}_i\) and the associated normal bundle by \(N_i = \mathbb{V}(\mathcal{C}_i)\). The canonical sheaf of \(D_i\) is denoted by \(\omega_i\). Since \(D_i\) is smooth over \(k\), there exists a canonical isomorphism of invertible sheaves on \(D_i\)
\begin{equation}\label{eq:hyp_Mumford2}
\omega_X|_{D_i} \simeq \omega_i \otimes \mathcal C_i
\end{equation}

The following lemma is immediate from the results we obtained previously.
\end{notation}

\begin{lm}\label{lm:Mumford_target}
Consider the notation established previously. We assume either that the sheaf $\T$ is orientable, or that the restriction $\omega_X|_D$ is orientable in the sense defined in \Cref{num:orientation_classes}.

For any orientation class $\epsilon \in \Or_D(\omega_X|_D)$, there exists a canonical composite isomorphism in $\T(k)$ given by
$$
\htp(X/X-D) \xrightarrow{\Theta} \htp(D,-T_X|_D)^\vee \xrightarrow{(\epsilon^{-1}_*)^\vee} \big(\htp(D)(-2)[-4]\big)^\vee = \htp(D)^\vee(2)[4]
$$
Here, the first map is the isomorphism described in \eqref{eq:generalizedhomotopypurity2}, which is derived from \Cref{thm:generalizedhomotopypurity2}. The second map is induced by the orientation class $\epsilon^{-1}$ of $\omega^{-1}_X|_D$, according to \Cref{cor:SL^c-orientability-curves}.
\end{lm}


Combining the previous lemma and the computation of the previous section,
 we get the following result, which is the main theorem of this section
 and can be thought of as a stable motivic homotopical interpretation of
 the computation obtained by Mumford in \cite{mumfordihes} via his
 \emph{plumbing construction}.
\begin{thm}
\label{thm:smhatoomatrix-general}
Consider the assumptions of \Cref{num:hyp_Mumford2} for the log-pair $(X,D)$ over $k$.
 We further assume one of the following conditions.
\begin{enumerate}
\item $\T$ is orientable.
\item The invertible sheaves $\omega_X|_D$ over $D$, and $\omega_i$ over $D_i$ for any $i \in I$, are orientable.
 In this case, we choose an arbitrary orientation class $\epsilon \in \Or_D(\omega_X|_D)$,
 and for each $i\in I$, a $\pi_i$-normalized orientation class $\tau_i \in \Or_{D_i}(\omega_i)$, where $\pi_i$ is any uniformizing parameter of the local ring $\cO_{D_i,x_i}$ 
 (see \Cref{num:normalization-splitting}).
\end{enumerate}
 Then the punctured tubular neighborhood $\TN_k(X,D)$ in $\ho \T(k)$
 --- or equivalently when $X$ is proper (\Cref{prop:basic_comput_thp-infty}), the homotopy at infinity $\htp^\infty_k(X-D)$ ---
 is isomorphic to the cone of the map
$$
\beta'=\begin{pmatrix}
a & b' \\
b & \mu
\end{pmatrix}:
\htp(\mathcal D) \oplus \bigoplus_{i \in I} \un_k(1)[2] 
\rightarrow 
\htp(\mathcal D)^\vee(2)[4] \oplus \bigoplus_{j \in I} \un_k(1)[2] 
$$
where $\mathcal D$ was defined in \Cref{prop:Mumford_source}.
In \Cref{num:Mumford_source}, we have 
\begin{align*}
a&=\alpha^\vee (\epsilon^{-1}_*)^\vee \Theta \beta_{X,Z} \alpha \\
b&=\alpha^\vee (\epsilon^{-1}_*)^\vee \Theta \beta_{X,Z} \delta \\
b'&=\delta^\vee (\epsilon^{-1}_*)^\vee \Theta \beta_{X,Z} \alpha \\
\mu&=\delta^\vee (\epsilon^{-1}_*)^\vee \Theta \beta_{X,Z} \delta
\end{align*}
where $\beta_{X,Z}$ refers to the map defined in \Cref{df:PuncturedTN}, 
viewed in the homotopy category $\ho \T(k)$. 
\end{thm}
\begin{proof}
To compute the map \(\beta_{X,D}\) from \Cref{df:PuncturedTN} in the homotopy category \(\ho \T(k)\), 
we follow a structured approach. First, we apply \Cref{num:Mumford_source} to determine the source. 
Next, we use \Cref{lm:Mumford_target} along with the previous result for the target. 
The formulas for the four maps are derived from the additive structure of \(\ho \T(k)\).
\end{proof}

\begin{df}
Under the assumptions of the preceding theorem, in the specific case (2), we refer to the $(I \times I)$-matrix $\mu$ with coefficients in $\GW(k)$ as the \emph{quadratic Mumford matrix} associated with the (log-)pair $(X,D)$. 
\end{df}

By applying the rank morphism $\GW(k) \rightarrow \mathbb{Z}$ to the coefficients of $\mu$, one obtains the intersection matrix of the divisor $D$ within $X$, as discussed by Mumford in \cite[\textsection 1]{mumfordihes} (see the formula \eqref{prop:quad_Mumford1} below).

\begin{rem}\label{rem:orient}
In the oriented case (1), the theorem applies more generally over any base scheme \( S \) — it is necessary for there to be \( S \)-points \( x_i \) of the \( D_i \). The same comment applies if we assume that \( \mathcal{T} \) is \( \mathrm{SL} \)-oriented, the conditions outlined in case (2) are satisfied, and we require that the normal cones \( \mathcal{C}_{x_i} \) are orientable, as indicated by invertible sheaves on the base \( S \). We leave the details to the interested reader.
%
\end{rem}

Note that, according to the additivity of $\ho \T(S)$, the map $\mu:\bigoplus_{i \in I} \un_S(1)[2] \rightarrow \bigoplus_{j \in I} \un_S(1)[2]$
 in the above theorem is given by a square matrix $(\mu_{ij})_{i,j \in I^2}$ with coefficients in the ring $\End_{\ho \T}(\un_k)$.
 Given the preceding formula, one can give a very concrete formula for its computation.
 
\begin{prop}\label{prop:comput_quad_Mumford_matrix}
Consider the assumptions of the previous theorem.
\begin{enumerate}
\item Let us assume that condition (1) of the previous theorem holds,
 and that $\End_{\T}(k)=\ZZ$.\footnote{Relevant cases are
 $\T=\DM, \DM_\et, \DB, D(-_\et,\ZZ_\ell), \DHdg$, from diagram \eqref{eq:motivic_cat}.}
 Then for every $(i,j) \in I$, 
\begin{equation}\label{prop:quad_Mumford1}
\mu_{ij}=\deg([D_i]\cdot[D_j])=(D_i,D_j)
\end{equation}
is the usual \emph{intersection number} of the (effective Cartier) divisors $D_i$ and $D_j$.
\item Let us consider the case $\T=\SH$ and assume that condition (2) in the previous theorem holds.
 Recall that $\End_{\ho \SH(k)}(\un_k)=\GW(k)$.

For any integer $i \in I$, one considers the orientation class $o_i=\epsilon_i \otimes (\tau_i^\vee)^{-1}$ 
of the conormal sheaf $\calC_i$, obtained via the isomorphism \eqref{eq:hyp_Mumford2}.
Then, for every $(i,j) \in I^2$, one gets the formula 
\begin{equation}\label{prop:quad_Mumford2}
\mu_{ij}=\tdeg_{\tau_i}\big(\nu_i^!([D_j,o_j])\big)
\end{equation}
computed using Chow-Witt groups, 
where $\tdeg_{\tau_i}$ is the quadratic degree of the oriented curve $(D_i,\tau_i)$ over $k$
(see \eqref{eq:epsilon-degree}), $\nu_i^!$ is the pullback map (using deformation to the normal cone
 as in \cite{FaselCHWring,Feld1}) associated with the
 regular closed immersion $\nu_i$, 
and $[D_j,o_j]_X$ is the class of the $o_j$-oriented divisor
 $D_j$ of $X$ (see \Cref{rem:oriented-class-CHt}).

In particular, if $i=j$, 
\begin{equation}\label{prop:quad_Mumford3}
\mu_{ii}=\tdeg_{\tau_i}e(N_i,o_i)
\end{equation}
where $e(N_i,o_i)\in  \CHW 1(D_i)$ is the Euler class of the oriented vector bundle $(N_i,o_i)$,
 $N_i=\VV(\calC_i)$
 (see Remark \ref{rem:Mumford-euler}(2)).
\end{enumerate}
\end{prop}
\begin{proof}
According to the formula for $\mu$ in the above theorem, for every $(i,j) \in I^2$,
 one can compute the coefficient $\mu_{ij}$ as the following composite map
\begin{align*}
\un_S(1)[2] \xrightarrow{p_i^!} \htp_S(D_i,-T_i)(1)[2]
 \xrightarrow{\tau_{i*}} \htp_S(D_i) & \xrightarrow{(\bar \nu_i)_*} \htp_S(X) \\
 &\xrightarrow{(\bar \nu_j)^!} \htp_S(D_j,N_j)
 \xrightarrow{o_{i*}^{-1}} \htp_S(D_j)(1)[2] 
 \xrightarrow{(p_j)_*} \un_S(1)[2]
\end{align*}
where we have used \Cref{num:Mumford_source} except that we have indicated by $p_i^!$
 and $x_i^!$ the twisted Gysin maps for clarity.
 Note that we obtain the Gysin map $(\bar \nu_j)^!$ by unwinding the definition
 of the purity isomorphism \eqref{eq:generalizedhomotopypurity2}.

In the case $\T=\SH$, the preceding composite map lives in $\End_{\ho \SH}(\un_k)$. To compute it,
 we can perform a computation of Chow-Witt groups by applying the functor
 $\Hom(-,\uKMW_*(1)[2])$, where $\uKMW_*$ is the unramified Milnor-Witt K-theory sheaves
 seen as a motivic spectrum over $k$. This yields formula \eqref{prop:quad_Mumford2},
 given that the covariant (resp. contravariant) functoriality of $\htp_S(X)$
 corresponds to a pullback (resp. pushforward) in Chow-Witt groups.
 The formula \eqref{prop:quad_Mumford2} follows by the (oriented version of the) 
 self-intersection formula \cite[3.2.9(ii)]{DJK}, and \eqref{prop:quad_Mumford1}
 is obtained by realizing in the appropriate motivic category.
\end{proof}

\begin{rem}\label{rem:Mumford-euler}
The element $\tdeg_{\tau_i}e(N_i,o_i) \in \GW(k)$ coincides with the Euler number $n^{\mathrm{GS}}(N_i,\sigma_0,\rho_i)$ of the zero section $\sigma_0$ of $N_i$ with respect to the
 \emph{relative orientation class} 
 $o_i\otimes (\tau_i^{\vee})^{-1}\in \Or_{D_i}(\calC_i\otimes \omega_i^\vee)$ 
 (see \Cref{ex:relative_Or} for explanations) 
 of $\calC_i$ considered by Bachmann-Wickelgren in \cite{BW_Euler}. 
 One can check that in our setting, this element is actually independent on the chosen orientations, equal to $\frac{1}{2}(D_i, D_i)h$, where $h=\langle 1,-1\rangle\in \GW(k)$ is the class of the hyperbolic plane and where $(D_i, D_i)=\deg(\mathcal{C}_i^\vee)\in 2\mathbb{Z}$ is the usual self-intersection number of 
 $D_i$\footnote{$\mathcal{C}_i$ has even degree on account of being orientable, see \Cref{rem:orient}(1).}.   
In contrast,  
the coefficients $\mu_{ij}$, $i\neq j$ of the matrix $\mu$ do depend by construction on the choice of the 
orientations $\epsilon_i$ and $\tau_i$ made in assumption (M5b) of \Cref{lm:Mumford_target}. 
\end{rem}

We finally give an explicit formula for the coefficients of the quadratic Mumford
 matrix based on the previous computation and computations of Chow-Witt groups.
\begin{prop} \label{prop:explicit-symmetric}
Consider assumption (2) in the previous proposition.
 Let us fix two indices $(i,j)\in I^2$ such that $i \neq j$.

For any point $x \in D_{ij}$, we let $\kappa(x)$ be the associated residue field,
 $\omega_x=\omega_{\kappa(x)/k}$ be the associated canonical sheaf,
 and $m_x(D_i,D_j)=\lg(\mathcal O_{D_{ij},x})$ be the intersection multiplicity at the point $x$
 of the divisors $D_i$ and $D_j$ of $X$.

Given such a point $x$, as $D_i$ and $D_j$ intersects transversally at $x$,
 one also gets a canonical isomorphism
 $\omega_x \simeq \calC_i|_x \otimes \calC_j|_x \otimes \omega_X|_x^\vee$.
 In particular, the product of orientation classes $o_i|_x \otimes o_j|_x \otimes \epsilon^\vee|_x$
 gives an orientation $o_x(D_i,D_j)$ of $\omega_x$,
 that we can view as a rank $1$ element of $\GW(\kappa(x),\omega_x)$
 (see \Cref{rem:orient&GW}).

Then we have 
$$
\mu_{ij}=\sum_{x \in D_{ij}} m_x(D_i,D_j)_\epsilon.\Tr^\omega_{\kappa(x)/k*}\big(o_x(D_i,D_j)\big)
$$
where $n_\epsilon=\sum_{i=0}^{n-1} \langle (-1)^n \rangle \in \GW(k)$,
 $\Tr_{\kappa(x)/k}^\omega:\omega_x \rightarrow k$ is the differential trace form
 of the finite extension $\kappa(x)/k$,
 and $\Tr^\omega_{\kappa(x)/k*}$ is the associated ``Scharlau transfer''
 (see \Cref{num:or_deg}).

In particular, the quadratic Mumford matrix $\mu$ is symmetric.
\end{prop}
\begin{proof}
According to the \Cref{prop:comput_quad_Mumford_matrix}(2),
 we need to determine the right hand-side of equality \eqref{prop:quad_Mumford2}.
 As the intersection of $D_i$ and $D_j$ is transversal
 and using the computation of the quadratic degrees from \Cref{num:quad_deg},
 one can work locally around the finite scheme $D_ {ij}$.
 In particular, one can assume that $D_i$ is principal, say
 with defining equation $\pi_i$.
 Then one can compute the pullback map $\nu_i^!$ at the level
 of quadratic cycles (as defined in \Cref{num:quadratic-cycles}) according to the formula
\begin{equation}\label{eq:pullback&KMW}
\nu_i^!=\partial_i \circ [\pi_i] \circ j^*
\end{equation}
Here, $j:(X-D_i) \rightarrow X$ is the obvious open immersion
 and we have used the maps defined in \cite[\textsection 5.8, 5.10]{Feld1}:
 $[\pi_i]$ is multiplication by the unit $\pi_i$ on $(X-D_i)$,
 and $\partial_i$ is the boundary map associated with the divisor $D_i \subset X$.
 Then relation \eqref{eq:pullback&KMW} can be derived using the proof of \cite[12.4]{RostChow},
 which allows for a reduction to the property (R3d) of the Milnor-Witt module $\KMW_*$
--- see also \cite[3.2.15]{DFJ} for a proof in terms of Chow-Witt groups as a Borel-Moore homology.
 Finally, the formula for $\mu_{ij}$ can be deduced from \eqref{eq:pullback&KMW},
 by coming back to the definition of the basic maps $[\pi_i]$ and $\partial_i$,
 applying \cite[Lem. 2.19]{MorLNM} to get the multiplicity $m_x(D_i,D_j)_\epsilon$,
 and finally use the formula \eqref{eq:degree-formula}.
\end{proof}

\begin{ex}
Let us assume that $D$ is a simple normal crossing divisor with only $k$-rational intersections.
 In this case, for each $x \in D_{ij}$, $\omega_x=k$ and the differential trace map is the identity.
 Moreover, the orientation class $o_x(D_i,D_j)$ belongs to $\Or_x(\omega_x)=Q(k)$, so that it is
 the quadratic class of a unit $u_{ij}^x \in k^\times$.
In this case, the formula for the non-diagonal coefficients of the quadratic Mumford matrix reads
$$
\mu_{ij}=\sum_{x \in D_{ij}} \langle u_{ij}^x \rangle \in \GW(k)
$$
Our main computation will show that one can choose orientation classes so that
 all the $u_{ij}^x=1$.
\end{ex}

\begin{rem}
It is possible to define the quadratic Mumford matrix in slightly greater generality.
In fact, according to \cite[Chap. 4, \textsection 1]{BCDFO}, for any Cartier divisor $D$
 in a smooth $k$-scheme $X$, classified by a line bundle $\cO(D)$ over $X$,
 one associates a canonical quadratic cycle class
 $[D] \in \CHt^1(X,\cO(D))$. 
 Now, if $X$ is a surface with canonical sheaf $\omega_X$, and $D$, $D'$ are Cartier divisor,
 we need only to give a relative orientation $o$ of $\cO(D+D')$;
 that is, a quadratic isomorphism $o:\cO(D+D') \qiso \omega_X$ (\Cref{ex:quad_iso}),
 to define the intersection degree as 
$$
(D.D')_o=\tdeg_o([D].[D'])
$$
Here, we use the quadratic $o$-degree \Cref{num:or_deg} and the intersection product 
$$
\CHt^1(X,\cO(D)) \otimes \CHt^1(X,\cO(D)) \rightarrow \CHt^1(X,\cO(D+D'))
$$

In particular, coming back to the situation of a log-pair $(X,D)$
 as in \Cref{num:hyp_Mumford2}, and under the assumption of \Cref{thm:smhatoomatrix-general}(2),
 one needs only to give a relative orientation of $o:\cO(D) \qiso \omega_X$ in order to
 define all the terms of the Mumford matrix by the formula
$$\mu_{ij}=(D_i.D_j)_o$$
 In \Cref{thm:maincomputation}, however, we need more orientations to split
 the stable homotopy types $\htp(D)$ and $\htp(X/X-D)$.
\end{rem}

\subsection{Abelian mixed motives (Nori and Artin-Tate)}
\begin{num}
In the next example, we apply \Cref{thm:smhatoomatrix-general}(1) to the case of Voevodsky and Nori mixed motives.
 We use $\T=\DM$ and consider a log-pair $(X,D)$ over a field $k$ as in \Cref{num:hyp_Mumford2}.
 We further assume that $k$ has characteristic $0$ with a fixed embedding in the field of complex numbers.

Then we can consider the abelian category $\MMN(k,\ZZ)$ of (mixed) integral Nori motives over $k$, 
 as defined in \cite[4.2.4]{RuiTub} (see \cite{HMS, IvMo} for rational coefficients).\footnote{As we work over a field,
 there is no difference between the ordinary and perverse $t$-structures from \emph{loc. cit.}}
 According to \emph{loc. cit.} Remark 3.1.6 and Proposition 5.1.1, 
 there exists a homological functor\footnote{That is:
 sending homotopy exact sequences to (long) exact sequences.}
$$
\HM_0:\DM_{gm}(k) \rightarrow \mathrm{DN}_{gm}(k) \rightarrow \MMN(k,\ZZ)
$$
Given a $k$-scheme $X$, we write $\HM_n(X)=\HM_0(M[-n])$ and 
refer to objects $M$ of $\DM_{gm}(k)$ as (geometric) Voevodsky motives.
We say that $M$ is concentrated in \emph{Nori-degrees} $[a,b]$ if for any $n \notin [a,b]$,
 $\HM_n(M)=0$.\footnote{Beware, however, that there is no underlying $t$-structure on $\DM_{gm}(k)$
 corresponding to this notion of Nori-degree. First of all, one needs to replace $\DM_{gm}(k)$ with its
 \'etale-localization --- or work with rational coefficients --- to hope that such a $t$-structure exists
 (see \cite[Chap. 5, Prop. 4.3.8]{FSV}). Even under these assumptions, the existence of the motivic $t$-structure
 is conjectural. But see the end of this subsection.}
 The category of geometric Nori motives is monoidal rigid. We let $N^\vee$ be the dual of a Nori motive.\footnote{As usual,
 this comes from resolution of singularities, which implies that every geometric Nori motive admits a finite resolution
 by Nori motives of smooth projective $k$-schemes.}

For $n \geq 0$, we define the Nori motive
$$
\HM_n(\TN(X,D)):=\HM_0(\TN(X,D)[-n])
$$
as the $n$-th (motivic) homology of the punctured tubular neighborhood of $(X,D)$.
When $X$ is proper over $k$, 
this is the homology of the boundary motive of $(X-D)$ 
(see \Cref{ex:realization} and \Cref{prop:basic_comput_thp-infty}), 
or the \emph{(motivic) homology at infinity}
$$
\HM_n^\infty(X-D)=\HM_i(\TN(X,D))
$$
According to \Cref{prop:Mumford_source}, 
we are led to consider the geometric Voevodsky motive $M(\mathcal D)$
given by the complex 
$$
\oplus_{i<j} [D_{ij}] \xrightarrow{\sum_{i<j} \nu_{ij*}^i-\nu_{ij*}^j} \oplus [D_i-\{x_i\}]
$$
in homological degrees $[0,1]$.
Here, 
$[Y]$ denotes the object associated to a smooth $k$-scheme $Y$ in the additive category 
$\mathrm{Sm}^{\mathrm{cor}}_k$ (see \cite[Chap. 5]{FSV}). 
Its image in $\DM(k)$ is precisely the object defined in \emph{loc. cit.}
We let $\HM_n(\mathcal D)$ be its $n$-th motivic homology.
\end{num}

\begin{prop}
\label{prop:Mumford_motives}
The Voevodsky motive $M(\mathcal D)$ is concentrated in Nori-degrees $[0,1]$
 and there exists an exact sequence in $\MMN(k,\ZZ)$
$$
0 \rightarrow \bigoplus_{i \in I} \HM_1(D_i) \rightarrow \HM_1(\mathcal D)
 \rightarrow \bigoplus_{i<j} \HM_0(D_{ij}) \xrightarrow{\sum_{i<j} p^i_{ij*}-p^j_{ij*}} \bigoplus_{i \in I} \un_S 
 \rightarrow \HM_0(\mathcal D) \rightarrow 0
$$
Moreover, the homology motive $\TN(X,D)$ is concentrated in Nori-degrees $[0,3]$
 such that
\begin{align*}
\HM_0(\TN(X,D)) & \simeq \HM_0(\mathcal D) \\
\HM_3(\TN(X,D)) & \simeq \HM_0(\mathcal D)^\vee(2)
\end{align*}
  and there is an exact sequence in $\MMN(k,\ZZ)$:
\begin{align*}
0 \rightarrow \HM_1(\mathcal D)^\vee(2) \rightarrow \HM_2(\TN(X,D))
 \rightarrow \bigoplus_{i \in I} \un(1) 
 \xrightarrow{\ \mu\ } \bigoplus_{j \in I} \un(1)
 \rightarrow \HM_1(\TN(X,D)) \rightarrow \HM_1(\mathcal D) \rightarrow 0
\end{align*}
where $p_{ij}^i:D_{ij} \rightarrow \Spec(k)$ is the canonical projection
 and $\mu$ is the Mumford intersection matrix (acting on the Nori Tate twist).
\end{prop}
\begin{proof}
The first exact sequence follows from the homology exact sequence associated to 
the cone $M(\mathcal D)$ since 
$M(D_i-\{x_i\})\simeq\un \oplus \HM_1(D_i)[1]$
(which follows from the Chow-K\"unneth decomposition of the smooth proper curve $D_i$).

The other statement follows from the homology long exact sequence deduced 
from the distinguished triangle provided by \Cref{thm:smhatoomatrix-general}.
\end{proof}
\begin{rem}
The Nori motive  $\HM_0(\TN(X,D))=\HM_0(\mathcal D)$ is a pure Artin motive.
 By contrast, $\HM_1(\mathcal D)$ is an extension of a pure $1$-motive
 (the sum of the dual Jacobian of each $D_i$) by a pure Artin motive.
 With rational coefficients, $\HM_0(\TN(X,D))$ and $\HM_3(\TN(X,D))$ are pure of respective weights $0$ and $-4$,
 while $\HM_1(\TN(X,D))$ and $\HM_2(\TN(X,D))$ are in general mixed of weights $\{0,-2\}$ and $\{-2,-4\}$, 
 respectively (see \cite{IvMo} for the notion of weights on Artin-Tate-Nori motives).
\end{rem}

\begin{rem}\label{rem:TN-higher-htp}
The computations in this example shows that $M(D)$ is in Nori-degree $[0,2]$
 while $M(X/X-D)$ is in Nori-degree $[2,4]$. We can take a closer look at the
 model of the motivic punctured tubular neighborhood from \Cref{prop:mainHZcomputation}
 and \Cref{rem:mainHZcomputation}. After inverting the characteristic exponent of $k$,
 it is obtained by applying the Suslin singular complex functor $C_*^{Sus}$ to the 
 following complex of sheaves with transfers
$$
\bigoplus_{i<j} \ZZ^{tr}(D_{ij})
 \xrightarrow{\ d_0\ } \bigoplus_{i \in I} \ZZ^{tr}(D_{i}) 
 \xrightarrow{\nu^*\nu_*} \bigoplus_{j \in I} \ZZ^{tr}(X/X-D_j)
 \xrightarrow{\ d^0\ } \bigoplus_{j<k} \ZZ^{tr}(X/X-D_{jk})
$$
where $\ZZ^{tr}(D_i)$ is placed in degree $0$. 
We note that the associated motivic complex 
\( C_*^{Sus} \mathbb{Z}^{tr}(X/X-D_j) \) (respectively, \( C_*^{Sus} \mathbb{Z}^{tr}(X/X-D_{jk}) \)) 
is in Nori degree \([2,4]\) (respectively, \{4\}). 
This observation explains why \(\HM_1\) and \(\HM_2\) of \(\TN(X,D)\) represent an extension of two 
Nori motives: one originating from \( M(D) \) and the other from \( M(X/X-D) \).
\end{rem}

\begin{num}
As another illustration of our main result,
 we consider the case where each branch $D_i$ of the divisor $D$ is rational. 
\Cref{thm:smhatoomatrix-general} shows the motive $M(\TN(X,D))$ over $k$ is Artin-Tate:
 it is in the smallest thick triangulated subcategory $\DMAT(k)$ of $\DM(k)$ which contains motives 
 of the form $M(L)(n)$, where $L/k$ is a finite separable extension of $k$.

If $k$ is of arbitrary characteristic, 
we will assume it has \emph{Kronecker index} at most one;\footnote{Recall the Kronecker index of a field $F$, 
of transcendence degree $d$ over its prime subfield and characteristic $p$, 
is either $d+1$ if $p=0$ or $d$ if $p>0$.}
for example, a number field, 
a finite field or a finitely generated field of transcendence degree $1$ over a finite field.
Let $\DMAT(k,\QQ)$ be the triangulated category of (constructible) Artin-Tate motives over $\QQ$.
From \cite{LevineAT}, it follows that $\DMAT(k,\QQ)$ admits a motivic t-structure (uniquely characterized),
whose heart is the Tannakian category $\MMAT(k,\QQ)$ of abelian Artin-Tate motives.
In particular, we obtain a homological and monoidal functor
$$
\HMat_0:\DMAT(K,\QQ) \rightarrow \MMAT(K,\QQ)
$$
\end{num}
\begin{prop}
\label{cor:Mumford_ATmotives}
Under the above assumptions, 
the Artin-Tate homology motive $\HM_i(X)$ vanishes for $i\not\in [0,3]$ and there is an exact sequence 
in the abelian category $\MMAT(S,\QQ)$ 
\begin{align*}
0 \rightarrow &\HMat_3(\TN(X,D)) \rightarrow \bigoplus_{i \in I} \un(2)
\xrightarrow{\sum_{i<j} p_{ij}^{i!}-p_{ij}^{j!}} \bigoplus_{i<j} M(D_{ij})(2)  \\
& \rightarrow \HMat_2(\TN(X,D)) \rightarrow \bigoplus_{i \in I} \un(1) 
\xrightarrow{\ \mu\ } \bigoplus_{j \in I} \un(1) \\
& \rightarrow \HMat_1(\TN(X,D)) \rightarrow \bigoplus_{i<j} \HMat_0(D_{ij})
\xrightarrow{\sum_{i<j} p^i_{ij*}-p^j_{ij*}} \bigoplus_{i \in I} \un
\rightarrow \HMat_0(\TN(X,D)) \rightarrow 0
\end{align*}
Here, we use a similar notation to that in the above proposition,
and $p_{ij}^{i!}$ is the Gysin map associated with the finite morphism $p_{ij}^i$.
\end{prop}

\begin{rem}
One obtains similar exact sequences of Artin-Tate mixed motives over more general bases $S$ using:
\begin{enumerate}
\item \cite{ScholbachAT}: when $S \subset  \Spec{\mathcal O_K}$, $\mathcal O_K$ a number ring;
\item \cite{IvMo}: $S$ a smooth $K$-scheme, for a field $K$ with a complex embedding $K \subset \CC$.
\end{enumerate}
Indeed, 
the indicated references provide us with a suitable category of Artin-Tate(-Nori) motives,
and one can make precisely the same calculation 
(considering the dimension of $S$ as we use perverse motivic t-structures). 
\end{rem}

\begin{ex} 
\label{ex:ramsurface}
To illustrate \Cref{thm:smhatoomatrix-general}, \Cref{prop:comput_quad_Mumford_matrix}, 
we compute Wildeshaus' \emph{boundary motive}, 
or equivalently the motive at infinity (\Cref{ex:realization}), 
of \emph{Ramanujam's surface} $\Sigma$ over a field $k$ of characteristic different from $2$.
We work in $\T=\DM$, the integral category of motives.

First, we recall the construction of $\Sigma$.
Given a cuspidal cubic $C\subset \mathbf{P}^2_k$ and a smooth $k$-rational conic $Q\subset \mathbf{P}^2_k$ 
intersecting $C$ with multiplicity $5$ in a $k$-rational point $p$, 
let $\Sigma$ be the complement of the proper transforms of $C$ and $Q$ in the blow-up 
$\sigma:\mathbf{F}_1\to \mathbf{P}^2_k$ of the remaining $k$-rational intersection point $q$ of $C$ and $Q$
(see \cite{zbMATH03089705} for Hirzebruch surfaces $\mathbf{F}_n$, $n\geq 0$). 
Over the complex numbers, 
the underlying analytic space of $\Sigma$ is a topologically contractible open smooth manifold non-homeomorphic 
to ${\bf R}^4$ whose topological fundamental group at infinity $\pi^\infty_1(\Sigma)$ is infinite 
with trivial abelianization, 
see \cite{Ramanujam}. 

A compactification $X=\bar \Sigma$ of $\Sigma$ with smooth reduced crossings boundary $D=\partial \Sigma$ (see Definition \ref{df:normal_crossing}) is obtained from $\mathbf{F}_1$ 
by blowing up the singular point of $C$, 
with exceptional divisor $E\simeq \mathbf{P}^1_k$. The irreducible components of $D$ are then $E$ and the proper transforms of $Q$ and $C$, 
with respective self-intersections $E^2=-1$, $Q^2=4$ and $C^2=3$. 
Furthermore, $Q$ and $C$ intersect with multiplicity $5$ at the unique point $p$, and $E$ and $Q$ intersect with multiplicity $2$ at a unique $k$-rational point. 
A minimal log-resolution $Y\to $X of the pair $(X,D)$ is then obtained by performing further sequences of blow-ups with centers over the intersections points of the proper transform of $C$ with those of $E$ and $Q$ in such a way that the total transform $B$ of $D$ in $Y$ has the following weighted dual graph $\Gamma$: 
\[
\xymatrix{ 
& (-2) \ar@{-}[d] & & (Q,-2) \ar@{-}[d]
\\
(E,-3) \ar@{-}[r] & (-1) \ar@{-}[r] & (C,-3) \ar@{-}[r] & (-1) \ar@{-}[r] & (-2) \ar@{-}[r] & \ar@{-}[r] (-2) \ar@{-}[r] & (-2) \ar@{-}[r] & (-2) }
\]
Next, we apply \Cref{thm:smhatoomatrix-general} to the pair $(Y,B)$.
 Since $\Gamma$ is a tree, one first obtains that the \emph{Artin part} $\mathcal D=\un_k$, and then that the maps $a$, $b$, $b'$
 are all zero for degree reasons (see also the proof of \Cref{prop:RationalTree}).
Then from \Cref{prop:comput_quad_Mumford_matrix},
 the map $\mu:(\un_k(1)[2])^{\oplus 10} \to (\un_k(1)[2])^{\oplus 10}$ is given by the integer valued intersection matrix $M(B)$ of $B$. Since the successive blow-up made to obtain the pair $(Y,B)$ from the pair $(X,D)$ do not change the absolute value of the determinant of the intersection matrix, $M(B)$ has the same determinant up to sign has the intersection matrix  
\[
N(D)=\left(\begin{array}{ccc}
4 & 5 & 2 \\
5 & 3 & 0 \\
2 & 0 & -1
\end{array}\right)
\]
of $D$. Since $\det N(D)=1$, we conclude that $M(B)\in \mathrm{GL}_{10}(\mathbb{Z})$. 
\Cref{thm:smhatoomatrix-general} then implies the boundary motive of $\Sigma$ is isomorphic to homotopy fiber 
of the trivial map $\un_k\to \un_k(2)[4]$. 
In summary, we obtain
$$
\partial M(\Sigma)=M^\infty(\Sigma)
\simeq
\un_k\oplus \un_k(2)[3]
$$
\end{ex}

\subsection{Punctured tubular neighborhoods of orientable trees of rational curves} 

\begin{num} 
\label{num:rational-tree}
Consider the assumptions of \Cref{thm:smhatoomatrix-general}(2) in the special case when $D$ is 
an orientable tree of smooth $k$-rational curves on a smooth surface $X$ over a field $k$, 
that is
\begin{enumerate}
\item  $D$ is a smooth normal crossing divisor on $X$ with irreducible components $D_i\simeq \mathbb{P}^1_k$, $i\in I$,  such that for every $i\neq j$, $D_{ij}$ is either empty or consists of a single $k$-rational  point.
\item For every $i\in I$, the conormal sheaf  $\mathcal{C}_i$ of $D_i$ is $X$ is orientable, hence isomorphic to $\mathcal{O}_{D_i}(2n_i)$ for some $n_i\in \mathbb{Z}$. 
\item The dual graph $\Gamma$ of $D$ (see \Cref{not:dual-graph}) is a connected tree.
\end{enumerate}
\end{num}

Since $D_i\simeq \mathbb{P}^1_k$, the canonical sheaf $\omega_i=\omega_{D_i}\cong \cO_{\mathbb{P}^1_k}(-2)$ is orientable for every $i\in I$.  The assumption on the orientability of the conormal sheaves $\calC_i$ implies in turn that $\omega_X|_{D_i}\cong \omega_i\otimes \calC_i$ is orientable for every $i\in I$, whence, by \Cref{thm:orientationncdsurfaces-bis}, that $\omega_X|_D$ is orientable.   Thus, assumption (2) of the theorem mentioned above is fulfilled:
  we can choose some orientation class $\epsilon \in \Or_X(\omega_X|_D)$
  and a $\pi_i$-normalized orientation $\tau_i \in \Or_{D_i}(\omega_i)$ for every $i\in I$.
 \medskip
  
Recall that $h=\langle 1\rangle+\langle-1\rangle=1+\langle-1\rangle \in \GW(k)$ 
denotes the class of the hyperbolic plane.
\begin{thm} 
\label{prop:RationalTree}\label{prop:RationalTree-oriented}
In the above setting, 
 there is a \emph{special choice} of the orientation classes $\epsilon$ and $\tau_i$ as above
 such that the punctured tubular neighborhood $\TN_S(X,D)$ in $\ho \SH(k)$ is isomorphic to
$$
\un_k\oplus \mathrm{hofib}(\mathbf{\mu})\oplus \un_k(2)[3]
$$
Moreover, this choice can be made so as to guarantee that
 the quadratic Mumford matrix $\mu:\bigoplus_{i\in I} \un_k(1)[2]\to \bigoplus_{i\in I} \un_k(1)[2]$ 
 is the same as the classical (integer-valued) Mumford matrix $(D_i,D_j)_{i,j}$ 
  except that each diagonal entry  $(D_i,D_i)=-2n_i$ is replaced by $-n_ih$.
 \end{thm}
\begin{proof}
We first consider an arbitrary choice of orientation classes $\epsilon$ and $\tau_i$
 and apply \Cref{thm:smhatoomatrix-general}(2).
Denote by $J\subset I\times I$ the subset consisting of pairs $i<j$
such that $D_{ij}\neq\emptyset$. 
Since $\Gamma$ is a tree,
we have $\sharp J=\sharp I-1$ and 
$\bigoplus_{i<j}\Pi_{k}(D_{ij})=\bigoplus_{(i,j)\in J}\mathbf{1}_{k}$. 
The map $q$ in \Cref{rem:compute-coker-double}
is given by a matrix in $\mathrm{M}_{\sharp J,\sharp I}(\mathbb{Z})$, 
whose Smith normal form is the diagonal matrix 
\[
\left(\begin{array}{c}
\mathrm{id}_{\sharp J}\\
0
\end{array}\right)
\]
The homotopy cofiber $\mathcal{D}$ of $q$ is thus equivalent to that 
of the trivial map $0\to\un_{S}$, 
i.e., to $\un_{S}$.
This implies $\mathcal D^\vee\simeq \un_S$.
By Morel's $\mathbb{A}^1$-connectivity theorem, 
$\mathrm{Hom}_{\mathrm{SH}(k)}(\mathbf{1}_{k},\mathbf{1}_{k}(i)[2i])=0$ for all $i>0$.
Thus, the maps $a$, $b$, and $b'$ appearing in \Cref{thm:smhatoomatrix-general} 
 must vanish, which implies that $\TN_S(X,D)$ is the homotopy fiber of a map of the form
$$
\beta=\begin{pmatrix}
0 & 0 \\
0 & \mu
\end{pmatrix}: \un_S \oplus \bigoplus_{i \in I} \un_S(1)[2] 
\rightarrow 
\mathcal \un_S(2)[4] \oplus \bigoplus_{j \in I} \un_S(1)[2]
$$
Moreover, by \Cref{prop:comput_quad_Mumford_matrix} and \Cref{rem:Mumford-euler}(2), 
the diagonal entries of $\mu$ are equal to the Euler classes 
$e(\mathcal{C}_i^\vee)=e(\mathcal{O}_{\mathbb{P}^1_k}(-2n_i))=-n_i h\in \GW(k)$. 
We finally show that we can find appropriate choices of orientation classes of the invertible sheaves $\omega_i$, $i\in I$, for which the associated matrix \(\mu\) defined above has the desired form. Assume that we have initially given orientations classes $\epsilon\in \Or_D(\omega_X|_D)$ and $\tau_i\in \Or_{D_i}(\omega_i)$ and let \(\mu_{ij} = \langle \alpha_{ij} \rangle \in \GW(k)\) denote the corresponding elements, as defined by formula \eqref{prop:quad_Mumford2} for these choices. Given units \(v_i \in k^*\), we can define new orientations classes \(\tau'_i = \langle v_i \rangle \tau_i\) in $\Or_{D_i}(\omega_i)$ for which the resulting family of multiplicities is then given by
\[
\mu'_{ij} = \langle \alpha_{ij} v_i v_j^{-1} \rangle,
\]
and our goal is thus to finds such units \(v_j\) for which \(\mu'_{ij} = \langle 1 \rangle\) for all \((i,j)\in J\).
Since \(D\) has \(n\) irreducible components and \(\Gamma\) is a connected tree, there are exactly \(n-1\) intersection points between these components. Each intersection contributes two nonzero coefficients \(\mu_{ij}\) and  \(\mu_{ji}\) which, 
by \Cref{prop:explicit-symmetric} are equal, which gives a total of $n-1$ equations to solve. Meanwhile, there are \(n\) degrees of freedom for the \(v_i\).
Our system is therefore underdetermined and, being multiplicatively linear, it admits
 at least one solution. This completes the proof.

\end{proof}

In the following, we illustrate our techniques by explicitly calculating the punctured tubular neighborhoods of Du Val singularities on normal surfaces. We also explore the stable homotopy types at infinity of Danielewski hypersurfaces, which are a family of smooth affine surfaces that hold historical significance in relation to the Zariski cancellation problem.

\subsubsection{Example 1: Stable motivic links of Du Val singularities on normal surfaces}
Let $X_0$ be a geometrically integral normal surface essentially of finite type over a field $k$ 
with an isolated $k$-rational rational double point $x$, also called a Du Val singularity. 
Recall from \cite{zbMATH03229289}, \cite{zbMATH03557923} that among many equivalent characterizations, 
this means that letting $\pi:X\to X_0$ be the minimal desingularization of $X_0$ and 
$\pi_{\bar{k}}:X_{\bar{k}}\to X_{0,\bar{k}}$ be the base change to an algebraic closure $\bar{k}$ of $k$, 
the following holds:

\begin{enumerate}
\item $\pi^{-1}_{\bar{k}}(x_{\bar{k}})$ is a smooth normal crossing divisor whose irreducible components are proper $\bar{k}$-rational curves $E_i$ intersecting each other transversely at $\bar{k}$-rational points only.
\item The curves $E_i$ have self-intersection number $-2$ and the intersection matrix $(E_i , E_j)_{i,j}$ 
is negative definite.
\end{enumerate}
The incidence graph of the divisor $E=\pi^{-1}_{\bar{k}}(x_{\bar{k}})$ is one of the classical Dynkin diagram of type $A_n$, $n\geq 1$,
$D_n$, $n\geq 4$, $E_6$, $E_7$ and $E_8$ depicted in the left column of Table \ref{tab:DuVal-links}.  If $\bar{k}$ has characteristic different from $2$, $3$ and $5$, the completion of the local ring $\mathcal{O}_{X_{0,\bar{k}},x_{\bar{k}}}$ is isomorphic to $\bar{k}[[x,y,z]]/(f)$ where $f$ is one of the polynomials listed in the second column of Table  \ref{tab:DuVal-links}, in particular the analytic local isomorphism type of the singularity depends only on the Dynkin diagram.\footnote{In characteristics $2$, $3$, and $5$, there are finitely many additional  "normal forms"; see \cite{zbMATH03557923} for the complete list.} Over a non-closed field, Du Val singularities $A_n$, $D_n$ and $E_6$ can in general have non-trivial $k$-forms depending on the action of the Galois group $\mathrm{Gal}(\bar{k}/k)$ on the irreducible components of $E$. We now assume, in addition that all the irreducible components of $E$ are defined over the base field $k$ and isomorphic to $\mathbb{P}^1_k$.\footnote{Over a field of characteristic zero, this amounts to restricting to "split" Du Val singularities $A_n^{-}$, $D_n^{-}$, $E_6^{-}$, $E_7$ and $E_8$, see  \cite{zbMATH01315234}.}
For such singularities, the closed pair $(X,E)$ satisfies the assumptions in \Cref{num:rational-tree}, 
and the punctured tubular neighborhood $\TN_S(X_0,x)$ of $x$ in $X_0$ is a natural invariant of the 
Nisnevich germ of $x$ in $X_0$ which, by \Cref{cor:cdh-invariance}, 
can be computed as the punctured tubular neighborhood $\TN_S(X,E)$. 
Applying \Cref{prop:RationalTree}, we obtain the following 

\begin{prop} 
With the assumption above, 
the punctured tubular neighborhood $\TN_k(X_0,x)$ 
is isomorphic to $$ \un_k\oplus \mathrm{hofib}(\mu(\Gamma))\oplus \un_k(2)[3]
$$ 
Here $\mu(\Gamma)$ is the square matrix with entries in $\GW(k)$ obtained from the integer valued intersection matrix $(E_i,E_j)_{i,j}$ associated to the Dynkin diagram $\Gamma=A_n,D_n,E_6,E_7,E_8$ by replacing each diagonal entry $-2$ by $-h$. 
\end{prop}

The above proposition implies that the \emph{stable motivic link} $\TN(\Gamma):=\TN_k(X_0,x)$ of the Du Val singularity 
germ $(X,x_0)$ depends only on the Dynkin diagram $\Gamma$. 
We summarize these links in \Cref{tab:DuVal-links}. 

\begin{table}[htb]
    \centering
    \begin{tabular}{|l|l|l|}
    \hline
    Dynkin diagram & Normal form over $k$ &   $\TN(\Gamma)$ \\  
    \hline 
    $A_n^{-} \quad \dynkin A{}$ &  $x^2-y^2-z^{n+1}=0$ &
    $\begin{cases}
     \un_k\oplus \mathrm{hofib}(-mh)\oplus \un_k(2)[3] & n=2m-1  \\
     \un_k\oplus \mathrm{hofib}(\tfrac{n}{2}h+1)\oplus \un_k(2)[3] & n\equiv 0 \; [4] \\
      \un_k\oplus \mathrm{hofib}((\tfrac{n}{2})+1)h-1)\oplus \un_k(2)[3] & n\equiv 2 \; [4] 
    \end{cases}$
     \\
    \hline
 $D_n^{-} \quad \dynkin D{}$ &  $x^2+y^2z-z^{n-1}=0$ &
    $\begin{cases}
    \un_k\oplus \mathrm{hofib}(-h)\oplus \un_k(2)[3] & n=2m \\
    \un_k\oplus \mathrm{hofib}(-2h)\oplus \un_k(2)[3] & n=2m+1
    \end{cases}$  
    \\
    \hline
    $E_6^{-} \quad \dynkin E6$ & $x^2+y^3-z^4=0$ &
     $\un_k\oplus \mathrm{hofib}(2h-1)\oplus \un_k(2)[3]$ 
     \\
    \hline
     $E_7 \quad \dynkin E7$ & $x^2+y^3+yz^3=0$ &
     $\un_k\oplus \mathrm{hofib}(-h)\oplus \un_k(2)[3]$ 
     \\
     \hline
      $E_8 \quad \dynkin E8$ & $x^2+y^3+z^5=0$ &
     $\un_k\oplus \un_k(2)[3]$ 
     \\
     \hline
    \end{tabular}
\vspace{0.1in}
    \caption{Stable motivic links of classical split forms of Du Val Singularities}
    \label{tab:DuVal-links}
\end{table}

\begin{num} Let us explain how to compute with Smith normal forms the part 
 $\mathrm{hofib}(\mu(\Gamma))$ of $\TN(\Gamma)$, the 
 stable homotopy punctured tubular neighborhood
 associated with du Val singularities in \Cref{tab:DuVal-links}.
 A priori, this is non-standard since we are considering
 a matrix $\mu(\Gamma)$ with coefficients in the 
 non-principal (even non-reduced!) ring 
$$
\ZZ_\epsilon:=\GG(\ZZ)=\ZZ[\epsilon]/(\epsilon^2-1)
$$
However, one can consider the two quotient rings $\ZZ_\pm=\ZZ_\epsilon/(\epsilon\pm 1)$,
 both isomorphic to $\ZZ$ and the canonical injective map 
 $\pi:\ZZ_\epsilon \rightarrow \ZZ_+ \times \ZZ_-$ with image given by pairs $(n,m)$
 such that $n \equiv m \bmod 2$ (see \cite[3.1.1, 3.1.2]{CDFH}).
We begin with the matrix $\mu(\Gamma)$ having coefficients in $\ZZ_\epsilon$, 
and compute the Smith normal form $\mu(\Gamma)_\pm=S_pm D_\pm T_\pm$ of the matrix obtained 
by mapping to the principal ring $\ZZ_\pm$ (\emph{i.e.,} setting $\epsilon=\pm1$).
\emph{If} the invertible matrices $(S_+,S_-)$, $(T_+,T_-)$,
as well as the diagonal matrix $(D_+,D_-)$, are in the image
of $\pi$ (coefficients by coefficients), one can define
unique lifts $(S,T,D)$ with coefficients in $\ZZ_\epsilon$,
such that $S$ and $T$ are invertible
and the relation $\mu(\Gamma)=SDT$ holds true.
In this situation, 
we deduce the desired Smith normal form and in $\SH(k)$ we obtain an isomorphism 
$$
\mathrm{hofib}\big( \mu(\Gamma)\big)
\simeq
\mathrm{hofib}\big(D\big)
$$
\end{num}

\begin{rem}
We observe that, with the exception of the $E_8$ case, the stable motivic link $\TN(\Gamma)$ of a Du Val singularity differs from the stable motivic link $\TN(\mathbb{A}^2_{k},\{0\}) \simeq \un_k \oplus \un_k(2)[3]$ of a regular point on a surface. In particular, $\TN(\Gamma)$ serves to distinguish Du Val singularities, excluding $E_8$, from regular points. This stands in contrast to the \'etale local fundamental groups of these singularities. In characteristic $p>0$, these groups do not differentiate a double point of the form $A_{p^e}$ from a regular point (see \cite{zbMATH03557923}).
For the case of $E_8$ over the complex numbers, we can interpret the isomorphism $\TN(E_8) \simeq \TN(\mathbb{A}^2_{k},\{0\})$ as a reminder that the topological link of $E_8$ is the Poincaré homology $3$-sphere $\Sigma(2,3,5)$. This is a compact topological $3$-manifold that shares the same singular homology groups as $S^3$, but its fundamental group is isomorphic to the binary dodecahedral group.
\end{rem}

\subsubsection{Example 2: Danielewski hypersurfaces}
\label{subsection:examples}

For a field $k$ and $n\geq 1$, 
the Danielewski hypersurface $D_{n}$ is the smooth affine surface $D_{n}$ in $\mathbf{A}_{k}^{3}$ cut out by 
the equation $x^{n}z=y(y-1)$. 
Owing to \cite{Danielewskisurfaces},
$D_{n}$ becomes a Zariski locally trivial $\mathbf{G}_{a}$-bundle over the affine line with two origins 
$\breve{\mathbf{A}}_{k}^{1}$ 
(using the factorization of the surjective projection $\pi_{n}=\mathrm{pr}_{x}:D_{n}\to\mathbf{A}_{k}^{1}$).  
Thus $D_{n}$ is $\mathbf{A}^{1}$-equivalent to $\breve{\mathbf{A}}_{k}^{1}$ and $\mathbf{P}_{k}^{1}$. 
The threefolds $D_n\times \mathbb{A}^1_k$ are isomorphic, 
but the surfaces $D_n$ are pairwise non-isomorphic.
Over $\mathbb{C}$, 
Danielewski \cite{Danielewskisurfaces}, Fieseler \cite{fieseler} established this by showing the underlying complex analytic manifolds have non-isomorphic first singular homology groups at infinity. 
Our methods provide a base field independent argument that distinguishes between the $D_n$'s via their stable homotopy types at infinity. 
\vspace{0.1in}

We begin by constructing explicit smooth projective completions $\bar{D}_{n}$ of the surfaces $D_{n}$, 
whose boundaries are strict normal crossing divisors. 
The morphism $\varphi_{n}=\mathrm{pr}_{x,y}:D_{n}\to\mathbf{A}_{k}^{2}$ expresses $D_{n}$ as the 
affine modification of $\mathbf{A}_{k}^{2}$ with center at the closed subscheme $Z_{n}$ with ideal 
$(x^{n},y(y-1))$ and divisor $D_{n}=\mathrm{div}(x^{n})$, 
cf.~\cite{zbMATH07149737}.
Furthermore, 
$\varphi_{n}$ decomposes into a sequence of affine modifications
\begin{equation}
\label{equation:affinemodifications}
\varphi_{n}
=
\varphi_{1}\circ\psi_{2}\cdots\circ\psi_{n} 
\colon 
D_{n}\to D_{n-1}\to\cdots D_{2}\to D_{1}
\to
{\mathbf{A}}_{k}^{2}
\end{equation}
given by $\psi_{\ell}:D_{\ell}\to D_{\ell-1}$; $(x,y,z)\mapsto(x,y,xz)$, 
with center at the closed subscheme $Y_{\ell-1}=(x,z)$ and divisor $H_{\ell}=\mathrm{div}(x)$.
That is, 
$\varphi_{1}:D_{1}\to\mathbf{A}_{k}^{2}$ is the birational morphism obtained by blowing-up the points $(0,0)$, 
$(0,1)$ in $\mathbf{A}_{k}^{2}$ and removing the proper transform of $\{0\}\times\mathbf{A}_{k}^{1}$, 
and $\psi_{\ell}:D_{\ell}\to D_{\ell-1}$ is the birational morphism obtained by blowing-up the points $(0,0,0)$, 
$(0,0,1)$ in $\pi_{\ell}^{-1}(0)$ and removing the proper transform of $\pi_{\ell-1}^{-1}(0)$.

Now consider the open embedding $\mathbf{A}_{k}^{2}\hookrightarrow\mathbf{P}_{k}^{1}\times\mathbf{P}_{k}^{1}$;
$(x,y)\mapsto([x:1],[y:1])$. 
Then $C_{\infty}=\mathbf{P}_{k}^{1}\times[1:0]$ and $F_{\infty}=[1:0]\times\mathbf{P}_{k}^{1}$ are irreducible
components of $\mathbf{P}_{k}^{1}\times\mathbf{P}_{k}^{1}$ and we set $F_{0}=[0:1]\times\mathbf{P}_{k}^{1}$. 
Let $\bar{\varphi}_{1}:\bar{D}_{1}\to\mathbf{P}_{k}^{1}\times\mathbf{P}_{k}^{1}$ be the blow-up of the points 
$([0:1],[0:1])$, $([0:1],[1:1])$ in $F_{0}$, with respective exceptional divisors $E_{1,0}$, $E_{1,1}$.
From now on the proper transform of $F_{0}$ in $\bar{D}_{1}$ is also denoted by $F_{0}$.
With these definitions, there is a commutative diagram 
$$
\xymatrix@=24pt{
D_{1} \ar_{\varphi_{1}}[d]\ar[r] & \bar{D}_{1} \ar^{\bar{\varphi}_{1}}[d] \\
\mathbf{A}^{2}_k \ar[r] & \mathbf{P}^{1}_k\times\mathbf{P}^{1}_k
}
$$
Here, 
$D_{1}\hookrightarrow\bar{D}_{1}$ is the open immersion given by the complement of the support of the 
strict normal crossing divisor $\partial D_{1}=C_{\infty}\cup F_{\infty}\cup F_{0}$.
The closures in $\bar{D}_{1}$ of the two irreducible components $\{x=y=0\}$
and $\{x=y-1=0\}$ of $\pi_{1}^{-1}(0)$ equal the exceptional divisors $E_{1,0}$ and $E_{1,1}$, 
respectively. 
We calculate the self-intersection numbers $C_{\infty}^{2}=F_{\infty}^{2}=0$, $F_{0}^{2}=-2$ in $\bar{D}_1$; 
that is, 
the usual degrees of the respective normal line bundles of these curves in $\bar{D}_1$, 
see e.g., 
\cite[Chapter 5.6]{zbMATH00790015}, \cite[Chapter IV]{zbMATH06176082}.

To construct $\bar{D}_{n}$, 
$n\geq2$, 
we start with $\bar{D}_{1}$ and proceed inductively by performing the same sequence of blow-ups as 
for the affine modifications $\psi_{\ell}:D_{l}\to D_{\ell-1}$ in \eqref{equation:affinemodifications}.
This yields birational morphisms $\bar{\psi}_{\ell}:\bar{D}_{\ell}\to\bar{D}_{\ell-1}$
consisting of the blow-up of one point on $E_{\ell,0}-E_{\ell-1,0}$ and another point on 
$E_{\ell,1}-E_{\ell-1,1}$ with respective exceptional divisors $E_{\ell+1,0}$ and $E_{\ell+1,1}$
(by convention $E_{0,0}=E_{0,1}=F_{0})$. 
Moreover, 
$D_{\ell}$ embeds into $\bar{D}_{\ell}$ as the complement of the support of the strict normal crossing divisor 
$\partial D_{\ell}=C_{\infty}\cup F_{\infty}\cup F_{0}\cup\bigcup_{i=1}^{\ell-1}(E_{i,0}\cup E_{i,1})$ in such 
a way that the closures of the two irreducible components $\{x=y=0\}$ and $\{x=y-1=0\}$ of $\pi_{\ell}^{-1}(0)$ 
coincide with the divisors $E_{\ell+1,0}$ and $E_{\ell+1,1}$, respectively.
By construction, there is a commutative diagram 
$$
\xymatrix@=22pt{
\bar{D}_{\ell} \ar^-{\bar{\psi_{\ell}}}[r] & \bar{D}_{\ell-1} \ar[r] & \cdots \ar[r] & \bar{D}_{2} 
\ar^-{\bar{\psi_{2}}}[r]  & \bar{D}_{1} \ar^-{\bar{\varphi}_{1}}[r] & \mathbf{P}^{1}_k\times\mathbf{P}^{1}_k\\
D_{\ell} \ar^-{\psi_{\ell}}[r] \ar[u] & D_{\ell-1} \ar[r] \ar[u] & \cdots \ar[r] & 
D_{2} \ar^-{\psi_{2}}[r] \ar[u] & D_{1} \ar^-{\varphi_{1}}[r] \ar[u] & \mathbf{A}^{2}_k \ar[u] }
$$
For every $n\geq2$, 
we may visualize the boundary divisor $\partial D_{n}$ as a fork of $2n+1$ copies of $\mathbf{P}^{1}_k$ 
\[
\xymatrix{ 
& & & (E_{1,0},-2) \ar@{-}[r] & \cdots \ar@{-}[r]  & (E_{n-1,0},-2) \\
(F_{\infty},0) \ar@{-}[r] & (C_{\infty},0) \ar@{-}[r] & (F_{0},-2) \ar@{-}[ur] \ar@{-}[dr] \\
& & & (E_{1,1},-2) \ar@{-}[r] & \cdots \ar@{-}[r] & (E_{n-1,1},-2) \\}
\]
intersecting transversally in $k$-rational points, 
with the indicated self-intersection numbers for each irreducible component. 
We may order the irreducible components of $\partial D_n$ by setting
$$
F_\infty < C_\infty < F_0  < E_{1,0} < \ldots < E_{n-1,0} < E_{1,1} < \ldots < E_{n-1,1}
$$

The above constructed boundary divisor $\partial D_n$ satisfies the assumption of \Cref{num:rational-tree}. 
Applying Proposition \ref{prop:RationalTree}, 
we deduce that $\Pi^\infty_k(D_n)$ is isomorphic to 
$$\un_k\oplus \mathrm{hofib}(\mu_n)\oplus \un_k(2)[3]$$ 
where $\mu_n$ is the following matrix (with zero entries mostly left out of the notation) 
\[
\mu_{n}=\left(\begin{array}{ccccccccc}
0 & 1\\
1 & 0 & 1\\
 & 1 & -h & 1 &  0 &  & 1\\
 &  & 1 & -h & 1 \\
 &  &  0 & 1 & \ddots & 1\\
 &  &  &  & 1 & -h & 0\\
 &  & 1 &  &  & 0 & -h & 1\\
 &  &  &  &  &  & 1 & \ddots & 1 \\
 &  &  &  &  &  &  & 1 & -h
\end{array}\right)
\in 
\mathrm{M}_{2n+1,2n+1}(\mathrm{GW}(k))
\]
Elementary row and column operations show that $\mu_{n}$ is equivalent to the diagonal matrix 
$\Delta(1,\ldots,1,nh)$. 
We deduce that $\Pi^\infty_k(-)$ distinguishes between all the Danielewski surfaces.

\begin{prop} 
\label{prop:hmDsurfaces}
Over a field $k$ and $n\geq 1$, 
the stable homotopy type at infinity of the Danielewski surface $D_n$ is given by 
$$
\Pi^\infty_k(D_n)\simeq \un_k\oplus \mathrm{hofib}(nh)\oplus \un_k(2)[3]
$$
\end{prop}

%% file: appendix-sheafversion.tex
\subsection{Oriented vector bundles and quadratic isomorphisms}
\label{subsection:ovbaqi}

\begin{num}
\label{num:quadratic_iso}
The notion of oriented real vector bundles was extended to the algebraic setting by Barges-Morel in \cite{BargeMorel}.
 In what follows, we extend their theory to take into account the functoriality properties of induced trivializations
 of Thom spaces.
\begin{df}
\label{def:quad-preiso} A \emph{quadratic pre-isomorphism} from an invertible sheaf $\cL$ to an invertible sheaf $\cL'$  is an isomorphism 
$\tau:\cL \rightarrow \cL' \otimes \cM^{\otimes 2}$,  where $\mathcal{M}$ is an arbitrary invertible sheaf on $X$.

Two quadratic pre-isomorphisms 
$\tau:\cL \rightarrow \cL' \otimes \cM^{\otimes 2}$ and $\tau':\cL \rightarrow \cL' \otimes \cN^{\otimes 2}$ are called equivalent if there exists an isomorphism $\phi:\cM \rightarrow \cN$ such that the following diagram commutes
$$
\xymatrix@C=60pt@R=0pt{
& \cL' \otimes \cM^{\otimes 2}\ar^{Id \otimes \phi^{\otimes 2}}[dd] \\
\cL\ar^-\tau[ru]\ar_-{\tau'}[rd]\ar[rd] & \\
& \cL' \otimes \cN^{\otimes 2}
}
$$
A \emph{quadratic isomorphism}  $\epsilon:\cL \qiso \cL'$ is the equivalence class of a quadratic pre-isomorphism.
\end{df}

The composition of quadratic pre-isomorphisms $\tau:\cL \rightarrow \cL' \otimes \cM^{\otimes 2}$ and 
$\tau':\cL' \rightarrow \cL'' \otimes \cN^{\otimes 2}$ is defined by the formula
\begin{equation}\label{eq:composition_iso_mod2}
\tau' \circ \tau: \cL \xrightarrow{\tau} \cL' \otimes \cM^{\otimes 2} 
\xrightarrow{\tau' \otimes Id} \cL'' \otimes \cN^{\otimes 2} 
\otimes \cM^{\otimes 2} \simeq \cL'' \otimes (\cN \otimes \cM)^{\otimes 2}
\end{equation}
The composition law is compatible with the equivalence relation on quadratic pre-isomorphism. It admits as the identity of an invertible sheaf $\cL$ 
 the canonical isomorphism $\Id_\cL \otimes m^{-1}:\cL \rightarrow \cL \otimes \mathcal{O}_X^{\otimes 2}$ where $m:\mathcal O_X \otimes \mathcal O_X \rightarrow \mathcal O_X$
 is the multiplication map, and it satisfies the associativity relation.
\end{num}

\begin{ex}\label{ex:quad_iso}
An invertible sheaf $\cL$ is orientable in the sense of Barge-Morel if and only if it is quadratically isomorphic to $\cO_X$,
 and an \emph{orientation} (resp. \emph{class of orientation}) of $\cL$ is a quadratic pre-isomorphism (resp. isomorphism)
 -- we will elaborate on this relation below.
 Moreover, if $X$ is a smooth scheme over a field $k$, with canonical sheaf $\omega_X$ and 
 $L=\mathbb{V}(\cL)$ is a line bundle on $X$, then a  \emph{relative orientation} of $L$ in the sense of Bachmann-Wickelgren \cite{BW_Euler} is the same as a quadratic isomorphism $\cL \qiso \omega_X$.
\end{ex}


\begin{df}\label{df:uPicard_mod2}
The \emph{quadratic Picard groupoid} $\uPicO(X)$ of a scheme $X$ is the category whose objects are invertible sheaves on $X$
and with quadratic isomorphisms as morphisms.
\end{df}

Let $\uPic(X)$ denote Deligne's Picard category  of invertible sheaves on $X$ (see \Cref{sub:conventions} 
for our conventions). There is a functor
$$
\rho_X:\uPic(X) \rightarrow \uPicO(X)
$$
which is the identity on objects and maps an isomorphism $\phi:\cL \rightarrow \cL'$ to the equivalence class of the quadratic pre-isomorphism
$\phi \otimes m^{-1}:\cL=\cL\otimes \cO_X  \rightarrow \cL' \otimes (\mathcal{O}_X)^{\otimes 2}$. 
Moreover, one checks the following properties
\begin{enumerate}
\item The tensor product of invertible sheaves induces a symmetric monoidal structure on $\uPicO(X)$, such that $\rho$ becomes monoidal.
 Therefore $\uPicO(X)$ is a Picard groupoid and $\rho_X$ is a natural transformation of Picard groupoids.
\item Given a morphism of schemes $f:Y \rightarrow X$, the pullback of invertible sheaves induces a functor
 $f^*:\uPicO(X) \rightarrow \uPicO(Y)$ such that $\rho_X$ is natural in $X$.
\end{enumerate}
We henceforth denote by $\Isom$ (resp. $\Isom_Q$) the sets of isomorphisms (resp. quadratic isomorphisms) 
of invertible sheaves. 

\begin{num}\label{num:orientation_classes}\textit{Orientation classes}. 
The notion of quadratic isomorphisms naturally covers Barge-Morel's formalism of orientations.
 Given an invertible sheaf $\cL$ over a scheme $X$,
 we define the set of \emph{orientation classes} of $\cL$ as
$$
\Or_X(\cL)=\Isom_Q(\cL,\cO_X) = \big\{(\epsilon,\cM) \mid 
\epsilon:\cL \xrightarrow{\simeq} \mathcal{O}_X\otimes \cM^{\otimes2} \big\}/\sim
$$
Naturally, we say that $\cL$ is \emph{orientable} if the above set is non-empty.
This assignment is functorial for quadratic isomorphisms. Given a morphism of schemes $f:Y\to X$, we denote by $f^*:\Or_X(\cL)\to \Or_Y(f^*\cL)$ the associated map.
The monoidal structure on $\uPicO(X)$ induces a product
$$
\Or_X(\cL) \otimes \Or_X(\cL') \rightarrow \Or_X(\cL \otimes \cL'), (\epsilon,\epsilon')
 \mapsto \epsilon.\epsilon'=(m^{-1}\otimes \mathrm{id}_{(\cM\otimes \cM')^{\otimes 2}}) \circ (\epsilon \otimes \epsilon')
$$
The composition law 
$$
\Or_X(\cO_X) \otimes \Or_X(\cO_X) \rightarrow \Or_X(\cO_X \otimes \cO_X) \xrightarrow{m^{-1}} \Or_X(\cO_X)
$$
defines an abelian group structure on
$\Or_X(\cO_X)$. Its neutral element is the class of the quadratic pre-isomorphism $m^{-1}:\cO_X\to \cO_X^{\otimes 2}$.
\footnote{One can check that the composition of quadratic isomorphisms also induces this group structure.}
Moreover, the preceding product induces an action of $\Or_X(\mathcal{O}_X)$ on $\Or_X(\cL)$.
 In fact, the set of orientations of $\cO_X$ has an interpretation in terms of torsors with
 coefficients in the sheaf $\mu_{2,X}$ of square roots of $X$,
 which we refer to as the sheaf of \emph{local orientations} of $X$
\begin{equation}\label{eq:orient&mu2}
\Or_X(\cO_X)=H^1_{\textrm{Zar}}(X,\mu_2)
\end{equation}
where the torsors are taken in the Zariski (or the Nisnevich) topology. 
 This immediately yields the following result that is very useful in practice.
\end{num}
\begin{prop}\label{prop:orient&Pic-mod2}
For any scheme $X$, there is a short exact sequence of abelian groups
$$
\xymatrix@C=20pt@R=-2pt{
0 \ar[r] & \GG(X)/\GG(X)^2\ar[r] & \Or_X(\cO_X)\ar[r] & \Pic(X)_2\ar[r] & 0 \\
& u \ar@{|->}[r] & m^{-1} \circ (\times u) && \\
&& (\epsilon,\cM)\ar@{|->}[r] & \cM &
}
$$
where $\Pic(X)_2$ is the $2$-torsion subgroup of $\Pic(X)$.

The action of $\Or_X(\cO_X)$ on $\Or_X(\cL)$ is faithful.
 In fact, $\Or_X(\cL)$ is a formally principal homogeneous $\Or_X(\cO_X)$-set: it is either empty
 or a principal homogeneous $\Or_X(\cO_X)$-set.

Moreover, when $\Pic(X)$ has no $2$-torsion and $\Or_X(\cL) \neq \emptyset$, 
 the abelian group $\Or_X(\cO_X) \simeq \GG(X)/\GG(X)^2$ acts fully faithfully on the set $\Or_X(\cL)$.
 In particular, two classes of orientations of $\cL$ differ by a uniquely defined element of $\GG(X)/\GG(X)^2$
 (modulo this action).
\end{prop}

\begin{rem}\label{rem:compare_or}
We first remark that our point of view differs slightly from other sources as we really focus
 on orientation \emph{classes}. This allows one to get structures on those classes,
 and to formulate the preceding result.

In practice, the preceding theorem means that an invertible sheaf $\cL$ on $X$ is orientable
 if and only if its class in $\Pic(X)$ is $2$-divisible.
 Moreover, if $\Pic(X)$ has no $2$-torsion, then two orientation classes of $\cL$ differs
 by a unique quadratic class $\bar \phi \in \GG(X)/\GG(X)^2$ for some global invertible function $\phi$ on $X$. 

For instance, if $X=\PP^1_k$ is the projective line over a field $k$, an invertible sheaf $\cL$ is orientable
 if and only if it has an even degree; moreover, two orientations of $\cL$ differ by a unique quadratic class in $Q(k)=k^*/(k^*)^2$.
\end{rem}

\begin{rem}\label{rem:orient&GW}
We remark that $\Or_X(\cL)$ can be seen as a subgroup of 
the $\cL$-twisted Grothendieck-Witt group $\GW(X,\cL)$ of $X$, defined as for the usual Grothendieck-Witt group
 except that one considers non-degenerate symmetric bilinear $\cL$-forms
 $\mathcal E \otimes_{\cO_X} \mathcal E \rightarrow \cL$.  
 Here, $\mathcal E$ is a finite rank locally free $\cO_X$-module.
Indeed, there is canonical rank map $\rk_{X,\cL}:\GW(X,\cL) \rightarrow \uZ_X$
 induced by the rank map of $\cO_X$-module, and one obtains
$$
\Or_X(\cL)=\rk_{X,\cL}^{-1}(1)
$$
That is, orientations classes of $\cL$ corresponds to classes
of $\cL$-twisted symmetric bilinear forms on line bundles of $X$.
\end{rem}

\begin{ex}\label{ex:relative_Or}
With reference to \Cref{ex:quad_iso}, 
the previous definitions (\Cref{def:quad-preiso}, \Cref{num:orientation_classes}) readily imply that the set $\Or_X(\cL \otimes \omega_X^\vee)$ 
is in bijection with quadratic isomorphisms $\epsilon:\cL \qiso \omega_X$ and also with relative orientations 
of $L=\mathbb V(\cL)$ in the sense of Bachmann-Wickelgren \cite{BW_Euler}. 
\end{ex}

To be precise, we formulate the following definition, which extends the previous case.
\begin{df}\label{df:orient_virtual}
Let $\cV$ be a virtual locally free sheaf over a scheme $X$.
 We say that $\cV$ is orientable if its determinant $\det(\cV)$ is orientable.
 An orientation (class) of $\cV$ is an orientation (class) of $\det(\cV)$.
 We put $\Or_X(\cV):=\Or_X(\det(\cV))$.
\end{df}

\begin{num}
In general, the Thom space functor  (see \ref{num:Thom_spaces})
$$
\Th_X:\uK(X) \rightarrow \h\SH(X)
$$
does not factor through Deligne's graded determinant functor (see \Cref{sec:notations}). 
The purpose of the next
theorem is to give a criterion for when this can be achieved.

Following \cite[\textsection 7.13]{DFJK}, one introduces a variant of Thom spaces
in the case of an invertible sheaf $\cL$ on $X$, using the formula
$$
\Tw_X(\cL):=\Th(\langle \cL\rangle-\langle \mathcal{O}_X\rangle)=\Th(\cL)(-1)[-2]
$$
As explained in \emph{loc. cit.},
 this kind of \emph{twists} is especially relevant when dealing with the so-called $\SL$-oriented theories
 (see \cite{PW1, Anan}). Nevertheless, we consider the functor
$$
\Tw_X:\uZ_X \times \uPic(X) \rightarrow \SH(X), (r,\cL) \mapsto \Tw(r,\cL):=\Tw_X(\cL)(r)[2r]
$$

The next theorem extends earlier considerations due to Röndigs 
\cite[Lemma 4.2]{RonTheta} and Ananyevskiy \cite[Lemma 4.1]{Anan}.
\end{num}
\begin{thm}\label{thm:SH-orientations}
Let $X$ be a scheme such that the canonical map of groups
$$
K_0(X)  \rightarrow H^0_\nis(X,\ZZ) \times \Pic(X), v \mapsto \big(\rk(v),\det(v)\big)
$$
is an isomorphism.

Then the twist map $\Tw_X$ defined above is monoidal and functorial
 with respect to quadratic isomorphisms. Moreover, it fits into the following commutative diagram, 
 in which the dotted arrow is uniquely defined
$$
\xymatrix@C=40pt@R=12pt{
\uK(X)\ar^{\Th_X}[r]\ar_{(\rk,\det)}[d]\ar@{}|{(1)}[rd] & \SH(X)\ar@{=}[d] \\
\uZ_X \times \uPic(X)\ar|-{\Tw_X}[r]\ar_{\Id \times \rho_X}[d]\ar@{}|{(2)}[rd] & \SH(X)\ar@{=}[d] \\
\uZ_X \times \uPicO(X)\ar@{-->}_-{\Tw^{or}_X}[r] & \SH(X)
}
$$
\end{thm}
In practice, the preceding theorem allows one to associate to any orientation class $\epsilon \in \Or_X(\cV)$,
 a canonical isomorphism
$$
\epsilon_*:\Th(\cV) \rightarrow \un_S(r)[2r]
$$
where $r=\rk \cV$.
\begin{proof}
The assumption implies that the functor $(\rk,\det)$ is an equivalence of categories,
 and therefore of Picard categories.
 Let $\tau:\uZ_X \times \uPic(X) \rightarrow \uK(X)$ be the functor
 which associates to $(r,\cL)$ the virtual locally free sheaf $\langle \cL\rangle+(r-1)\langle \mathcal{O}_X\rangle$ on $X$.
 It is clear that $(\rk,\det) \circ \tau \simeq \Id$. Therefore $\tau$ is the inverse
 of $(\rk,\det)$, and as such, it is an equivalence of Picard categories.
 By definition, $\Tw_X=\Th \circ \tau$. And this implies that $\Tw$ is monoidal,
 as well as commutativity of the square (1).

According to \cite[Lem. 4.1]{Anan}, for any invertible sheaf $\cL$ over $X$,
 there exists an isomorphism $\Th(\cL) \simeq \Th(\cL^\vee)$, 
 which is equivalent to the existence of an isomorphism $\Th(\langle \cL \rangle-\langle \cL^\vee \rangle) \simeq \un_X$,
 functorial in $\cL$ (with respect to isomorphisms of invertible sheaves on $X$).
 As $(\rk,\det)$ is an equivalence of categories, one deduces the existence of an isomorphism
$$
\twist{\cL}-\twist{\cL^\vee} \simeq \twist{\cL^{\otimes 2}}-\twist{\cO_X}
$$ 
which is functorial in $\cL$, as both virtual locally frees sheaves have the same rank and determinant.
 One deduces the existence of an isomorphism $\Tw_X(\cL^{\otimes 2}) \simeq \un_X$
 that is functorial with respect to isomorphisms in $\cL$. This implies the existence of $\Tw^{or}_X$
 such that the square (2) commutes.
 The uniqueness follows as $\Id \times \rho_X$ is full (and the identity on objects).
\end{proof}

\begin{ex}
Let $\E$ be a ring spectrum over a scheme $S$ equiped with an $\SL$-orientation $\thom$
 in the sense of Panin-Walter (see \cite{Anan, DFJK}). Let $X$ be a separated $S$-scheme and $\mathcal{V}$ a virtual locally free sheaf on $X$.
 Let us consider the category of modules $\E-\mathrm{mod}_X$ over the monoid $\E_X$
 in the monoidal category $\ho \SH(X)$. One then considers the canonical functors

\begin{align*}
\Th^\E_X:\uK(X) &\rightarrow \E-\mathrm{mod}_X, \cV \mapsto \E_X\otimes \Th_X(\cV) \\
 \Tw^\E_X:\uZ(X) \times \uPic(X) &\rightarrow \E-\mathrm{mod}_X, (r,\cL) \mapsto \E_X\otimes \Tw_X(\cL) (r)[2r].
\end{align*}

The existence of the Thom isomorphism associated with the $\SL$-orientation of $\E$ enables
 the construction of an essentially commutative diagram analogous to the one above
$$
\xymatrix@C=40pt@R=12pt{
\uK(X)\ar^{\Th^\E_X}[r]\ar_{(\rk,\det)}[d] & \E-\mathrm{mod}_X\ar@{=}[d] \\
\uZ_X \times \uPic(X)\ar|-{\Tw^\E_X}[r]\ar_{\Id \times \rho_X}[d]\ar@{}|{(2)}[rd]\ar@{=>}|{\thom}[ru]
 & \E-\mathrm{mod}_X\ar@{=}[d] \\
\uZ_X \times \uPicO(X)\ar@{-->}_-{\Tw^{or,\E}_X}[r] & \E-\mathrm{mod}_X
}
$$
The upper commutative square witnesses that
 for any virtual locally free sheaf $\cV$ of rank $r$ and determinant $\cL$,
 one gets a canonical "Thom" isomorphism
$$
\thom(v):\E_X \otimes \Th_X(v) \xrightarrow \sim \E_X \otimes \Tw_X(\cL)(r)[2r]
$$
This depends on the chosen orientation $\thom$ of $\E$.

The second square means that for any orientation class $\epsilon \in \Or_X(\cL)$, 
there exists a canonical isomorphism in $\ho \SH(X)$
$$
\epsilon^\thom_*:\E_X \otimes \Tw_X(\cL) \xrightarrow \sim \E_X
$$
This isomorphism a priori depends on the chosen orientation $\thom$
but it exists for arbitrary smooth $S$-scheme $X$.\footnote{One can take $X=S$, which may be singular.}

When $X$ satisfies the assumption of the previous theorem,
 by functoriality of the constructions, one deduces the equality of homotopy classes
$$
\E_X \otimes \epsilon_*=\epsilon^\thom_*
$$
where the left-hand side refers to the isomorphism obtained in the previous theorem.

In particular, the above isomorphisms induces the following more usual Thom isomorphisms in cohomology
\begin{align*}
\thom(v)_*:&\E^{**}(X,v) \xrightarrow \sim \E^{*-2r,*-r}(X,\Tw_X(\det v)), v \in \uK(X) \\
\epsilon^{\thom}_*:&\E^{**}(X,\Tw_X(\cL)) \xrightarrow \sim \E^{**}(X), \epsilon \in \Or_X(\cL)
\end{align*}
explaining that $\SL$-oriented cohomologies are bigraded and depend only upon the twist by a line bundle up to orientation.
 Chow-Witt groups provide the most fundamental example for us (the unramified Milnor-Witt sheaf $\uKMW_*$ represents these groups over fields).
\end{ex}

\subsection{Quadratic 0-cycles and quadratic degrees} 
\label{subsection:q0caqd}

\begin{num}\label{num:quadratic-cycles}
 Next, we recall a few definitions of Chow-Witt groups suitable for our needs.\footnote{We focus on zero cycles and emphasize (quadratic) cycles rather than cycle classes.}
 We fix a base field $k$, not necessarily perfect.
 
Given a finitely generated extension field $K/k$, 
we let $K_*^{MW}(K)$ be the Milnor-Witt ring of $K$ (see \cite[Def. 3.1]{MorLNM},
 or \cite{Deg18}).
 Given an invertible $K$-vector space $\mathcal L$, we define the twisted Milnor-Witt ring of $K$ by the formula 
 in \cite[Rem. 3.21]{MorLNM}
\begin{equation}\label{eq:twists_KMW}
K_*^{MW}(K,\mathcal L):=K_*^{MW}(K) \otimes_{\ZZ[K^\times]} \ZZ[\cL^\times]
\end{equation}
where $\cL^\times=\cL-\{0\}$, 
using the action of $K^\times$ on $K_*^{MW}(K)$ via the canonical map $K^\times \rightarrow \GW(K)=K_0^{MW}(K)$.

%

Let now $X$ be an essentially smooth $k$-scheme of dimension $d$ and $\cL$ an invertible sheaf on $X$. 
One defines the group of \emph{quadratic ($d$-codimensional) cycles} on $X$ twisted by $\cL$ as 
\begin{equation}\label{eq:quad-cycles}
\ZW d(X,\cL):=\bigoplus_{x \in X^{(d)}} \GW(\kappa(x),\omega_{x/X}^\vee \otimes_{\kappa(x)} \cL|_x)
\end{equation}
Here $X^{(d)}$ is the set of closed points $x$ of $X$ and $\omega_{x/X}$ is the determinant of the $\kappa(x)$-vector
 space $\mathcal{C}_{x/X}=\mathfrak{m}_x/\mathfrak{m}_x^2$.
 The \emph{support} of a quadratic cycle $\alpha$ is the set of points $x \in X^{(d)}$ whose coefficient in $\alpha$ is
 non-zero. We will consider it as a finite reduced closed subscheme of $X$.

Owing to \cite[Rem. 5.13]{MorLNM}, \cite{FaselCHW}, or \cite[Def. 7.2]{Feld1}\footnote{In Morel's notation, $\ZW d(X,\cL)$ is the $d$-th term of the Rost-Schmid complex $C^*_{\mathrm{RS}}(X, \KMW_d\{\cL\})$.
 In Feld's notation (which uses a different normalization for twists, see the following remark) it is the end of the complex $C_*(X,\KMW_*,\omega^\vee_{X/k} \otimes \cL)$, where $\omega_{X/k}=\det(\Omega_{X/k})$ is the canonical sheaf of $X/k$. 
 Note that in both references, the definition is only given under the additional assumption
 that $k$ is finitely generated over some perfect base field $k_0$.
 We refer the reader to \cite[1.3.8, 1.4.4]{DFJ} (with homological conventions)
 or \cite{Deg18} for the case of an arbitrary base field $k$.} there is a map $$
\mathrm{div}:\bigoplus_{y \in X^{(d-1)}} \KMW_1(\kappa(y),\omega_{y/X}^\vee \otimes \cL|_y) \longrightarrow \Zt^d(X,\cL)
$$
Two quadratic cycles are said to be \emph{rationally equivalent} if their difference is in the image of $\mathrm{div}$. The above defines an additive equivalence relation $\sim_{rat}$ on quadratic $d$-codimensional cycles, and the \emph{$d$-th Chow-Witt group} of $X$ twisted by $\cL$ is the quotient
$$
\CHt^d(X,\cL)=\Zt^d(X,\cL)/\!\sim_{rat}=\mathrm{coKer}(\mathrm{div})
$$
This group depends functorially on $\cL$ for quadratic isomorphisms.
\end{num}

\begin{rem} Several remarks are in order to explain our choice of conventions.
\begin{enumerate}
\item The advantage of considering quadratic cycles of maximum codimension
 is that one can define them by the simple formula \eqref{eq:quad-cycles}.
 In the case of arbritrary codimension, one has to consider additionally a condition
 of \emph{non-ramification}; in other words, quadratic cycles can be define as
 the kernel of the differential in the Rost-Schmid complex.\footnote{Thus they are
 ``cycles'' in the traditional sense with respect to the Rost-Schmid complex!}
\item Formula \eqref{eq:quad-cycles} follows \emph{cohomological conventions}.
 These conventions do coincide with the original definition of Fasel in \cite{FaselCHW},
 and for example with the one chosen in \cite[Chap. 2, \textsection 3]{BCDFO}.
 But they do not coincide with the convention of Feld in  \cite[\textsection 5.2]{Feld1},
 for which the twists differ. Note that the groups defined by Feld only differs
 up to a canonical isomorphism, obtained by changing the twists.
 In fact, the convention of Feld is rather \emph{homological} (though he uses a graduation
 with respect to codimension).

We formalized the passage from cohomological to homological
 after the next proposition.
 This is also explained in \cite[Chap. 6, Rem. 4.2.14]{BCDFO}. 
\end{enumerate}
\end{rem}

\begin{prop}
\label{prop:cohomology-tchd}
Let $X$ be an essentially smooth $k$-scheme of dimension $d$ and let $v=\mathbb{V}(\cV)$ be a virtual vector bundle 
of rank $d$ on $X$. Then there is a canonical isomorphism
$$
H^0_{\SH}(X,v):=[\un_X,\Th(v)] \simeq \CHt^d\big(X,\det\cV\big)
$$
\end{prop}
\begin{proof}
Because the stable homotopy category satisfies continuity (\cite[Def. 4.3.2]{CD3}),
 and $k$ is a filtered colimit of finitely generated field extension over its prime sub-field $F$,
 one can assume that $k=F$, and therfore $k$ is perfect.
 The coniveau spectral sequence (see \cite{Deg11},
 \textsection 1.1.1 and Def. 1.4)
 associated with the cohomology theory $H^*_{\SH}(X,v)$ takes the form
$$
E_1^{p,q}=\oplus_{x \in X^{(p)}} H^{p+q}_{\SH}(\Th(N_xX_{(x)}),v)
 \Rightarrow H_{\SH}^{p+q}(X,v)
$$
Here $X_{(x)}=\Spec(\mathcal O_{X,x})$ and $N_x X_{(x)}$ is the normal bundle of $x$;
 and we have used Morel-Voevodsky's homotopy purity theorem to
 identify cohomology with support with the cohomology of the relevant Thom space, which applies as $\kappa(x)/k$
 is separable (therefore essentially smooth) as we assumed $k$ is perfect. 
The $E_1$-term is concentrated in the range $p \in [0,d]$ and by the $\AA^1$-connectivity theorem, 
in the range $q \leq 0$.
According to Morel's computation of the $0$-stable stem and Feld's theory \cite{Feld2},
there is an isomorphism between complexes
$$
E_1^{*,0} \simeq C^*(X,\KMW_*,\omega_{X/k}^\vee \otimes \det\cV)
$$
We conclude by looking at the line $p+q=d$.
\end{proof}

\begin{num}\label{num:Chow-hlg}\textit{Quadratic $0$-cycles and homological conventions}.
Let $X/k$ be an essentially smooth scheme of dimension $d$
 with canonical sheaf $\omega_X=\det(\Omega_{X/k})$.

To begin with, note that there is a canonical isomorphism
\begin{equation}\label{eq:duality-CHt}
\ZW d(X,\omega_X)=\oplus_{x \in X^{(d)}} \GW(\kappa(x),\omega_{x/X}^\vee \otimes \omega_{X}|_x)
\simeq \oplus_{x \in X_{(0)}} \GW\big(\kappa(x),\omega_{x/k}\big)=:\ZWhlg 0(X)
\end{equation}
The elements of the latter group deserves the name of \emph{quadratic $0$-cycles},
 and corresponds to homological conventions (after taking rational equivalence
 classes, the group coincides with some Borel-Moore homology; see \cite[Chap. 6, Rem. 4.2.14]{BCDFO}).
 The above isomorphism is a lift of the natural Poincaré duality isomorphism between
 cohomological and homological Chow-Witt groups of the smooth $k$-scheme $X$.
 We have used, for any closed point $x \in X$, the conormal exact sequence 
$$ 0\to \mathcal{C}_{x/X}\to \Omega_{X}|_x\to \Omega_{x/k}\to 0$$ 
gives a canonical isomorphism $\omega_{x/X}^\vee \otimes \omega_{X}|_x \simeq \omega_{x/k}$
 of invertible $\kappa(x)$-vector spaces.
Thus, an $\omega_X$-twisted quadratic $0$-cycle can be identified with a formal sum 
$\alpha=\sum_{i \in I} \langle\sigma_i\rangle.x_i$, 
where $x_i \in X$ is a closed point and $\sigma_i$ is the class of
 a non-degenerate $\omega_{x_i/k}$-symmetric bilinear form over $\kappa(x_i)$.\footnote{That is
 the class in the Grothendieck-Witt group of a non-degenerate
 symmetric bilinear morphism $V_i \otimes V_i \rightarrow \omega_{x_i,/k}$
 where $V_i$ is a finite $\kappa(x_i)$-vector space (see \cite[2.1.14]{Deg18}).}
\end{num}

\begin{num}\label{num:quad_deg}\textit{Quadratic Degree}.
It is natural to consider quadratic $0$-cycles when it comes to
 the question of \emph{quadratic degree.}

Let $X$ be a proper smooth $k$-scheme.
 One defines the \emph{quadratic degree} $\tdeg$ of a quadratic $0$-cycle $\alpha \in \ZWhlg 0(X)$
 as the proper pushforward associated with the projection of $X/k$ (see \cite[\textsection 5.3]{Feld1},
 or \cite[1.3.8]{DFJ} in general).
 It is defined at the level of cycles, and factorizes through rational equivalence,
 as follows. For any point $x_i$ in the support of $\alpha$, one can consider the differential trace
 map of the finite (lci) extension $\kappa(x_i)/k$ deduced from Grothendieck duality
 (see e.g., \cite[Def. 6.2.4]{Deg18})
$$
\Tr^\omega_{\kappa(x_i)/k}:\omega_{x_i/k} \rightarrow k
$$
Then one defines the quadratic degree of $\alpha$ over $k$ as the element 
$$
\tdeg(\alpha)=\sum_{i \in I} \langle\Tr^\omega_{\kappa(x_i)/k} \circ \sigma_i\rangle \in \GW(k)
$$
In the classical terminology of Grothendieck-Witt rings, one can write
$$
\Tr^\omega_{\kappa(x_i)/k*}(\langle \omega_i \rangle):=\langle\Tr^\omega_{\kappa(x_i)/k} \circ \sigma_i\rangle
$$
and call it the \emph{Scharlau transfer} associated with the differential trace $\Tr^\omega_{\kappa(x_i)/k}$
 (though Scharlau transfers are usually considered without twists, \cite{Scharlau}).
 If $\kappa(x_i)/k$ is separable, then one has $\omega_{x_i/k}=\kappa(x_i)$ and the differential trace map
 corresponds to the usual trace map $\Tr_{x_i/k}:\kappa(x_i) \rightarrow k$.

More generally, let $\cL$ be an invertible sheaf over $X$ with a relative orientation
 (see \Cref{ex:relative_Or}) given by a quadratic isomorphism $\epsilon:\cL \qiso \omega_X$.
We define the \emph{quadratic $\epsilon$-degree} as the composite
\begin{equation}\label{eq:epsilon-degree}
\tdeg_\epsilon:\ZW d(X,\cL) \xrightarrow{\epsilon_*} \ZW d(X,\omega_X)\simeq \ZWhlg 0(X)
\xrightarrow{\tdeg} \GW(k)
\end{equation}
When $\epsilon$ is the identity quadratic isomorphism of $\omega_X$,
 we just write $\tdeg$, hiding the duality isomorphism \eqref{eq:duality-CHt}.
\end{num}

\begin{num}\label{num:or_deg}\textit{Oriented degree of oriented $0$-cycles}.
Let $X$ be a $d$-dimensional proper smooth $k$-scheme and assume that 
$\omega_X$ is orientable, with chosen orientation class $\tau \in \Or_X(\omega_X)$. 

Suppose that $Z$ is a reduced, regularly immersed closed subscheme of $X$ of pure codimension $d$,
 such that for each generic point $x \in Z^{(0)}$,
 the corresponding irreducible component $Z(x)$ (with its reduced subscheme structure) is also regularly
 immersed in $X$.\footnote{Examples
 can be smooth subschemes of $X$, or normal crossing divisors in $X$.}
 Let $\omega_{Z/X}=\det\mathcal{C}_{Z/X}$ be the determinant of the conormal sheaf of $Z$ in $X$,
 which is locally free of finite rank by assumption.
 Let finally $\epsilon \in \Or_Z(\omega_{Z/X})$ be an orientation class.

This allows us to define a canonical quadratic $d$-codimensional cycle
 $[Z,\epsilon]_X \in \tilde Z^d(X)$ associated with $(Z,\epsilon)$,
 as the image of $\sum_{x \in Z_{(0)}} \langle 1 \rangle . x \in \tilde Z^0(Z)$ under the composite map
$$
\tilde Z^0(Z) \xrightarrow{\epsilon^{-1}_*} \tilde Z^0(Z,\omega_{Z/X}) \xrightarrow{i_*} \tilde Z^d(X)
$$
With more notation, we can give an explicit formula for this quadratic cycle.
 Note that $\epsilon$ is represented by an isomorphism $\omega_{Z/X} \rightarrow \cL \otimes \cL$,
 that we also denote by $\epsilon$.
 By restriction to $x=\Spec{\kappa(x)}$, taking dual and passing to the reciprocal isomorphism, we get another
 non degenerate symmetric bilinear $\omega_{x/X}^\vee$-form on $Z(x)$\footnote{Note that $\epsilon_x^\vee$ can also be seen
 as the orientation of $\omega_{x/X}^\vee=(\omega_{Z/X}|_x)^\vee$
 obtained by restriction of $\epsilon$ to $x$ and passage to the dual.}
$$
\epsilon_x^\vee:\cL_x^\vee \otimes \cL_x^\vee \rightarrow  \omega_{x/X}^\vee
$$
and the formula
$$
[Z,\epsilon]_X=\sum_{x \in X^{(d)} \cap Z} \langle \epsilon_x^\vee \rangle.x 
$$
Considering $\tau$ as a relative orientation of the trivial bundle $\cO_X$,
 one gets from the above the quadratic $\tau$-degree $\tdeg_\tau$.
We can also give an explicit formula for the $\tau$-oriented degree of the 
$\epsilon$-oriented cycle $[Z,\epsilon]_X$.
First, by definition, 
$\tau$ is the quadratic class of (the inverse of) an isomorphism $\cM \otimes \cM \rightarrow \omega_X$.
One deduces an isomorphism
$$
\epsilon_x^\vee \otimes \tau|_x:(\cL_x^\vee \otimes \cM_x) \otimes (\cL_x^\vee \otimes \cM_x) \rightarrow (\omega_{x/X}^\vee \otimes \omega_X|_x) \simeq \omega_{x/k}
$$
which we can view as a symmetric bilinear $(\omega_{x/k})$-form.
 Seen as an isomorphism, its inverse determines a quadratic class which is
 an orientation in $\Or_x(\omega_{x/k})$. By abuse of notation, we write
 $\epsilon_x^\vee \otimes \tau|_x \in \Or_x(\omega_{X/k})$ for this specific orientation.
In $\GW(k)$, we deduce the formula 
\begin{equation}\label{eq:degree-formula}
\tdeg_{\tau}([Z,\epsilon]_X)=\sum_{x \in Z_{(0)}} \langle \Tr^\omega_{\kappa(x)/k} \circ (\epsilon_x^\vee \otimes \tau|_x) \rangle
\end{equation}
For each term, 
we consider the class of $\epsilon_x^\vee \otimes \tau|_x$ in 
$\Or_x(\omega_{x/k}) \subset \GW(\kappa(x),\omega_{x/k})$,
and apply the twisted Scharlau transfer 
$\Tr^{\omega}_{\kappa(x)/k*}:\GW(\kappa(x),\omega_{x/k}) \rightarrow \GW(k)$,
where $\Tr^{\omega}_{\kappa(x)/k}$ is the differential trace map.

Note, finally, that if $\kappa(x)/k$ is separable, then $\omega_{x/k}=\kappa(x)$. 
In particular,
 the class
 $\langle (\epsilon_x^\vee \otimes \tau|_x)^{-1} \rangle \in \Or_x(\kappa(x))=Q(\kappa(x))$
 (see \Cref{prop:orient&Pic-mod2}) is actually the quadratic class of a unit $u_x \in \kappa(x)^\times$
 (uniquely determined up to a square), and $\langle \Tr^\omega_{\kappa(x)/k} \circ (\epsilon_x^\vee \otimes \tau|_x) \rangle$
 is the class of the symmetric bilinear form
$$
\kappa(x) \otimes_k \kappa(x) \rightarrow k, (a,b)  \mapsto \Tr_{\kappa(x)/k}(u_x.ab)
$$
\end{num}

\begin{rem}\label{rem:oriented-class-CHt}

As a final remark, we note that the construction of the quadratic \(d\)-codimensional cycle \([Z,\epsilon]_X\) can be extended to arbitrary codimension (where \(d = \dim(X)\)). We provide a concise definition up to rational equivalence; that is, using Chow-Witt groups.
Consider a closed pair \((X, Z)\) consisting of smooth \(k\)-schemes. Assume that \(Z \subset X\) has pure codimension \(n\), and let \(\epsilon\) be an orientation of the cotangent sheaf \(\mathcal{C}_{Z/X}\), which is equivalent to orienting its determinant \(\omega_{Z/X} = \det(\mathcal{C}_{Z/X})\). The quadratic class \([Z, \epsilon]_X\) is defined as the image of the rational class of the quadratic cycle \(\sum_{x \in Z^{(0)}} \langle 1 \rangle \cdot x\) under the composite map
\[
\CHt^0(Z) \xrightarrow{\epsilon^{-1}_*} \tilde{\CHt}^0(Z,\omega_{Z/X}) \xrightarrow{i_*} \CHt^n(X)
\]
where \(i_*\) is the direct image morphism in Chow-Witt groups (refer to \cite[\textsection 3]{BCDFO} for our conventions).
Indeed, as discussed previously, one can infer from the definition of \(i_*\) that \([Z, \epsilon]_X\) corresponds to the class of the element
\[
\sum_{x \in X^{(n)} \cap Z} \langle \epsilon_x^\vee \rangle \cdot x
\]
This formula serves as a definition for the corresponding quadratic \(n\)-codimensional cycle, making it a canonical representative of the class \([Z, \epsilon]_X\).
\end{rem}